\documentclass[a4paper,11pt,reqno]{article}

\usepackage{amsmath,amssymb,amsthm}
\usepackage[colorlinks,allcolors=blue]{hyperref}
\usepackage{graphicx}
\usepackage{xcolor}
\usepackage{tikz}
\usetikzlibrary{patterns}
\usepackage[margin=1.2in]{geometry}
\usepackage[nottoc]{tocbibind} 

\usepackage{genyoungtabtikz}
\YFrench
\Yboxdim{12pt}

\newtheorem{thm}{Theorem}[section]
\newtheorem{lem}[thm]{Lemma}
\newtheorem{prop}[thm]{Proposition}
\newtheorem{cor}[thm]{Corollary}

\theoremstyle{definition}

\newtheorem{defn}[thm]{Definition}
\newtheorem{rem}[thm]{Remark}
\newtheorem{ex}[thm]{Example}

\DeclareMathOperator{\SYT}{SYT}
\DeclareMathOperator{\SSYT}{SSYT}
\DeclareMathOperator{\ps}{ps}

\DeclareMathOperator{\nw}{nw}

\newcommand{\ds}{\displaystyle}
\newcommand{\bb}[1]{\mathbb{#1}}
\newcommand{\bbZ}{\bb{Z}}
\newcommand{\mc}[1]{\mathcal{#1}}
\newcommand{\mb}[1]{\mathbf{#1}}
\newcommand{\tb}[1]{\textbf{#1}}

\newcommand{\tw}[1]{\widetilde{#1}}
\newcommand{\ov}[1]{\overline{#1}}
\newcommand{\la}{\lambda}
\newcommand{\vp}{\varphi}

\newcommand{\RS}{{\rm RS}}
\newcommand{\qrst}{q{\rm RS}t}

\newcommand{\arm}{{\rm arm}}
\newcommand{\leg}{{\rm leg}}

\newcommand{\condP}[2]{\tw{P}(#1 \,|\, #2)}
\newcommand{\qtP}[2]{P(#1 \,|\, #2)}
\newcommand{\col}{{\rm col}}
\newcommand{\row}{{\rm row}}
\newcommand{\id}{{\rm id}}

\newcommand{\U}{\mathcal{U}}
\newcommand{\D}{\mathcal{D}}
\newcommand{\fP}{\mathcal{P}}
\newcommand{\bP}{\ov{\mathcal{P}}}
\newcommand{\Yred}{\Yfillcolour{red}}
\newcommand{\Ycyan}{\Yfillcolour{cyan}}
\newcommand{\Ywhite}{\Yfillcolour{white}}

\numberwithin{equation}{section}
\title{$\qrst$: A probabilistic Robinson--Schensted correspondence for Macdonald polynomials}
\author{Florian Aigner and Gabriel Frieden}
\date{}

\begin{document}

\maketitle

\begin{abstract}
We present a probabilistic generalization of the Robinson--Schensted correspondence in which a permutation maps to several different pairs of standard Young tableaux with nonzero probability. The probabilities depend on two parameters $q$ and $t$, and the correspondence gives a new proof of the squarefree part of the Cauchy identity for Macdonald polynomials (i.e., the equality of the coefficients of $x_1 \cdots x_n y_1 \cdots y_n$ on either side, which are related to permutations and standard Young tableaux). By specializing $q$ and $t$ in various ways, one recovers the row and column insertion versions of the Robinson--Schensted correspondence, several $q$- and $t$-deformations of row and column insertion which have been introduced in recent years in connection with $q$-Whittaker and Hall--Littlewood processes, and the Plancherel measure on partitions. Our construction is based on Fomin's growth diagrams and the recently introduced notion of a probabilistic bijection between weighted sets.
\end{abstract}

\tableofcontents

\section{Introduction}
\label{sec: Intro}

The Robinson--Schensted (RS) correspondence is a bijection between permutations and pairs of standard Young tableaux of the same shape.\footnote{The Robinson--Schensted correspondence often refers to a more general map defined on words, but in this paper we restrict the input to permutations.} This bijection, along with its generalization due to Knuth (RSK), has significant applications in combinatorics, representation theory, algebraic geometry, and probability. One of the most important features of RSK is that it gives a bijective proof of the Cauchy identity
\begin{equation}
\label{eq_intro_Cauchy}
\prod_{i,j \geq 1} \dfrac{1}{1-x_iy_j} = \sum_\la s_\la(\mathbf{x}) s_\la(\mathbf{y}),
\end{equation}
where the sum is over all partitions, and $s_\la(\mathbf{z})$ denotes a Schur function in the variables $\mathbf{z}=(z_1,z_2,\ldots)$. In particular, the RS case of RSK gives a bijective proof of the identity
\begin{equation}
\label{eq_intro_RS_Cauchy}
n! = \sum_{\la \vdash n} (f_\la)^2,
\end{equation}
where the sum is over all partitions of $n$, and $f_\la$ is the number of standard Young tableaux of shape $\la$; this identity arises from \eqref{eq_intro_Cauchy} by comparing the coefficients of the squarefree monomial $x_1 \cdots x_n y_1 \cdots y_n$ on either side.

In the past decade, several randomized versions of RS and RSK have been introduced \cite{BorodinPetrov16, BufetovMatveev18, BufetovPetrov15, MatveevPetrov17, OConnellPei13, Pei14, Pei17}. In these versions, a permutation (or, for RSK, a nonnegative integer matrix) has nonzero probability of mapping to several different pairs of tableaux. The probabilities depend on a parameter $q$ or $t$ in $[0,1)$, and the algorithms give proofs of generalized Cauchy identities for $q$-Whittaker or Hall--Littlewood symmetric functions. These randomized insertion algorithms have applications to probabilistic models such as the TASEP \cite{BorodinPetrov16, MatveevPetrov17} and the ASEP and stochastic six-vertex model \cite{BufetovMatveev18}, as well as to the asymptotics of infinite matrices over a finite field \cite{BufetovPetrov15}.

In this paper, we define a randomized generalization of Robinson--Schensted which depends on two parameters $q$ and $t$. Our map is designed to give a new proof of the squarefree part of the Cauchy identity for the Macdonald symmetric functions $P_\la(\mb{x};q,t)$. The $P_\la(\mb{x};q,t)$ are ``master'' symmetric functions, in the sense that they specialize to many other important families of symmetric functions, including the Schur, $q$-Whittaker, Hall--Littlewood, and Jack symmetric functions. Similarly, our randomized algorithm, which we call $\qrst$, specializes to many of the known variants of RS, including the row and column insertion versions of ordinary RS, $q$-deformations of row insertion and column insertion \cite{BorodinPetrov16, OConnellPei13, Pei14}, and a $t$-deformation of column insertion \cite{BufetovPetrov15}. Our algorithm also specializes to a $t$-deformation of row insertion that does not seem to have been previously considered (although it is related to the $q$-deformation of column insertion by transposition). Figure \ref{fig_spec_chart} summarizes these specializations.

 \begin{figure}
\begin{center}
\begin{tikzpicture}

\draw (7,7) node{\textcolor{blue}{$\qrst$}};

\footnotesize
\draw (1.5,5.19) node{\begin{tabular}{c} \color{red}{$t$-RS} \\ \color{red}{(row insertion)} \end{tabular}};
\draw (4.5,5) node{\begin{tabular}{c} \color{cyan}{$q$-RS} \\ \color{cyan}{(row insertion)} \\ \cite{BorodinPetrov16, MatveevPetrov17} \end{tabular}};
\draw (9.5,5) node{\begin{tabular}{c} \color{cyan}{$q$-RS} \\ \color{cyan}{(column insertion)} \\ \cite{OConnellPei13, Pei14, MatveevPetrov17} \end{tabular}};
\draw (12.5,5) node{\begin{tabular}{c} \color{red}{$t$-RS} \\ \color{red}{(column insertion)} \\ \cite{BufetovPetrov15, BufetovMatveev18} \end{tabular}};

\draw[->] (7,6.7) --node[right]{$q = t$} (7,4);
\draw[->] (7.2,6.7) --node[right]{$\;\; t \rightarrow \infty, q \rightarrow q^{-1}$} (8.7,5.7);
\draw[->] (7.5,6.7) -- (11,6.7) --node[right]{$\;\; q \rightarrow \infty, t \rightarrow t^{-1}$} (12,5.7);
\draw[->] (6.8,6.7) --node[left]{$t \rightarrow 0 \;\;$} (5.3,5.7);
\draw[->] (6.5,6.7) -- (3,6.7) --node[left]{$q \rightarrow 0 \;\;$} (2,5.7);

\draw[->] (7,3) --node[right]{$q \rightarrow 1$} (7,2.1);
\draw[->] (6.8,3) --node[above]{$q \rightarrow 0 \;\;$} (4,2.1);
\draw[->] (7.2,3) --node[above]{$\quad\;\; q \rightarrow \infty$} (10,2.1);

\draw[->] (1.5,4.3) --node[left]{$t \rightarrow 0$} (3,2.1);
\draw[->] (4.5,4.3) --node[left]{$q \rightarrow 0$} (3.5,2.1);
\draw[->] (9.5,4.3) --node[right]{$q \rightarrow 0$} (10.5,2.1);
\draw[->] (12.5,4.3) --node[right]{$t \rightarrow 0$} (11,2.1);

\draw (3,1.8) node{\textcolor{purple}{RS (row insertion)}};
\draw (11,1.8) node{\textcolor{purple}{RS (column insertion)}};

\draw (7,3.5) node{\textcolor{purple}{\begin{tabular}{c} $q$-Plancherel measure \\ (for $\sigma = \id$) \end{tabular}}};
\draw (7,1.61) node{\textcolor{purple}{\begin{tabular}[b]{c} Plancherel measure \\ (for $\sigma = \id$) \end{tabular}}};
\normalsize

\end{tikzpicture}
\end{center}
\caption{Specializations of the $(q,t)$-Robinson--Schensted correspondence. The color indicates the corresponding specialization of the Macdonald polynomials: $q$-Whittaker polynomials for the $q$-RS insertions; Hall--Littlewood polynomials for the $t$-RS insertions; Schur polynomials for the remaining four specializations.}
\label{fig_spec_chart}
\end{figure}
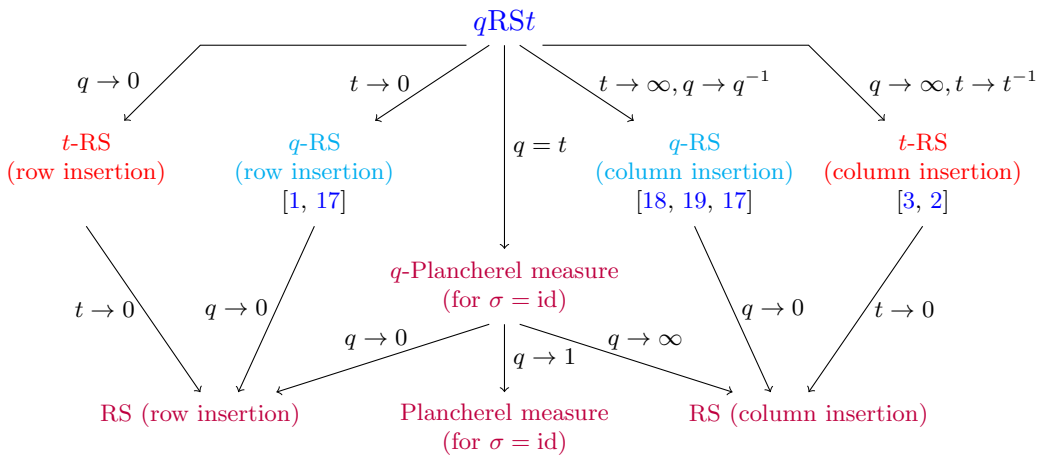

Another interesting specialization of $\qrst$ comes from setting $q = t$. This specialization reduces the Macdonald functions to the Schur functions, but it does not remove the randomness from our algorithm. Instead, it produces a one-parameter family of probabilistic insertion algorithms which interpolate between row insertion ($q=t \rightarrow 0$) and column insertion $(q=t \rightarrow \infty)$. For generic values of $q = t$, the probability that the identity permutation inserts to a pair of standard Young tableaux of shape $\la$ is equal to a $q$-analogue of the Plancherel measure of $\la$. At the intermediate value $q = t \rightarrow 1$, our results specialize to a pair of identities involving hook-lengths and the numbers $f_\la$ (Corollary \ref{cor_r>0_q=t=1}), which we believe are new.

\subsection{Methods}

The squarefree part of the Cauchy identity for Macdonald functions says that
\begin{equation}
\label{eq_intro_qt_Cauchy_sqfree}
\dfrac{(1-t)^n}{(1-q)^n} n! = \sum_{\la \vdash n} \sum_{P,Q} \psi_P(q,t) \vp_Q(q,t),
\end{equation}
where $\psi_P$ and $\vp_Q$ are rational functions in $q$ and $t$ coming from the monomial expansion of Macdonald polynomials, and in the inner sum, $P$ and $Q$ range over standard Young tableaux of shape $\la$. Although both sides of \eqref{eq_intro_qt_Cauchy_sqfree} are sums over sets of size $n!$, there is no bijection between these sets that proves the identity. Our motivation for introducing $\qrst$ was to prove \eqref{eq_intro_qt_Cauchy_sqfree} ``as bijectively as possible.'' Specifically, $\qrst$ assigns to each permutation $\sigma \in S_n$ a probability distribution\footnote{The expressions $\mc{P}(\sigma \rightarrow P,Q)$ are rational functions in $q,t$ which satisfy $\sum_{P,Q} \mc{P}(\sigma \rightarrow P,Q) = 1$, rather than honest probabilities. We find it convenient to refer to them as probabilities anyway, with the justification that whenever $q,t \in [0,1)$ or $q,t \in (1,\infty)$, these rational functions take on values in $[0,1]$.} $\mc{P}(\sigma \rightarrow P,Q)$ on pairs $P,Q$ of standard Young tableaux of shape $\la \vdash n$, such that for each $P$ and $Q$,
\begin{equation}
\label{eq_intro_P_Q}
\dfrac{(1-t)^n}{(1-q)^n} \sum_{\sigma \in S_n} \mc{P}(\sigma \rightarrow P,Q) = \psi_P(q,t) \vp_Q(q,t).
\end{equation}
It is clear that \eqref{eq_intro_qt_Cauchy_sqfree} follows from the existence of probability distributions satisfying \eqref{eq_intro_P_Q}. We say that these probability distributions give a \emph{probabilistic bijection} between the weighted sets of permutations and pairs of standard Young tableaux, where the weight functions come from \eqref{eq_intro_qt_Cauchy_sqfree}. Figure \ref{fig_n=2} shows the probabilistic bijection $\qrst$ and the explicit form of \eqref{eq_intro_qt_Cauchy_sqfree} in the case $n=2$; note that the weights on the left- and right-hand sides are different, so it is not possible to prove \eqref{eq_intro_qt_Cauchy_sqfree} by an ordinary bijection.

\begin{figure}
\begin{center}
\begin{tikzpicture}

\draw (12,3) node{$\textcolor{blue}{\dfrac{(1-t)^3(1-q^2)}{(1-q)^3(1-qt)}}$};
\draw (12,0) node{$\textcolor{blue}{\dfrac{(1-t)(1-t^2)}{(1-q)(1-qt)}}$};
\draw (0,3) node{$\textcolor{red}{\dfrac{(1-t)^2}{(1-q)^2}}$};
\draw (0,0) node{$\textcolor{red}{\dfrac{(1-t)^2}{(1-q)^2}}$};

\draw[dashed,->] (2.8,3) -- (7.6,3);
\draw[dashed,->] (2.8,2.8) -- (7.6,0.2);
\draw[dashed,->] (2.8,0.2) -- (7.6,2.8);
\draw[dashed,->] (2.8,0) -- (7.6,0);

\draw (3.7,3.5) node{\footnotesize $\dfrac{1-t}{1-qt}$};
\draw (3,2) node{\footnotesize $\dfrac{t(1-q)}{1-qt}$};
\draw (3,1) node{\footnotesize $\dfrac{q(1-t)}{1-qt}$} (7.6,2.7);
\draw (3.7,-0.5) node{\footnotesize $\dfrac{1-q}{1-qt}$} (7.6,0);

\draw (2,3) node{$12$};
\draw (2,0) node{$21$};
\draw (9,3) node{$\young(12), \young(12)$};
\draw (9,0) node{$\young(1,2), \young(1,2)$};

\draw[dashed,->] (4,4.5) --node[above]{$\mc{P}(\sigma \rightarrow P,Q)$} (6,4.5);
\draw (0,4.5) node{$\omega(\sigma)$};
\draw (2,4.5) node{$\sigma$};
\draw (9,4.5) node{$P,Q$};
\draw (12,4.5) node{$\ov{\omega}(P,Q) = \psi_P\vp_Q$};

\draw (5.5,-2) node{$\implies \quad \textcolor{red}{2 \dfrac{(1-t)^2}{(1-q)^2}} = \textcolor{blue}{\dfrac{(1-t)^3(1-q^2)}{(1-q)^3(1-qt)} + \dfrac{(1-t)(1-t^2)}{(1-q)(1-qt)}}$};

\end{tikzpicture}
\end{center}

\caption{The probabilistic bijection $\qrst$ for $n=2$, which proves the identity shown at the bottom of the figure.}
\label{fig_n=2}
\end{figure}
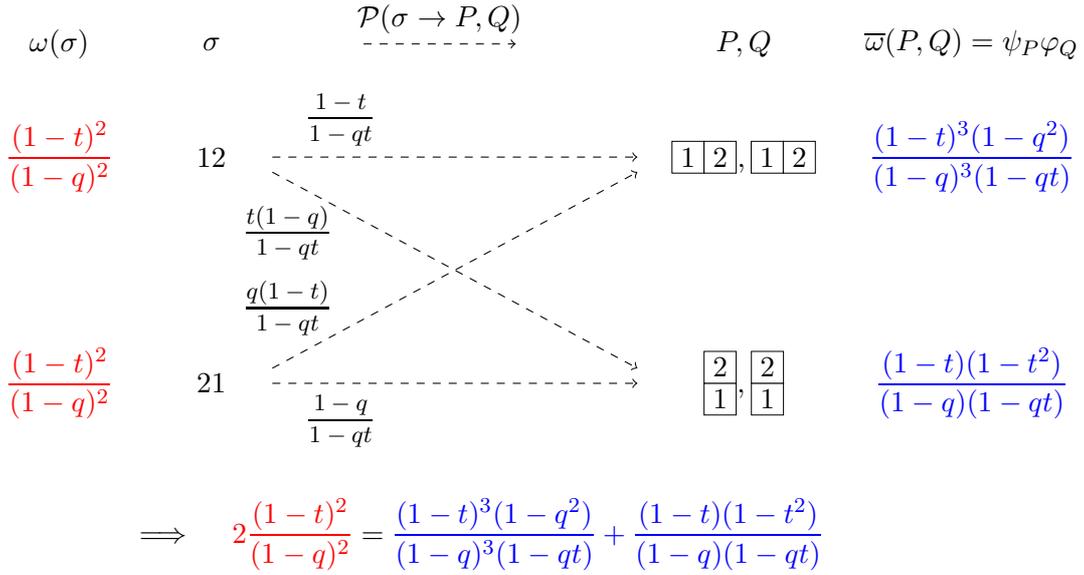

The probabilities $\mc{P}(\sigma \rightarrow P,Q)$ are built recursively out of probabilistic rules for inserting a new number into a tableau $T$, and for re-inserting a number in $T$ that is displaced, or ``bumped,'' by the insertion of a smaller number. Borrowing the approach used in \cite{Pei14, BufetovMatveev18}, we define our probabilistic rules using the framework of Fomin's growth diagrams \cite{Fomin86,Fomin_DGG_2}. This allows us to reduce our problem to finding a family of probabilistic bijections that prove the commutation relation
\begin{equation}
\label{eq_intro_UD}
U_{q,t}D_{q,t} - D_{q,t}U_{q,t} = \dfrac{1-t}{1-q} I,
\end{equation}
where $U_{q,t}$ and $D_{q,t}$ are $(q,t)$-weighted versions of the up and down operators on Young's lattice, and $I$ is the identity.

\subsubsection{More on probabilistic bijections}

We take a moment here to further discuss the notion of probabilistic bijection between two weighted sets, which we believe deserves more attention in the combinatorics community. Suppose $X$ and $Y$ are sets equipped with nonzero weight functions $\omega$ and $\ov{\omega}$, and that
\begin{equation}
\label{eq_intro_prob_bij}
\sum_{x \in X} \omega(x) = \sum_{y \in Y} \ov{\omega}(y).
\end{equation}
As we saw above, one way to prove this identity is to define, for each $x \in X$, a probability distribution $\mc{P}(x \rightarrow y)$ on $Y$ such that
\begin{equation}
\label{eq_intro_x_sums}
\sum_{x \in X} \omega(x) \mc{P}(x \rightarrow y) = \ov{\omega}(y)
\end{equation}
for each $y \in Y$. Given probability distributions $\mc{P}(x \rightarrow y)$ satisfying \eqref{eq_intro_x_sums}, we may define, for each $y \in Y$, a ``backward'' or ``inverse'' probability distribution $\ov{\mc{P}}(x \leftarrow y)$ on $X$ by the equation
\begin{equation}
\label{eq_intro_prob_bij_compatibility}
\omega(x) \mc{P}(x \rightarrow y) = \ov{\mc{P}}(x \leftarrow y) \ov{\omega}(y).
\end{equation}

A simple but powerful observation, which was used in \cite{BufetovMatveev18} and formalized in \cite{BufetovPetrov19}, is that this line of reasoning is reversible: if one can find probability distributions $\mc{P}(x \rightarrow y)$ and $\ov{\mc{P}}(x \leftarrow y)$ such that \eqref{eq_intro_prob_bij_compatibility} holds, then equations \eqref{eq_intro_x_sums} and \eqref{eq_intro_prob_bij} follow. In practice, it may be much easier to verify that a collection of expressions $\mc{P}(x \rightarrow y)$ and $\ov{\mc{P}}(x \leftarrow y)$ define probability distributions (that is, $\sum_{y \in Y} \mc{P}(x \rightarrow y) = \sum_{x \in X} \ov{\mc{P}}(x \leftarrow y) = 1$) and satisfy \eqref{eq_intro_prob_bij_compatibility} than to prove \eqref{eq_intro_x_sums} directly. This is analogous to the fact that in many cases, the easiest (and most useful) way to establish that a map is bijective is not to show that it is injective and surjective, but to explicitly exhibit its inverse. Furthermore, just as a bijection allows for the transportation of combinatorial information between two sets, an explicit description of the ``forward'' and ``backward'' probabilities $\mc{P}(x \rightarrow y)$ and $\ov{\mc{P}}(x \leftarrow y)$ allows for the transportation of probabilistic information (such as a probability distribution) from one set to the other.\footnote{One can imagine the possibility of discovering a bijection by finding a parameter-dependent probabilistic bijection which becomes deterministic for a certain value of the parameter. This is similar in spirit to Kashiwara's combinatorial crystal operators, which are defined as the $q \rightarrow 0$ limit of certain linear maps which depend on $q$ \cite{Kashiwara}.}

Our proof of the up-down commutation relation \eqref{eq_intro_UD} essentially boils down to finding, for each partition $\la$, a probabilistic bijection between the set $\mc{U}(\la)$ of partitions which cover $\la$ in Young's lattice, and the set $\mc{D}^*(\la) = \{\la\} \cup \mc{D}(\la)$, where $\mc{D}(\la)$ is the set of partitions covered by $\la$. We use the ``forward and backward'' approach; that is, we define expressions $\mc{P}_\la(\mu \rightarrow \nu)$ and $\ov{\mc{P}}_{\la}(\mu \leftarrow \nu)$ for $\mu \in \mc{D}^*(\la)$ and $\nu \in \mc{U}(\la)$, in such a way that the compatibility condition \eqref{eq_intro_prob_bij_compatibility} is immediate. The main difficulty is then to prove that the expressions actually sum to 1, which we do by using Lagrange interpolation. In a special case, we are able to interpret the expressions as the probabilities arising from a $(q,t)$-analogue of the Greene--Nijenhuis--Wilf random hook walk \cite{GNW1,GNW2}, thereby giving a more conceptual explanation for the fact that they sum to 1. Our hook walk is modeled on the $(q,t)$-hook walk introduced by Garsia and Haiman \cite{GarHai}.

\subsection{Future directions}
\label{sec_future}

The first problem suggested by our work is to extend $\qrst$ to a probabilistic bijection between nonnegative integer matrices and pairs of semistandard Young tableaux, with weights coming from the Cauchy identity for Macdonald functions. Three of the one-parameter specializations of $\qrst$ (the $q$-deformations of row and column insertion and the $t$-deformation of column insertion) have been extended to $q$-RSK and $t$-RSK algorithms\footnote{To be more precise, the deformations of column insertion extend to deformations of a variant of RSK known as the Burge correspondence; see Remark \ref{rem_RSK_etc}.} \cite{MatveevPetrov17, BufetovMatveev18}, and we hope that these three algorithms will turn out to be specializations of a single $(q,t)$-RSK algorithm. A related problem is to extend $\qrst$ to a ``dual $(q,t)$-RSK,'' which would give a probabilistic bijection for the dual Cauchy identity for Macdonald functions. Two $q$-deformations of dual RSK were constructed in \cite{MatveevPetrov17}; perhaps these can be simultaneously generalized.

In a different direction, one could try to find a probabilistic bijection that explains why the monomial expansion of the Macdonald functions is symmetric in the $\mb{x}$-variables. The symmetry of the Schur functions can be proved bijectively using the Bender--Knuth involutions, or using the Lascoux--Sch\"utzenberger symmetric group action on semistandard Young tableaux. Perhaps one or both of these bijections can be probabilistically generalized to the $(q,t)$-setting. We remark that a probabilistic approach to the symmetry of Hall--Littlewood functions based on the Yang--Baxter equation has been developed in \cite{BufetovPetrov19}; the Schur degeneration of this approach has been used to ``reverse'' the time evolution in the TASEP \cite{PetrovSaenz}.

\subsection{Outline of paper}

In \S \ref{sec: Preliminaries}, we review several approaches to proving the squarefree part of the Cauchy identity \eqref{eq_intro_RS_Cauchy}. We start with the Robinson--Schensted correspondence, proceed to an algebraic proof using up and down operators on Young's lattice, and then explain how Fomin's growth diagrams can be used to ``bijectivize'' the up-down proof, leading to a family of insertion algorithms that generalize Robinson--Schensted. We do not assume that the reader is an expert in any of these topics.

In \S \ref{sec: Macdonald polynomials}, we review basic facts about Macdonald polynomials, and introduce $(q,t)$-up and down operators. In \S \ref{sec: probabilistic bijections}, we give the definition of a probabilistic bijection, and in \S \ref{sec_weighted_sets}-\ref{sec_proofs}, we construct a probabilistic bijection which proves the commutation relation for the $(q,t)$-up and down operators. We give three different formulations of this probabilistic bijection (Definitions \ref{defn_r=0} and \ref{defn_r>0}, Proposition \ref{prop_explicit_formulas}, and Proposition \ref{prop: 3rd version of probabilities}), each of which provides a distinct perspective. In \S \ref{sec_qrst}, we use the probabilistic bijection to define a probabilistic version of growth diagrams, which corresponds to a probabilistic insertion procedure; this is the $(q,t)$-Robinson--Schensted correspondence.

In \S \ref{sec_specs}, we consider the various specializations of $\qrst$ appearing in Figure \ref{fig_spec_chart}: \S \ref{sec_column_insertion} discusses the transformation $(q,t) \mapsto (q^{-1},t^{-1})$ and its relation to column insertion, \S \ref{sec_qWhit_HL} discusses the four $q$-Whittaker and Hall--Littlewood specializations, and \S \ref{sec_q=t_spec} discusses the $q = t$ specialization. In \S \ref{sec_hook_walk}, we review a version of the Greene--Nijenhuis--Wilf hook walk, and show that the $\qrst$ insertion probabilities (but not the re-insertion/bumping probabilities) arise from a $(q,t)$-generalization of this hook walk.

\subsection*{Acknowledgments}

This project grew out of a working group at LaCIM, the combinatorics group at the Universit\'e du Qu\'ebec \`a Montr\`eal, during 2019-2020. We are grateful to all the members of the working group, and especially to Hugh Thomas, Fran\c cois Bergeron, and Steven Karp, for many interesting discussions.

FA acknowledges support from the Austrian Science Fund FWF:  Erwin Schr\"odinger Fellowship J 4387. GF was supported by a CRM-ISM postdoctoral fellowship.

\section{Robinson--Schensted, up and down operators, and growth diagrams}
\label{sec: Preliminaries}

\subsection{The Robinson--Schensted correspondence}
\label{sec: RSK}

A \emph{partition} is a weakly decreasing sequence of nonnegative integers $\la = (\la_1, \ldots, \la_k)$. We say that $\la$ is a partition of the number $|\la| = \lambda_1+\ldots + \lambda_k$, and we write $\la \vdash n$ to indicate that $\la$ is a partition of $n$. We identify the partition $\la$ with its \emph{Young diagram}, which, following the French convention, consists of $\la_1$ boxes in the bottom row, $\la_2$ boxes in the second row from the bottom, etc., with all rows left-justified. We index the boxes, or \emph{cells}, using strictly positive Cartesian coordinates: $c=(x,y) \in (\mathbb{Z}_{>0})^2$ is a cell of the Young diagram of $\lambda$ if $x \leq \lambda_y$. We follow the standard practice of identifying a partition with its Young diagram, and of identifying two partitions which differ only by a sequence of trailing zeroes. We write $\la'$ for the \emph{conjugate} of $\la$, that is, the partition obtained by reflecting the Young diagram of $\la$ in the diagonal $x = y$.

For a cell $c = (x,y) \in \lambda$ we define its \emph{arm-length} $a_\lambda(c)$ and its \emph{leg-length} $\ell_\lambda(c)$ by
\[
a_\lambda(c)= \lambda_y-x, \qquad \qquad
\ell_\lambda(c)=\lambda_x^\prime-y.
\]
The \emph{hook-length} of $c$ is defined by $h_\lambda(c)=a_\lambda(c)+\ell_\lambda(c)+1$. For example, the Young diagram of the partition $\la = (7,6,3,2,1,1)$ is shown below. The cell $c = (2,1)$ has arm-length $a_\la(c) = 5$, leg-length $\ell_\la(c) = 3$, and hook-length $h_\la(c) = 9$.
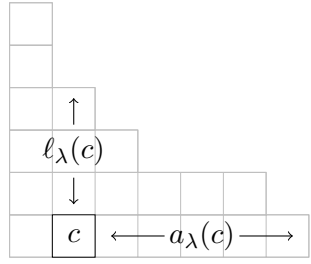
\begin{center}
\begin{tikzpicture}
\Yboxdim{16 pt}
\Ylinecolour{lightgray}
\tyng(0,0,7,6,3,2,1,1)
\Ylinecolour{black}
\tgyoung(0pt,0pt,:;)
\node at (24pt,8pt) {$c$};
\node at (72pt,8pt) {$a_\lambda(c)$};
\draw[->] (86pt,8pt) -- (106pt,8pt);
\draw[->] (58pt,8pt) -- (38pt,8pt);
\node at (24pt,40pt) {$\ell_\lambda(c)$};
\draw[<-] (24pt,20pt) -- (24pt,30pt);
\draw[->] (24pt,50pt) -- (24pt,60pt);
\end{tikzpicture}
\end{center}
We will make frequent use of the following two quantities associated to $\la$:
\[
n(\la) = \sum_{c \in \la} \ell_\la(c), \quad\quad n'(\la) = \sum_{c \in \la} a_\la(c).
\]
It is easy to see that $n'(\la) = n(\la')$.\\

Let $\lambda$ be a partition. A \emph{semistandard Young tableau} of shape $\lambda$ is a filling of the cells of $\lambda$ with positive integers such that all rows are weakly increasing from left to right, and all columns are strictly increasing from bottom to top. A semistandard Young tableau is called a \emph{standard Young tableau} if its set of entries is precisely $\{1,\ldots,|\lambda|\}$. We denote by $\SSYT(\lambda)$ (resp., $\SYT(\lambda)$) the set of all semistandard (resp., standard) Young tableaux of shape $\lambda$. The \emph{content} of an SSYT $T$ is the sequence $(\mu_1,\mu_2,\ldots)$, where $\mu_i$ is the number of entries in $T$ equal to $i$.

\begin{ex}
The following is a semistandard Young tableau of shape $(5,4,2)$ and content $(2,2,3,3,1)$.
\begin{center}
\young(11233,2344,45)
\end{center}
\end{ex}

Let $\mb{x} = (x_1, x_2, \ldots)$ be an infinite sequence of indeterminates. For an SSYT $T$, we set $\mb{x}^T=x_1^{\mu_1} x_2^{\mu_2} \cdots$, where $(\mu_1,\mu_2,\ldots)$ is the content of $T$. The \emph{Schur function} $s_\lambda(\mb{x})$ associated to the partition $\lambda$ is defined by
\[
s_\lambda(\mathbf{x}) = \sum_{T \in \SSYT(\lambda)} \mathbf{x}^T.
\]
The following identity is a fundamental result in the theory of symmetric functions.

\begin{thm}[Cauchy identity]
For two sequences of indeterminates $\mb{x} = (x_1, x_2, \ldots)$ and $\mb{y} = (y_1, y_2, \ldots)$, we have
\[
\prod_{i,j \geq 1}\frac{1}{1-x_iy_j} = \sum_\lambda s_\lambda(\mathbf{x})s_\lambda(\mathbf{y}),
\]
where the sum is over all partitions $\lambda$.
\end{thm}

The Cauchy identity can be rewritten in the form
\begin{equation}
\label{eq_Cauchy_bijective}
\sum_{A=(a_{ij})} \prod_{i,j \geq 1}(x_iy_j)^{a_{ij}} = \sum_{P,Q} \mb{x}^P \mb{y}^Q,
\end{equation}
where the sum on the left-hand side runs over all matrices with nonnegative integer entries and finite support, and the sum on the right-hand side runs over all (ordered) pairs of SSYTs of the same shape.

In this paper, we focus on the coefficient of the squarefree monomial $x_1 \cdots x_n y_1 \cdots y_n$ in the Cauchy identity. The coefficient of this monomial on the left-hand side of \eqref{eq_Cauchy_bijective} is the number of $n \times n$ permutation matrices; the coefficient on the right-hand side is the number of pairs of standard Young tableaux of shape $\la$, where $\la$ is a partition of $n$. Thus, if we write $f_\la$ for the number of standard Young tableaux of shape $\la$, then we have the identity
\begin{equation*}
n! = \sum_{\la \vdash n} (f_\la)^2.
\end{equation*}
This identity can be proved combinatorially by the \emph{Robinson--Schensted (RS) correspondence}, which is a bijection between permutations and pairs of standard Young tableaux of the same shape.

We recall Schensted's description of the RS correspondence as a recursive insertion algorithm. The basic building block of the algorithm is the \emph{insertion} of a number $k$ into a row of a semistandard Young tableau, which is defined as follows: if $k$ is greater than or equal to all entries in the row, add $k$ to the end of the row. Otherwise, let $z$ be the left-most entry in the row which is larger than $k$, and replace $z$ with $k$. We say that $z$ is \emph{bumped} out of the row.

Next, define the insertion of $k$ into an SSYT $T$, denoted $T \leftarrow k$, by the following procedure:
\begin{itemize}
\item Insert $k$ into the first (bottom) row of $T$. If no entry is bumped, the process terminates.
\item If $z$ is bumped out of row $i$, insert $z$ into row $i+1$. Continue in this manner until no bumping occurs.
\end{itemize}
Note that the shape of $T \leftarrow k$ differs from the shape of $T$ by the addition of a single cell.

Finally, let $\sigma = \sigma_1 \cdots \sigma_n$ be a permutation written in one-line notation. Let $P$ be the standard Young tableau obtained by successively inserting the numbers $\sigma_1, \ldots, \sigma_n$ into the empty tableau. Let $Q$ be the standard Young tableau that records the growth of $P$ by placing the number $i$ in the cell that was added during the insertion of $\sigma_i$. The tableaux $P$ and $Q$ are called the \emph{insertion tableau} and \emph{recording tableau} of $\sigma$, respectively. One can show that these two tableaux provide enough information to reverse the insertion procedure and recover $\sigma$, so the map $\sigma \mapsto P,Q$ is a bijection. We will refer to this bijection as the row insertion version of Robinson--Schensted.

\begin{ex}
\label{ex: RSK algorithm by insertion}
Let $\sigma = 526134$. The process of constructing $P$ and $Q$ by the successive insertions
\[
\emptyset \leftarrow 5 \leftarrow 2 \leftarrow 6 \leftarrow 1 \leftarrow 3 \leftarrow 4
\]
is shown below.
\begin{center}
\begin{tikzpicture}
\node at (0.2,0.2) {$\emptyset$};
\tyoung(1cm,0cm,5)
\tyoung(2cm,0cm,2,5)
\tyoung(3cm,0cm,26,5)
\tyoung(4.4cm,0cm,16,2,5)
\tyoung(5.8cm,0cm,13,26,5)
\tyoung(7.2cm,0cm,134,26,5)
\node at (9.2,0.6) {$= P$};

\node at (0.2,-1.8) {$\emptyset$};
\tyoung(1cm,-2cm,1)
\tyoung(2cm,-2cm,1,2)
\tyoung(3cm,-2cm,13,2)
\tyoung(4.4cm,-2cm,13,2,4)
\tyoung(5.8cm,-2cm,13,25,4)
\tyoung(7.2cm,-2cm,136,25,4)
\node at (9.2,-1.4) {$= Q$};
\end{tikzpicture}
\end{center}
\end{ex}

There is also a column insertion version of Robinson--Schensted, which provides a different bijection between permutations and pairs of SYTs of the same shape. This differs from the row insertion described above by using columns instead of rows. That is, to insert $k$ into $T$, one inserts $k$ into the first (left-most) column. If $k$ is larger than all the entries in this column, it is added to the top of the column. Otherwise, $k$ bumps the smallest entry $z$ which is larger than $k$, and $z$ is inserted into the second column. The process continues until no bumping occurs. These maps are of course related by conjugation: if $\sigma \mapsto P,Q$ under row insertion, then $\sigma \mapsto P^t, Q^t$ under column insertion, where $P^t$ and $Q^t$ are the standard Young tableaux obtained by reflecting $P$ and $Q$ in the diagonal $x = y$.

\begin{rem}
\label{rem_RSK_etc}
The \emph{Robinson--Schensted--Knuth (RSK) correspondence} generalizes the row insertion version of RS to a bijection between nonnegative integer matrices of finite support and pairs of SSYTs of the same shape, thereby providing a combinatorial proof of the Cauchy identity. The column insertion version of RS can be generalized to a different bijective proof of the Cauchy identity, which is known as the \emph{Burge correspondence} \cite{Burge74}; this is non-trivial because, unlike SYTs, SSYTs cannot be conjugated. Column insertion can be generalized in a different way to \emph{dual RSK}, which provides a bijective proof of the \emph{dual Cauchy identity}. For details on these generalizations, we refer the reader to \cite{Fulton97,EC2}.
\end{rem}

\subsection{Up and down operators}
\label{sec: Up and Down}

In this section, we review a different approach to proving the identity
\begin{equation}
\label{eq: n! = sum over SYT}
n! = \sum_{\lambda \vdash n} \left( f_\lambda \right)^2,
\end{equation}
which is based on a pair of linear operators acting on Young's lattice.

For two partitions $\la$ and $\mu$, we write $\mu \subseteq \la$ if the Young diagram of $\mu$ is contained in the Young diagram of $\la$. We write $\la \cap \mu$ (resp., $\la \cup \mu$) for the partition obtained by intersecting (resp., taking the union of) the Young diagrams of $\la$ and $\mu$. \emph{Young's lattice} is the partial order on partitions defined by the inclusion relation $\subseteq$; its meet and join are given by $\cap$ and $\cup$, respectively.

If $\mu \subseteq \la$, we write $\la/\mu$ for the \emph{skew diagram} which consists of the cells in $\la$ but not in $\mu$. We write $\mu \lessdot \la$ if $\la$ covers $\mu$ in Young's lattice---that is, if $\mu \subseteq \la$ and $|\la/\mu| = 1$---and we define
\[
\D(\lambda) = \{\mu \,|\, \mu \lessdot \lambda\}, \qquad \U(\lambda) = \{\nu \,|\, \nu \gtrdot \lambda\}.
\]
An \emph{inner corner} of $\la$ is a cell $c \in \la$ such that $\la/\mu = \{c\}$ for some $\mu \in \D(\la)$. An \emph{outer corner} of $\la$ is a cell $c \not \in \la$ such that $\nu/\la = \{c\}$ for some $\nu \in \U(\la)$. For example, Figure \ref{fig: map between down and up} shows a partition with inner and outer corners colored red and blue, respectively. We will often identify the elements of $\D(\la)$ and $\U(\la)$ with the corresponding inner and outer corners of $\la$.

Let $\bb{Y}$ denote the set of all partitions, and $\bb{Q} \bb{Y}$ the $\bb{Q}$-vector space with basis $\bb{Y}$. The \emph{up operator} $U$ and \emph{down operator} $D$ are linear maps on $\bb{Q} \bb{Y}$ defined by
\[
U \lambda = \sum_{\nu \in \U(\la)} \nu, \qquad D\lambda = \sum_{\mu \in \D(\la)} \mu.
\] 
These two operators satisfy the commutation relation
\begin{equation}
\label{eq: up down commutator}
DU-UD = I,
\end{equation}
where $I$ is the identity map.

Before explaining why this commutation relation holds, we use it to deduce \eqref{eq: n! = sum over SYT}. Let $\left\langle \cdot, \cdot \right\rangle$ be the inner product on $\bb{Q} \bb{Y}$ defined by $\left\langle \la, \mu \right\rangle= \delta_{\lambda,\mu}$ for $\la,\mu \in \bb{Y}$. On the one hand, we have
\[
\left\langle D^nU^n\emptyset,\emptyset \right\rangle = \sum_{\lambda \vdash n}  \left\langle U^n\emptyset,\lambda \right\rangle \left\langle D^n\lambda, \emptyset \right\rangle.
\]
A standard Young tableau $T$ of shape $\la$ can be viewed as a saturated chain
\[
\emptyset = T^{(0)} \lessdot T^{(1)} \lessdot \cdots \lessdot T^{(n-1)} \lessdot T^{(n)} = \la
\]
in Young's lattice from the empty partition to $\la$, where $T^{(i)}$ is the shape of the subtableau of $T$ consisting of entries less than or equal to $i$, and $n = |\la|$. This implies that $\left\langle U^n\emptyset,\lambda \right\rangle = \left\langle D^n\lambda, \emptyset \right\rangle = f_\lambda$, so $\left\langle D^nU^n\emptyset,\emptyset \right\rangle$ is equal to the right-hand side of \eqref{eq: n! = sum over SYT}.

On the other hand, by repeatedly using the commutation relation to move a $D$ past all the $U$'s, we obtain
\[
D^nU^n = D^{n-1}UDU^{n-1} + D^{n-1}U^{n-1} = \cdots = D^{n-1}U^nD + nD^{n-1}U^{n-1}.
\]
Since $D\emptyset = 0$, we have $\langle D^nU^n \emptyset, \emptyset \rangle = n \langle D^{n-1}U^{n-1} \emptyset, \emptyset \rangle$. By induction, this is equal to $n!$, proving \eqref{eq: n! = sum over SYT}.

To prove the commutation relation \eqref{eq: up down commutator}, we reformulate it as
\begin{align*}
\left\langle DU \lambda,\lambda \right\rangle &= \left\langle UD \lambda,\lambda \right\rangle +1 & \text{ for all } \la,\\
\nonumber \left\langle DU \lambda,\rho \right\rangle &= \left\langle UD \lambda,\rho \right\rangle & \text{ for } \la \neq \rho,
\end{align*}
or, equivalently,
\begin{align}
\label{eq_la_la}
|\U(\lambda)| &= |\D(\la)| + 1 & \text{ for all } \la, \\
\label{eq_la_rho}
|\U(\lambda) \cap \U(\rho)| &= |\D(\lambda) \cap \D(\rho)| & \text{ for } \la \neq \rho.
\end{align}
Equation \eqref{eq_la_la} says that each partition $\la$ has exactly one more outer corner than inner corner, which is easy to see. Equation \eqref{eq_la_rho} is true because if $\la \neq \rho$, then either $\D(\lambda) \cap \D(\rho) = \{\lambda \cap \rho\}$ and $\U(\lambda) \cap \U(\rho) = \{\lambda \cup \rho\}$, or both of these intersections are empty.

Although it is not difficult to prove \eqref{eq_la_la} and \eqref{eq_la_rho}, it turns out to be quite fruitful, as we will see in the next section, to make the proofs of these equations explicitly bijective. For $\la \neq \rho$, this is uninteresting, as there is a unique bijection between two sets of size 0 or 1. For the equation $|\U(\la)| = |\D(\la)| + 1$, we set
\[
\D^*(\la) = \D(\la) \cup \{\la\},
\]
and we choose, for each $\la$, a bijection
\[
F_\lambda: \D^*(\lambda) \rightarrow \U(\lambda).
\]
There are of course many possibilities for $F_\la$, but two choices are particularly natural: the \emph{row insertion bijection} $F_\la^{\textnormal{row}}$, and the \emph{column insertion bijection} $F_\la^{\textnormal{col}}$. The row insertion bijection sends $\la$ to the outer corner in the first row of $\la$, and the inner corner in row $i$ to the outer corner in row $i+1$. The column insertion bijection sends $\la$ to the outer corner in the first column of $\la$, and the inner corner in column $i$ to the outer corner in column $i+1$. Figure \ref{fig: map between down and up} illustrates these two bijections.

\begin{figure}
\begin{center}
\begin{tikzpicture}
\begin{scope}[scale=1.5]
\Ylinecolour{lightgray}
\tyng(0cm,0cm,7,5,5,2,1)
\Yred
\tgyoung(0cm,0cm,::::::;,,::::;,:;,;)
\Ywhite
\Ycyan
\tgyoung(0cm,0cm,:::::::;,:::::;,,::;,:;,;)
\Ywhite
\draw[line width=2pt] (0,0) -- (7*12pt,0) -- (7*12pt,12pt) -- (5*12pt,12pt) -- (5*12pt,3*12pt) -- (2*12pt,3*12pt) -- (2*12pt,4*12pt) -- (1*12pt,4*12pt) -- (1*12pt,5*12pt) -- (0*12pt,5*12pt) -- (0,0) --(7*12pt,0);
\draw[line width=1pt, ->] (.5*12pt,4.5*12pt) -- (.5*12pt,5.5*12pt);
\draw[line width=1pt, ->] (1.5*12pt,3.5*12pt) -- (1.5*12pt,4.5*12pt);
\draw[line width=1pt, ->] (4.5*12pt,2.5*12pt) -- (2.5*12pt,3.5*12pt);
\draw[line width=1pt, ->] (6.5*12pt,.5*12pt) -- (5.5*12pt,1.5*12pt);
\node at (3.5*12pt,-12pt) {$F_{\lambda}^{\textnormal{row}}$};

\begin{scope}[xshift=5cm]
\Ylinecolour{lightgray}
\tyng(0cm,0cm,7,5,5,2,1)
\Yred
\tgyoung(0cm,0cm,::::::;,,::::;,:;,;)
\Ywhite
\Ycyan
\tgyoung(0cm,0cm,:::::::;,:::::;,,::;,:;,;)
\Ywhite
\draw[line width=2pt] (0,0) -- (7*12pt,0) -- (7*12pt,12pt) -- (5*12pt,12pt) -- (5*12pt,3*12pt) -- (2*12pt,3*12pt) -- (2*12pt,4*12pt) -- (1*12pt,4*12pt) -- (1*12pt,5*12pt) -- (0*12pt,5*12pt) -- (0,0) --(7*12pt,0);
\draw[line width=1pt, ->] (.5*12pt,4.5*12pt) -- (1.5*12pt,4.5*12pt);
\draw[line width=1pt, ->] (1.5*12pt,3.5*12pt) -- (2.5*12pt,3.5*12pt);
\draw[line width=1pt, ->] (4.5*12pt,2.5*12pt) -- (5.5*12pt,1.5*12pt);
\draw[line width=1pt, ->] (6.5*12pt,.5*12pt) -- (7.5*12pt,.5*12pt);
\node at (3.5*12pt,-12pt) {$F_{\lambda}^{\textnormal{col}}$};
\end{scope}
\end{scope}
\end{tikzpicture}
\end{center}
\caption{\label{fig: map between down and up} The Young diagram of the partition $\lambda=(7,5,5,2,1)$, with inner corners colored red and outer corners colored blue. The arrows in the left diagram depict the bijection $F_\lambda^{\textnormal{row}}$, and the arrows in the right diagram depict $F_\lambda^{\textnormal{col}}$. In both cases, the outer corner with no arrow pointing to it is the image of $\la$.}
\end{figure}
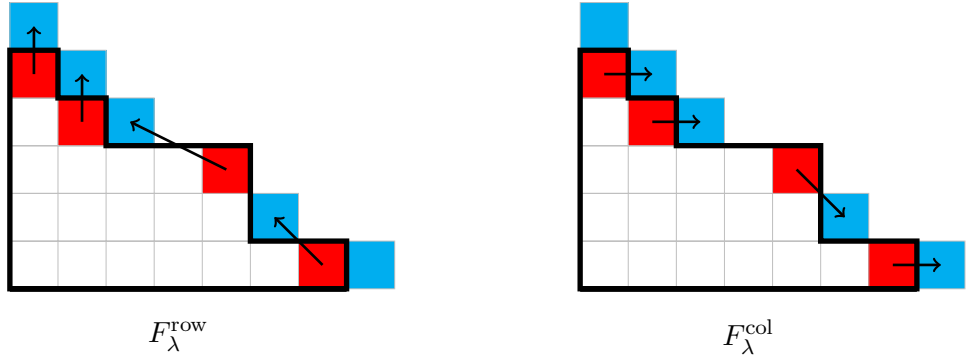

\begin{rem}
The commutation relation $DU - UD = I$ and its application to the enumeration of pairs of standard Young tableaux is the starting point of Stanley's theory of differential posets \cite{Stanley88}. The full Cauchy identity can be proved using a commutation relation for more general up and down operators that add and remove horizontal strips of cells, rather than just single cells \cite{Fomin_Knuth, Gessel93}.
\end{rem}

\subsection{Growth diagrams}
\label{sec: Fomin growth diagrams}

In this section, we review Fomin's growth diagrams \cite{Fomin86, Fomin_DGG_2}, which provide a mechanism for turning a bijective proof of the commutation relation $DU - UD = I$ into a bijective proof of the squarefree part of the Cauchy identity. We explain how each bijective proof arising in this way can be interpreted as an insertion algorithm, with the row and column versions of Robinson--Schensted as special cases. We only consider growth diagrams for permutations; for the generalization to nonnegative integer matrices and the RSK and Burge correspondences, we refer the reader to \cite[\S\S 2.2, 3.1-3.2]{vanLeeuwen05}.

Let $\sigma \in S_n$ be a permutation. We associate to $\sigma$ the permutation matrix $A_\sigma$ which has a 1 in positions $(\sigma(j),j)$, and zeroes elsewhere. We will view $A_\sigma$ as an $n \times n$ grid of squares, and consider labelings of the $(n+1)^2$ vertices of this grid with partitions. For this purpose, we index the vertices by $(i,j)$, with $0 \leq i, j \leq n$, so that the square $(i,j)$ in $A_\sigma$ is surrounded by vertices $(i-1,j-1), (i-1,j), (i,j-1), (i,j)$. (These indices are interpreted as matrix coordinates, rather than Cartesian coordinates.)

\begin{defn}
Fix $\sigma \in S_n$. A \emph{growth associated with $\sigma$} is a labeling $\Lambda = (\Lambda_{ij})$ of the vertices of the grid $A_\sigma$ with partitions, such that
\begin{itemize}
\item $\Lambda_{ij} \subseteq \Lambda_{i,j+1}$ if $j < n$, and $\Lambda_{i,j} \subseteq \Lambda_{i+1,j}$ if $i < n$.
\item $|\Lambda_{ij}|$ is equal to the number of $1$'s in $A_\sigma$ to the northwest of the vertex $(i,j)$.
\end{itemize}
We will often omit reference to $\sigma$ and refer to $\Lambda$ simply as a \emph{growth}.
\end{defn}

This definition is a special case of the notion of ``two-dimensional growth'' introduced by Fomin \cite{Fomin86}.

\begin{ex}
\label{ex_growths_213}
There are four growths associated with the permutation $213 \in S_3$:
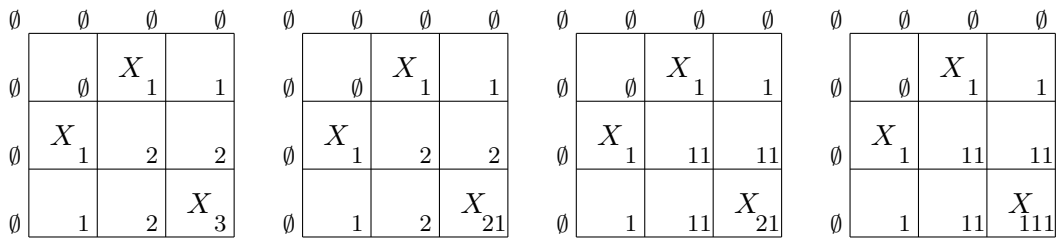
\begin{center}
\begin{tikzpicture}[scale=0.6]

	\node at (1.5*1.5,-.5*1.5) {$X$};
	\node at (.5*1.5,-1.5*1.5) {$X$};
	\node at (2.5*1.5,-2.5*1.5) {$X$};
	
	\footnotesize
	
	\foreach \x  in {0,...,3}{
		\draw (0,-1.5*\x) -- (1.5*3,-1.5*\x);
		\draw (1.5*\x,0) -- (1.5*\x,-1.5*3); 
		\node at (-.3,-1.5*\x+0.3) {$\emptyset$};
	}
	\foreach \x in {1,...,3}
		\node at (-.3+1.5*\x,0.3) {$\emptyset$};
		
	\node at (1*1.5-0.3,-1*1.5+0.3) {$\emptyset$};
	\node at (2*1.5-0.3,-1*1.5+0.3) {$1$};
	\node at (3*1.5-0.3,-1*1.5+0.3) {$1$};
	
	\node at (1*1.5-0.3,-2*1.5+0.3) {$1$};
	\node at (2*1.5-0.3,-2*1.5+0.3) {$2$};	
	\node at (3*1.5-0.3,-2*1.5+0.3) {$2$};	
	
	\node at (1*1.5-0.3,-3*1.5+0.3) {$1$};
	\node at (2*1.5-0.3,-3*1.5+0.3) {$2$};
	\node at (3*1.5-0.3,-3*1.5+0.3) {$3$};

	\normalsize
	
	\node at (6+1.5*1.5,-.5*1.5) {$X$};
	\node at (6+.5*1.5,-1.5*1.5) {$X$};
	\node at (6+2.5*1.5,-2.5*1.5) {$X$};
	
	\footnotesize
	
	\foreach \x  in {0,...,3}{
		\draw (6+0,-1.5*\x) -- (6+1.5*3,-1.5*\x);
		\draw (6+1.5*\x,0) -- (6+1.5*\x,-1.5*3); 
		\node at (6+-.3,-1.5*\x+0.3) {$\emptyset$};
	}
	\foreach \x in {1,...,3}
		\node at (6+-.3+1.5*\x,0.3) {$\emptyset$};
		
	\node at (6+1*1.5-0.3,-1*1.5+0.3) {$\emptyset$};
	\node at (6+2*1.5-0.3,-1*1.5+0.3) {$1$};
	\node at (6+3*1.5-0.3,-1*1.5+0.3) {$1$};
	
	\node at (6+1*1.5-0.3,-2*1.5+0.3) {$1$};
	\node at (6+2*1.5-0.3,-2*1.5+0.3) {$2$};	
	\node at (6+3*1.5-0.3,-2*1.5+0.3) {$2$};	
	
	\node at (6+1*1.5-0.3,-3*1.5+0.3) {$1$};
	\node at (6+2*1.5-0.3,-3*1.5+0.3) {$2$};
	\node at (6+3*1.5-0.3,-3*1.5+0.3) {$21$};

	\normalsize
	
	\node at (12+1.5*1.5,-.5*1.5) {$X$};
	\node at (12+.5*1.5,-1.5*1.5) {$X$};
	\node at (12+2.5*1.5,-2.5*1.5) {$X$};
	
	\footnotesize
	
	\foreach \x  in {0,...,3}{
		\draw (12+0,-1.5*\x) -- (12+1.5*3,-1.5*\x);
		\draw (12+1.5*\x,0) -- (12+1.5*\x,-1.5*3); 
		\node at (12+-.3,-1.5*\x+0.3) {$\emptyset$};
	}
	\foreach \x in {1,...,3}
		\node at (12+-.3+1.5*\x,0.3) {$\emptyset$};
		
	\node at (12+1*1.5-0.3,-1*1.5+0.3) {$\emptyset$};
	\node at (12+2*1.5-0.3,-1*1.5+0.3) {$1$};
	\node at (12+3*1.5-0.3,-1*1.5+0.3) {$1$};
	
	\node at (12+1*1.5-0.3,-2*1.5+0.3) {$1$};
	\node at (12+2*1.5-0.3,-2*1.5+0.3) {$11$};	
	\node at (12+3*1.5-0.3,-2*1.5+0.3) {$11$};	
	
	\node at (12+1*1.5-0.3,-3*1.5+0.3) {$1$};
	\node at (12+2*1.5-0.3,-3*1.5+0.3) {$11$};
	\node at (12+3*1.5-0.3,-3*1.5+0.3) {$21$};	
	
	\normalsize

	\node at (18+1.5*1.5,-.5*1.5) {$X$};
	\node at (18+.5*1.5,-1.5*1.5) {$X$};
	\node at (18+2.5*1.5,-2.5*1.5) {$X$};
	
	\footnotesize
	
	\foreach \x  in {0,...,3}{
		\draw (18+0,-1.5*\x) -- (18+1.5*3,-1.5*\x);
		\draw (18+1.5*\x,0) -- (18+1.5*\x,-1.5*3); 
		\node at (18+-.3,-1.5*\x+0.3) {$\emptyset$};
	}
	\foreach \x in {1,...,3}
		\node at (18+-.3+1.5*\x,0.3) {$\emptyset$};
		
	\node at (18+1*1.5-0.3,-1*1.5+0.3) {$\emptyset$};
	\node at (18+2*1.5-0.3,-1*1.5+0.3) {$1$};
	\node at (18+3*1.5-0.3,-1*1.5+0.3) {$1$};
	
	\node at (18+1*1.5-0.3,-2*1.5+0.3) {$1$};
	\node at (18+2*1.5-0.3,-2*1.5+0.3) {$11$};	
	\node at (18+3*1.5-0.3,-2*1.5+0.3) {$11$};	
	
	\node at (18+1*1.5-0.3,-3*1.5+0.3) {$1$};
	\node at (18+2*1.5-0.3,-3*1.5+0.3) {$11$};
	\node at (18+3*1.5-0.4,-3*1.5+0.3) {$111$};	
\end{tikzpicture}
\end{center}

\noindent Here we represent the partitions $\Lambda_{i,j}$ as the concatenation of their parts; e.g., $21$ represents the partition $(2,1)$. Also, for purposes of readability, we represent the 1's in $A_\sigma$ with $X$, and we omit the 0's.
\end{ex}

The following properties of a growth $\Lambda$ are immediate from the definition:
\begin{enumerate}

\item[(a)] Either $\Lambda_{i,j} = \Lambda_{i,j+1}$ or $\Lambda_{i,j} \lessdot \Lambda_{i,j+1}$, and similarly for $\Lambda_{i,j}$ and $\Lambda_{i+1,j}$.

\item[(b)] The permutation $\sigma$ is determined by the partitions $\Lambda_{i,j}$.

\item[(c)] The southeast corner $\Lambda_{n,n}$ is a partition of $n$, and the right column and bottom row of $\Lambda$ form saturated chains from $\emptyset$ to $\Lambda_{n,n}$.

\end{enumerate}

\noindent In light of (c), we define $P(\Lambda)$ and $Q(\Lambda)$ to be the standard Young tableaux of shape $\Lambda_{n,n}$ corresponding to the right column and bottom row of $\Lambda$, respectively. For example, the third growth in Example \ref{ex_growths_213} has $P(\Lambda) = Q(\Lambda) = \young(13,2)$. If $\Lambda$ is associated with $\sigma$ and $P = P(\Lambda), Q = Q(\Lambda)$, we will write
\[
\Lambda : \sigma \rightarrow P,Q
\]
and say that ``$\Lambda$ is a growth from $\sigma$ to $P,Q$.'' In general, there may be multiple growths from $\sigma$ to $P,Q$.

Another property of growths that follows easily from the definition is that each square in the grid has one of the following four types:
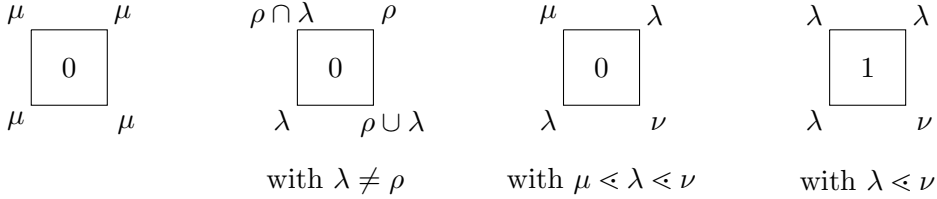
\begin{center}
\begin{tikzpicture}
\draw (0,0) -- (0,-1) -- (1,-1) -- (1,0) -- (0,0);
\node at (-0.2,0.2) {$\mu$};
\node at (1.2,0.2) {$\mu$};
\node at (-0.2,-1.2) {$\mu$};
\node at (1.25,-1.25) {$\mu$};
\node at (.5,-.5) {$0$};

\begin{scope}[xshift=3.5cm]
\draw (0,0) -- (0,-1) -- (1,-1) -- (1,0) -- (0,0);
\node at (-0.2,0.2) {$\rho \cap \la$};
\node at (1.2,0.2) {$\rho$};
\node at (-0.2,-1.2) {$\lambda$};
\node at (1.25,-1.25) {$\rho \cup \la$};
\node at (.5,-.5) {$0$};
\node at (.5,-2) {with $\lambda \neq \rho$};
\end{scope}

\begin{scope}[xshift=7cm]
\draw (0,0) -- (0,-1) -- (1,-1) -- (1,0) -- (0,0);
\node at (-0.2,0.2) {$\mu$};
\node at (1.2,0.2) {$\la$};
\node at (-0.2,-1.2) {$\lambda$};
\node at (1.25,-1.25) {$\nu$};
\node at (.5,-.5) {$0$};
\node at (.5,-2) {with $\mu \lessdot \lambda \lessdot \nu$};
\end{scope}

\begin{scope}[xshift=10.5cm]
\draw (0,0) -- (0,-1) -- (1,-1) -- (1,0) -- (0,0);
\node at (-0.2,0.2) {$\la$};
\node at (1.2,0.2) {$\lambda$};
\node at (-0.2,-1.2) {$\lambda$};
\node at (.5,-.5) {$1$};
\node at (1.25,-1.25) {$\nu$};
\node at (.5,-2) {with $\la \lessdot \nu$};
\end{scope}
\end{tikzpicture}
\end{center}
Note that in the second type of square, we either have $\la \lessdot \rho$, $\rho \lessdot \la$, or $\la \cap \rho \lessdot \la, \rho \lessdot \la \cup \rho$.

\begin{defn}
A set of \emph{local growth rules} $F_\bullet$ is a choice of bijection
\[
F_\lambda: \D^*(\la) \rightarrow \U(\la)
\]
for each partition $\la$. A growth $\Lambda$ is an \emph{$F_\bullet$-growth diagram} if all squares of the third type satisfy $\nu = F_\la(\mu)$, and all squares of the fourth type satisfy $\nu = F_\la(\la)$.
\end{defn}

The key property of local growth rules is that they make the process of constructing growths deterministic. More precisely, if $\Lambda$ is an $F_\bullet$-growth diagram and
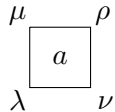
\begin{center}
\begin{tikzpicture}[scale=0.8]
\draw (0,0) -- (0,-1) -- (1,-1) -- (1,0) -- (0,0);
\node at (-0.2,0.2) {$\mu$};
\node at (1.2,0.2) {$\rho$};
\node at (-0.2,-1.2) {$\la$};
\node at (1.25,-1.25) {$\nu$};
\node at (.5,-.5) {$a$};
\end{tikzpicture}
\end{center}
is a square in $\Lambda$ (so $a \in \{0,1\}$), then $\nu$ is determined by $\mu,\la,\rho,a$, and $\mu,a$ are determined by $\la,\rho,\nu$. This implies that the set of $F_\bullet$-growth diagrams is in bijection with both permutations and pairs of standard Young tableaux of the same shape. Given a permutation $\sigma$, one constructs the unique $F_\bullet$-growth diagram associated with $\sigma$ by filling the top row and left column of the grid with the empty partition, and using the bijections $F_\la$, along with the positions of 1's in the permutation matrix, to recursively fill in the rest of the grid. Similarly, given standard Young tableaux $P,Q$ of the same shape, one fills the right and bottom edges of the grid with the saturated chains in Young's lattice corresponding to $P$ and $Q$, and then recursively fills in the rest of the grid using the bijections $F_\la^{-1}$. Thus, each set of local growth rules induces a different bijective proof of the identity $n! = \sum_{\la \vdash n} (f_\la)^2$.

\begin{rem}
In effect, the framework of growth diagrams transforms a bijective proof of the up-down commutation relation into a bijective proof of the identity $n! = \sum_{\la \vdash n} (f_\la)^2$ by ``bijectivizing'' the algebraic argument given in the previous section.
\end{rem}

We denote by $\RS_{F_\bullet}$ the bijection from permutations to pairs of standard Young tableaux induced by the local growth rules $F_\bullet$. We now explain how to translate $\RS_{F_\bullet}$ into an insertion algorithm. Let $\sigma$ be a permutation, and $\Lambda$ the $F_\bullet$-growth diagram associated to $\sigma$. Let $T^{(0)}, \ldots, T^{(n)}$ be the sequence of partitions in column $j-1$ of $\Lambda$, from top to bottom. This sequence corresponds to a \emph{partial standard Young tableau} $T$, i.e., a tableau with increasing rows and columns and no repeated entries (as in the previous section, $T^{(i)}$ is the shape of the subtableau of $T$ consisting of entries at most $i$). Let $\hat T$ be the partial standard Young tableau corresponding to column $j$ of $\Lambda$. Suppose $\sigma(j) = k$, so that the unique 1 in column $j$ of the permutation matrix $A_\sigma$ occurs between rows $k-1$ and $k$ of $\Lambda$, as shown below.

\begin{center}
\begin{tikzpicture}
\node at (0,0) {$T^{(0)}$};
\node at (0,-1.5) {$T^{(1)}$};
\node at (0,-2*1.5) {$T^{(k-1)}$};
\node at (0,-3*1.5) {$T^{(k)}$};

\node at (1.7,0) {$\hat T^{(0)}$};
\node at (1.7,-1.5) {$\hat T^{(1)}$};
\node at (1.9,-2*1.5) {$\hat T^{(k-1)}$};
\node at (1.7,-3*1.5) {$\hat T^{(k)}$};

\foreach \x in {0,...,3}{
		\draw (0.6,-1.5*\x) -- (1.1,-1.5*\x);
	}
	
\draw (0,-0.45) -- (0,.45-1.5*1);
\draw (0,-0.45-1.5*2) -- (0,.45-1.5*3);
\draw (1.5,-0.4) -- (1.5,.45-1.5*1);
\draw (1.5,-0.4-1.5*2) -- (1.5,.45-1.5*3);

\draw [dotted, thick] (1.5*0.5,-1.5-.3) -- (1.5*0.5,-1.5*2+.3);
\draw [dotted, thick] (1.5*0.5,-1.5*3-.3) -- (1.5*0.5,-1.5*4+.3);
\node at (1.5*0.5,-1.5*0.5) {$0$};
\node at (1.5*0.5,-1.5*2.5) {$1$};
\end{tikzpicture}
\end{center}

\begin{lem}
\label{lem_growth_insertion}
In the situation described above, $\hat T$ is obtained by inserting $k$ into $T$ according to the following algorithm:
\begin{itemize}
\item Let $c$ be the outer corner of $T^{(k)}$ which corresponds to $F_{T^{(k)}}\left(T^{(k)}\right)$. Place $k$ into this cell. If $c$ is unoccupied in $T$, the process terminates.
\item If $c$ is occupied by an entry $z > k$, then $c$ is an inner corner of $T^{(z)}$, and $T^{(z-1)}$ is obtained by removing this inner corner from $T^{(z)}$. Let $c'$ be the outer corner of $T^{(z)}$ corresponding to $F_{T^{(z)}}\left(T^{(z-1)}\right)$, and add $z$ to the cell $c'$. If $c'$ is unoccupied in $T$, the process terminates. Otherwise, repeat this step with the entry $z' > z$ which occupies $c'$.
\end{itemize}
\end{lem}

\begin{proof}
It follows from the definition of growths that $T^{(i)} = \hat T^{(i)} $ for $i < k$, and that $T^{(k-1)} = T^{(k)}$. The partition $\hat T^{(k)}$ is therefore equal to $F_{T^{(k)}}\left(T^{(k)}\right)$; this explains the first step of the algorithm.

For $i>k$, the definition of growths implies that $T^{(i-1)} \lessdot \hat T^{(i-1)}$ and $T^{(i)} \lessdot \hat T^{(i)}$. If $T^{(i-1)} = T^{(i)}$, then since $\hat T^{(i-1)}$ is contained in $\hat T^{(i)}$, we must have $\hat T^{(i-1)} = \hat T^{(i)}$. This says that if $T$ does not have an entry $i$, then neither does $\hat T$. If $T^{(i-1)} \lessdot T^{(i)}$, then there are two possible cases:
\begin{enumerate}
\item $T^{(i)} \neq \hat T^{(i-1)}, \quad \hat T^{(i)} = T^{(i)} \cup \hat T^{(i-1)}$
\item $T^{(i)} = \hat T^{(i-1)}, \quad \hat T^{(i)} = F_{T^{(i)}}\left(T^{(i-1)}\right)$.
\end{enumerate} 
Let $i_0$ be the smallest value of $i > k$ for which the second case occurs (assuming there is such a value). If $i$ is an entry in $T$ which is strictly between $k$ and $i_0$, then the first case must occur; this says that $i$ is not in the cell $c$ into which $k$ was inserted, and $i$ remains in the same location in $\hat T$. The fact that the second case occurs for $i_0$, and not for any smaller value of $i$, means that $i_0$ is located in the cell $c$. Thus, $i_0 = z$, and this number appears in $\hat T$ in the outer corner of $T^{(z)}$ corresponding to the partition $F_{T^{(z)}}\left(T^{(z-1)}\right)$, as claimed. The argument now repeats for $i > z$.
\end{proof}

If $T$ is a partial standard Young tableau and $k$ is a number not appearing in $T$, we define the \emph{$F_\bullet$-insertion} of $k$ into $T$ by the algorithm of Lemma \ref{lem_growth_insertion}. It follows from the preceding discussion that if $\RS_{F_\bullet}(\sigma) = P,Q$, then $P$ can be obtained by the successive $F_\bullet$-insertion of $\sigma_1, \ldots, \sigma_n$ into the empty tableau, and $Q$ records the growth of $P$, just as for usual Robinson--Schensted insertion. In fact, as the reader may verify, the row and column versions of Robinson--Schensted described in \S \ref{sec: RSK} are the special cases of $F_\bullet$-insertion corresponding to the local growth rules $F_\bullet^\row$ and $F_\bullet^\col$ introduced in the previous section.

It follows immediately from the definition of $F_{\bullet}$-growth diagrams that the $F_\bullet$-insertion algorithms have the symmetry property
\[
\RS_{F_\bullet}(\sigma) = P,Q \iff \RS_{F_\bullet}(\sigma^{-1}) = Q,P.
\]
This is not at all obvious from the description of the insertion procedure, even in the case of row or column insertion.

For completeness, we give a simple description of the local growth rules arising from the row insertion bijections $F_\bullet^\row$. For a square

\begin{center}
\begin{tikzpicture}[scale=0.8]
\draw (0,0) -- (0,-1) -- (1,-1) -- (1,0) -- (0,0);
\node at (-0.2,0.2) {$\mu$};
\node at (1.2,0.2) {$\rho$};
\node at (-0.2,-1.2) {$\la$};
\node at (.5,-.5) {$a$};
\end{tikzpicture}
\end{center}

\noindent the southeast vertex $\nu$ is determined by the following rules:

\begin{enumerate}
\item If $\lambda \neq \rho$, then $\nu = \lambda \cup \rho$.
\item If $\la = \rho$ and $\la$ is obtained by adding $1$ to $\mu_i$, then $\nu$ is obtained by adding $1$ to $\lambda_{i+1}$.
\item If $\mu=\la=\rho$, then $\nu=\mu$ if $a = 0$, and $\nu$ is obtained by adding $1$ to $\mu_1$ if $a = 1$.
\end{enumerate}

\begin{ex}
The $F_\bullet^\row$-growth diagram associated to the permutation $\sigma=526134$ is shown below. As in Example \ref{ex_growths_213}, we omit the 0's in the permutation matrix, and write $X$ instead of 1.
\label{ex: Fomin growth diagram}
\begin{center}
\begin{tikzpicture}[scale=0.8]
	\node at (3.5*1.5,-.5*1.5) {$X$};
	\node at (1.5*1.5,-1.5*1.5) {$X$};
	\node at (4.5*1.5,-2.5*1.5) {$X$};
	\node at (5.5*1.5,-3.5*1.5) {$X$};
	\node at (.5*1.5,-4.5*1.5) {$X$};
	\node at (2.5*1.5,-5.5*1.5) {$X$};
	\foreach \x  in {0,...,6}{
		\draw (0,-1.5*\x) -- (1.5*6,-1.5*\x);
		\draw (1.5*\x,0) -- (1.5*\x,-1.5*6); 
		\node at (-.2,-1.5*\x+0.2) {$\emptyset$};
	}
	\foreach \x in {1,...,6}
		\node at (-.2+1.5*\x,0.2) {$\emptyset$};
	\node at (1*1.5-0.2,-1*1.5+0.2) {$\emptyset$};
	\node at (2*1.5-0.2,-1*1.5+0.2) {$\emptyset$};
	\node at (3*1.5-0.2,-1*1.5+0.2) {$\emptyset$};
	\node at (4*1.5-0.2,-1*1.5+0.2) {$1$};
	\node at (5*1.5-0.2,-1*1.5+0.2) {$1$};
	\node at (6*1.5-0.2,-1*1.5+0.2) {$1$};
	
	\node at (1*1.5-0.2,-2*1.5+0.2) {$\emptyset$};
	\node at (2*1.5-0.2,-2*1.5+0.2) {$1$};	
	\node at (3*1.5-0.2,-2*1.5+0.2) {$1$};	
	\node at (4*1.5-0.3,-2*1.5+0.2) {$11$};	
	\node at (5*1.5-0.3,-2*1.5+0.2) {$11$};	
	\node at (6*1.5-0.3,-2*1.5+0.2) {$11$};	
	
	\node at (1*1.5-0.2,-3*1.5+0.2) {$\emptyset$};
	\node at (2*1.5-0.2,-3*1.5+0.2) {$1$};
	\node at (3*1.5-0.2,-3*1.5+0.2) {$1$};
	\node at (4*1.5-0.3,-3*1.5+0.2) {$11$};
	\node at (5*1.5-0.3,-3*1.5+0.2) {$21$};
	\node at (6*1.5-0.3,-3*1.5+0.2) {$21$};
	
	\node at (1*1.5-0.2,-4*1.5+0.2) {$\emptyset$};
	\node at (2*1.5-0.2,-4*1.5+0.2) {$1$};
	\node at (3*1.5-0.2,-4*1.5+0.2) {$1$};
	\node at (4*1.5-0.3,-4*1.5+0.2) {$11$};
	\node at (5*1.5-0.3,-4*1.5+0.2) {$21$};
	\node at (6*1.5-0.3,-4*1.5+0.2) {$31$};
	
	\node at (1*1.5-0.2,-5*1.5+0.2) {$1$};
	\node at (2*1.5-0.3,-5*1.5+0.2) {$11$};
	\node at (3*1.5-0.3,-5*1.5+0.2) {$11$};
	\node at (4*1.5-0.4,-5*1.5+0.2) {$111$};
	\node at (5*1.5-0.4,-5*1.5+0.2) {$211$};
	\node at (6*1.5-0.4,-5*1.5+0.2) {$311$};
	
	\node at (1*1.5-0.2,-6*1.5+0.2) {$1$};
	\node at (2*1.5-0.3,-6*1.5+0.2) {$11$};
	\node at (3*1.5-0.3,-6*1.5+0.2) {$21$};
	\node at (4*1.5-0.4,-6*1.5+0.2) {$211$};
	\node at (5*1.5-0.4,-6*1.5+0.2) {$221$};
	\node at (6*1.5-0.4,-6*1.5+0.2) {$321$};
\end{tikzpicture}
\end{center}
Reading the right column and bottom row of this diagram, we find
\[
P = \young(134,26,5) \;\;, \qquad Q= \young(136,25,4) \;\;,
\]
in agreement with Example \ref{ex: RSK algorithm by insertion}.
\end{ex}

\section{Preliminaries on Macdonald polynomials}
\label{sec: Macdonald polynomials}

\subsection{Monomial expansion of Macdonald polynomials}
\label{sec: Monomial expansion}

We review some basic properties of Macdonald polynomials, following \cite[Ch. VI]{Mac}.

The \emph{Macdonald symmetric functions} $P_\la(\mb{x}; q,t)$ are symmetric functions in variables $\mb{x} = (x_1, x_2, \ldots)$ with coefficients in the field $\bb{Q}(q,t)$ of rational functions in two additional variables $q$ and $t$. They were originally defined as the orthogonal basis obtained by applying the Gram--Schmidt orthogonalization procedure to the basis of monomial symmetric functions (ordered by dominance order), with respect to a certain inner product $\langle \cdot, \cdot \rangle_{q,t}$ that depends on $q$ and $t$. The elements of the basis dual to the $P_\la(\mb{x};q,t)$ with respect to this inner product are denoted by $Q_\la(\mb{x};q,t)$; in other words, $Q_\la$ is proportional to $P_\la$, with proportionality ``constant'' $\langle P_\la, P_\la \rangle_{q,t} \in \bb{Q}(q,t)$ (see \eqref{eq_P_Q_proportional} and the discussion preceding it for an explicit formula for this constant). We will usually refer to the symmetric functions $P_\la$ and $Q_\la$ as \emph{Macdonald polynomials}, even though they are not polynomials over any ring!

Macdonald polynomials generalize many families of symmetric functions. Of importance to this paper are the $q$-Whittaker functions $W_\la(\mathbf{x};q)$, the Hall--Littlewood functions $P_\la(\mathbf{x};t)$, and the Schur functions $s_\la(\mathbf{x})$, which are obtained from the Macdonald polynomials by
\[
W_\la(\mathbf{x};q) = P_\la(\mathbf{x};q,0), \qquad P_\la(\mathbf{x};t) = P_\la(\mathbf{x};0,t), \qquad s_\la(\mathbf{x}) = P_\la(\mathbf{x};q,q).
\]
It follows easily from the definition of the inner product $\langle \cdot, \cdot \rangle_{q,t}$ that the Macdonald polynomials satisfy a generalization of the Cauchy identity for Schur functions.

\begin{thm}[{\cite[Ch. VI (4.13)]{Mac}}]
\label{thm_Cauchy_Mac}
Let $\tb{x} = (x_1, x_2, \ldots)$ and $\tb{y} = (y_1, y_2, \ldots)$ be two sets of variables. Then
\begin{equation}
\label{eq_Cauchy_Mac}
\prod_{i,j \geq 1} \dfrac{(tx_iy_j ; q)_\infty}{(x_iy_j ; q)_\infty} = \sum_\la P_\la(\tb{x};q,t) Q_\la(\tb{y};q,t),
\end{equation}
where $(\alpha;q)_\infty = (1-\alpha)(1-\alpha q)(1-\alpha q^2) \cdots$ is the \emph{infinite $q$-Pochhammer symbol}.
\end{thm}

With a good deal of effort, Macdonald was able to derive explicit formulas for the monomial expansions of $P_\la$ and $Q_\la$ as weighted sums over semistandard Young tableaux. To describe these expansions, we need some notation. For a partition $\lambda$ and a cell $c \in \lambda$, define
\[
b_\lambda(c) = \dfrac{1-q^{a_\la(c)}t^{\ell_\la(c)+1}}{1-q^{a_\la(c)+1}t^{\ell_\la(c)}}.
\]
It is useful to view the numerator and denominator of $b_\la(c)$ as two different $(q,t)$-analogues of the hook-length $h_\la(c)$, which both specialize to the $q$-analogue $1-q^{h_\la(c)}$ when $q = t$. For $\mu \subseteq \la$, define
\[
\psi_{\la/\mu}(q,t) = \prod_{c \in \mc{R}_{\la/\mu} - \, \mc{C}_{\la/\mu}} \dfrac{b_\mu(c)}{b_\la(c)}, \qquad\qquad
\vp_{\la/\mu}(q,t) = \prod_{c \in \lambda/\mu}b_\lambda(c) \prod_{c \in \mc{C}_{\la/\mu}} \dfrac{b_\la(c)}{b_\mu(c)},
\]
where $\mc{R}_{\la/\mu}$ is the set of all cells of $\mu$ which are in a row that intersects the skew shape $\la/\mu$, and $\mc{C}_{\la/\mu}$ is the set of all cells of $\mu$ which are in a column that intersects $\la/\mu$.\footnote{In \cite{Mac}, Macdonald defines $\mc{R}_{\la/\mu}$ and $\mc{C}_{\la/\mu}$ to include the cells in $\la/\mu$, which allows him to express $\vp_{\la/\mu}$ in a more compact form. We have found our modified definition of $\mc{R}_{\la/\mu}$ and $\mc{C}_{\la/\mu}$ to be more convenient for the purposes of this paper, even though it requires us to express $\vp_{\la/\mu}$ less compactly.} Note that if $c \in \la - \left(\la/\mu \cup \mc{R}_{\la/\mu} \cup \mc{C}_{\la/\mu}\right)$, then $b_\la(c) = b_\mu(c)$, so the rational functions $\psi_{\la/\mu}$ and $\vp_{\la/\mu}$ are related by
\begin{equation}
\label{eq_psi_vp_relation}
\vp_{\la/\mu}(q,t) = \dfrac{b_\la(q,t)}{b_\mu(q,t)} \psi_{\la/\mu}(q,t),
\end{equation}
where $b_\kappa(q,t) = \prod_{c \in \kappa} b_\kappa(c)$.

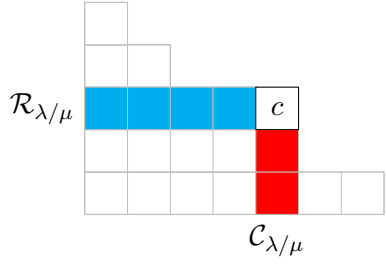
\begin{figure}
\begin{center}
\begin{tikzpicture}
	\Yboxdim{16 pt}
	\Ylinecolour{lightgray}
	\tyng(0,0,7,5,5,2,1)
	\Ycyan \tgyoung(0pt,0pt,,,;;;;)
	\Yred \tgyoung(0pt,0pt,::::;,::::;)
	\Ywhite
	\Ylinecolour{black}
	\tgyoung(0pt,0pt,,,::::;)
	\node at (4.5*16pt,2.5*16pt) {$c$};
	\node at (4.5*16pt,-10pt) {$\mc{C}_{\la/\mu}$};
	\node at (-16pt,2.5*16pt) {$\mc{R}_{\la/\mu}$};
\end{tikzpicture}
\end{center}
\caption{\label{fig: lambda mu} A skew diagram $\la/\mu$ consisting of a single cell $c$. The cells of $\mc{R}_{\lambda/\mu}$ and $\mc{C}_{\lambda/\mu}$ are colored blue and red, respectively.}
\end{figure}

\begin{ex}
For $\la = (7,5,5,2,1)$ and $\mu = (7,5,4,2,1)$ as in Figure \ref{fig: lambda mu}, we have
\[
\psi_{\la/\mu}(q,t) = \dfrac{b_{0,0} b_{2,1}b_{3,2}}{b_{2,0} b_{3,1}b_{4,2}}, \qquad\qquad
\vp_{\la/\mu}(q,t) = \dfrac{b_{0,1} b_{2,2}}{b_{2,1}},
\]
where $b_{i,j}$ is shorthand for $\frac{1-q^it^{j+1}}{1-q^{i+1}t^j}$.
\end{ex}

For a semistandard Young tableau $T$, define rational functions $\psi_T(q,t), \vp_T(q,t)$ by
\[
\psi_T(q,t) = \prod_{i \geq 1} \psi_{T^{(i)}/T^{(i-1)}}(q,t), \qquad\qquad
\vp_T(q,t) = \prod_{i \geq 1} \vp_{T^{(i)}/T^{(i-1)}}(q,t)
\]
where, as in previous sections, $T^{(i)}$ is the shape of the subtableau of $T$ consisting of entries less than or equal to $i$.

\begin{thm}[{\cite[Ch. VI ($7.13,7.13'$)]{Mac}}]
\label{thm_Mac_monomial}
The Macdonald polynomials $P_\la(\tb{x}; q,t)$ and $Q_\la(\tb{x}; q,t)$ have the following monomial expansions over semistandard Young tableaux of shape $\la$:
\[
P_\la (\mathbf{x}; q,t) = \sum_{T \in \SSYT(\la)} \psi_T(q,t) \tb{x}^T, \qquad\qquad
Q_\la (\mathbf{x}; q,t) = \sum_{T \in \SSYT(\la)} \vp_T(q,t) \tb{x}^T.
\]
\end{thm}

It follows from Theorem \ref{thm_Mac_monomial} and \eqref{eq_psi_vp_relation} that the functions $P_\la$ and $Q_\la$ are related by
\begin{equation}
\label{eq_P_Q_proportional}
Q_\la(\mb{x};q,t) = b_\la(q,t)  P_\la(\mb{x};q,t).
\end{equation}

In this paper, we take the somewhat unusual perspective of viewing Theorem \ref{thm_Mac_monomial} as the \emph{definition} of the Macdonald polynomials $P_\la$ and $Q_\la$. The theory of Schur functions can be developed in elegant combinatorial fashion by taking the monomial expansion over semistandard Young tableaux as the starting point (for example, this is Stanley's approach in \cite[Ch. 7]{EC2}). We believe that trying to mimic this approach in the more general Macdonald setting will lead to interesting combinatorial and probabilistic results.

\subsection{Up and down operators for Macdonald polynomials}
\label{sec_qt_up_down}

In \S \ref{sec: Preliminaries}, we saw how the squarefree part of the Cauchy identity for Schur functions can be proved combinatorially, starting from the monomial expansion of the Schur functions. We now take the first steps toward an analogous proof of the squarefree part of the generalized Cauchy identity \eqref{eq_Cauchy_Mac}.

It is a standard exercise (see, e.g., \cite[Ch. VI.2, Ex. 1]{Mac}) to show that
\begin{equation}
\label{eq_Poch}
\dfrac{(tz;q)_\infty}{(z;q)_\infty} = \sum_{k \geq 0} \dfrac{(t;q)_k}{(q;q)_k} z^k,
\end{equation}
where $(\alpha;q)_k = (1-\alpha) (1-\alpha q) \cdots (1-\alpha q^{k-1})$. Using this identity and the monomial expansions of the Macdonald polynomials, one can rewrite the generalized Cauchy identity in the form
\[
\sum_{A = (a_{ij})} \prod_{i,j \geq 1} \dfrac{(t;q)_{a_{ij}}}{(q;q)_{a_{ij}}} (x_iy_j)^{a_{ij}} = \sum_{P,Q} \psi_P(q,t) \vp_Q(q,t) \mb{x}^P \mb{y}^Q,
\]
where the first sum is over nonnegative integer matrices with finite support, and the second sum is over pairs of SSYTs of the same shape. In particular, taking the coefficients of the squarefree monomial $x_1 \cdots x_n y_1 \cdots y_n$, we obtain the identity
\begin{equation}
\label{eq: squarefree Macdonald Cauchy}
\dfrac{(1-t)^n}{(1-q)^n} n! = \sum_{\la \vdash n} \sum_{P,Q} \psi_P(q,t) \vp_Q(q,t),
\end{equation}
where the inner sum is over pairs of standard Young tableaux of shape $\la$.

The identity \eqref{eq: squarefree Macdonald Cauchy} can be proved with the aid of a suitable modification of the up and down operators from \S \ref{sec: Up and Down}. Let $\bb{Q}(q,t) \bb{Y}$ be the vector space over the field of rational functions in $q$ and $t$ with (orthonormal) basis $\bb{Y}$, and define linear operators $U_{q,t}$ and $D_{q,t}$ on $\bb{Q}(q,t) \bb{Y}$ by
\begin{equation}
\label{eq_def_UD_qt}
U_{q,t} \la = \sum_{\nu \in \U(\la)} \psi_{\nu/\la}(q,t) \, \nu, \qquad\qquad
D_{q,t} \la = \sum_{\mu \in \D(\la)} \vp_{\la/\mu}(q,t) \, \mu.
\end{equation}
If $\kappa/\rho$ consists of a single cell, then $\mc{R}_{\kappa/\rho}$ and $\mc{C}_{\kappa/\rho}$ are disjoint, and the cell $\kappa/\rho$ contributes a factor of $\frac{1-t}{1-q}$ to $\vp_{\kappa/\rho}$. Thus, the coefficients in \eqref{eq_def_UD_qt} are given by
\begin{equation}
\label{eq_psi_vp_one_cell}
\psi_{\nu/\la}(q,t) = \prod_{c \in \mc{R}_{\nu/\la}} \dfrac{b_\la(c)}{b_\nu(c)}, \qquad\qquad
\vp_{\la/\mu}(q,t) = \dfrac{1-t}{1-q} \prod_{c \in \mc{C}_{\la/\mu}} \dfrac{b_\la(c)}{b_\mu(c)}.
\end{equation}
 
Taking the monomial expansions of Theorem \ref{thm_Mac_monomial} as the definition of $P_\la$ and $Q_\la$, it follows immediately from \eqref{eq_def_UD_qt} that the right-hand side of \eqref{eq: squarefree Macdonald Cauchy} is equal to $\langle D_{q,t}^nU_{q,t}^n \emptyset, \emptyset \rangle$. Thus, by exactly the same argument given in \S \ref{sec: Up and Down}, \eqref{eq: squarefree Macdonald Cauchy} can be deduced from the following commutation relation.

\begin{thm}
\label{thm: q,t commutator}
The $(q,t)$-up and down operators satisfy the commutation relation
\begin{equation}
\label{eq: q,t commutator}
D_{q,t}U_{q,t}-U_{q,t}D_{q,t} = \frac{1-t}{1-q} I.
\end{equation}
\end{thm}

The main goal of this paper is to give a direct proof of Theorem \ref{thm: q,t commutator}. Reasoning as in \S \ref{sec: Up and Down}, one sees that \eqref{eq: q,t commutator} is equivalent to the identities
\begin{equation}
\label{eq_la_la_qt}
\sum_{\nu \in \U(\la)} \psi_{\nu/\la} \vp_{\nu/\la} = \dfrac{1-t}{1-q} + \sum_{\mu \in \D(\la)} \psi_{\la/\mu} \vp_{\la/\mu}
\end{equation}
for all partitions $\la$, and
\begin{equation}
\label{eq_la_rho_qt}
\sum_{\nu \in \U(\la) \cap \U(\rho)} \psi_{\nu/\la} \vp_{\nu/\rho} = \sum_{\mu \in \D(\la) \cap \D(\rho)} \psi_{\rho/\mu} \vp_{\la/\mu}
\end{equation}
for $\la \neq \rho$ (for notational convenience, we suppress the dependence of the rational functions $\psi$ and $\vp$ on $q,t$.)

We can quickly dispense with \eqref{eq_la_rho_qt}. As noted in \S \ref{sec: Up and Down}, the sets $\U(\la) \cap \U(\rho)$ and $\D(\la) \cap \D(\rho)$ are either empty, or they are the singleton sets consisting of $\la \cup \rho$ and $\la \cap \rho$, respectively. In the former case, there is nothing to prove. The latter case occurs when $\la$ and $\rho$ are the same size and differ by exactly two boxes, which necessarily lie in distinct rows and columns. In this case, set $\nu = \la \cup \rho$ and $\mu = \la \cap \rho$. In light of \eqref{eq_psi_vp_one_cell}, we must show that
\[
\dfrac{1-t}{1-q} \prod_{c \in \mc{R}_{\nu/\la}} \dfrac{b_\la(c)}{b_\nu(c)} \prod_{c \in \mc{C}_{\nu/\rho}} \dfrac{b_\nu(c)}{b_\rho(c)} = \dfrac{1-t}{1-q} \prod_{c \in \mc{R}_{\rho/\mu}} \dfrac{b_\mu(c)}{b_\rho(c)} \prod_{c \in \mc{C}_{\la/\mu}} \dfrac{b_\la(c)}{b_\mu(c)}.
\]
It is clear that $\mc{R}_{\nu/\la} = \mc{R}_{\rho/\mu}$ and $\mc{C}_{\nu/\rho} = \mc{C}_{\la/\mu}$; denote these sets $\mc{R}$ and $\mc{C}$, respectively, and let $\mc{S} = \mc{R} \cap \mc{C}$. Note that $\mc{S}$ is either empty, or consists of a single cell. If $c \in \mc{R} - \mc{S}$, then $b_\la(c) = b_\mu(c)$ and $b_\nu(c) = b_\rho(c)$. Similarly, if $c \in \mc{C} - \mc{S}$, then $b_\nu(c) = b_\la(c)$ and $b_\rho(c) = b_\mu(c)$. Thus, we have
\begin{align*}
\prod_{c \in \mc{R}_{\nu/\la}} \dfrac{b_\la(c)}{b_\nu(c)} \prod_{c \in \mc{C}_{\nu/\rho}} \dfrac{b_\nu(c)}{b_\rho(c)}
&= \prod_{c^* \in \mc{S}} \dfrac{b_\la(c^*)}{b_\rho(c^*)} \prod_{c \in \mc{R} - \mc{S}} \dfrac{b_\la(c)}{b_\nu(c)} \prod_{c \in \mc{C} - \mc{S}} \dfrac{b_\nu(c)}{b_\rho(c)} \\
&= \prod_{c^* \in \mc{S}} \dfrac{b_\la(c^*)}{b_\rho(c^*)} \prod_{c \in \mc{R} - \mc{S}} \dfrac{b_\mu(c)}{b_\rho(c)} \prod_{c \in \mc{C} - \mc{S}} \dfrac{b_\la(c)}{b_\mu(c)}
= \prod_{c \in \mc{R}_{\rho/\mu}} \dfrac{b_\mu(c)}{b_\rho(c)} \prod_{c \in \mc{C}_{\la/\mu}} \dfrac{b_\la(c)}{b_\mu(c)},
\end{align*}
which proves \eqref{eq_la_rho_qt}.

The proof of \eqref{eq_la_la_qt} is much more involved. In effect, we superimpose all the bijections $F_\la : \mc{D}^*(\la) \rightarrow \mc{U}(\la)$ to obtain a set of ``probabilistic local growth rules.'' This is the subject of \S \ref{sec_qt_growth_rules}.

\begin{rem}
As in the Schur case, one can generalize the $(q,t)$-up and down operators to allow for the addition and removal of horizontal strips. The Cauchy identity for Macdonald polynomials is then an immediate consequence of a certain commutation relation for these generalized operators. This commutation relation (which encompasses Theorem \ref{thm: q,t commutator}) is a special case of the skew Cauchy identity for Macdonald polynomials, which, in turn, can be easily derived from the Macdonald Cauchy identity. Thus, the up-down operator approach is really an equivalent formulation of the Macdonald Cauchy identity. The point for us is that the up-down operators come directly from the monomial expansions.
\end{rem}

\section{$(q,t)$-local growth rules}
\label{sec_qt_growth_rules}

In \S \ref{sec: probabilistic bijections}, we define the notion of a probabilistic bijection between weighted sets. In \S \ref{sec_weighted_sets}-\ref{sec_proofs}, we complete the proof of the commutation relation for the $(q,t)$-up and down operators by constructing a probabilistic bijection between the sets $\D^*(\la)$ and $\U(\la)$. In \S \ref{sec_qrst}, we define the $(q,t)$-Robinson--Schensted correspondence by treating this probabilistic bijection as a set of probabilistic local growth rules.

\subsection{Probabilistic bijections}
\label{sec: probabilistic bijections}

The following definition is due to Bufetov and Petrov \cite{BufetovPetrov19}, although they use the term ``bijectivization'' (or ``coupling'') rather than ``probabilistic bijection.'' This notion also plays an important role in \cite{BufetovMatveev18}.
\begin{defn}
\label{def: prob bij}
Let $X$ and $Y$ be finite sets equipped with weight functions $\omega : X \rightarrow k$, $\ov{\omega} : Y \rightarrow k$, where $k$ is a field. A \emph{probabilistic bijection} from $(X,\omega)$ to $(Y,\ov{\omega})$ is a pair of maps $\mc{P},\ov{\mc{P}} : X \times Y \rightarrow k$ satisfying
\begin{enumerate}
\item For each $x \in X$, $\ds \sum_{y \in Y} \mc{P}(x,y) = 1$.
\item For each $y \in Y$, $\ds \sum_{x \in X} \ov{\mc{P}}(x,y) = 1$.
\item For each $x \in X$ and $y \in Y$, $\ds \omega(x)\mc{P}(x,y) = \ov{\mc{P}}(x,y)\ov{\omega}(y)$.
\end{enumerate}
\end{defn}
We will usually write $\mc{P}(x \rightarrow y)$ for $\mc{P}(x,y)$ and $\ov{\mc{P}}(x \leftarrow y)$ for $\ov{\mc{P}}(x,y)$, and think of the former as the ``probability'' of moving ``forward'' from $x$ to $y$, and the latter as the ``probability'' of moving ``backward'' from $y$ to $x$. Thus, (1) says that $\mc{P}$ defines a ``probability distribution'' on $Y$ for each $x$, and (2) says that $\ov{\mc{P}}$ defines a ``probability distribution'' on $X$ for each $y$. We put ``probability'' in quotes because we do not require $\mc{P}(x,y), \ov{\mc{P}}(x,y) \in [0,1]$ (they need not even be real-valued). We refer to (3) as the \emph{compatibility condition}.

\begin{lem}
\label{lem_prob_bij_easy}
If $\mc{P},\ov{\mc{P}}$ is a probabilistic bijection from $(X,\omega)$ to $(Y,\ov{\omega})$, then
\[
\sum_{x \in X} \omega(x) = \sum_{y \in Y} \ov{\omega}(y).
\]
\end{lem}

\begin{proof}
Using properties (1), (3), and (2) successively, we compute
\[
\sum_{x \in X} \omega(x) = \sum_{x \in X} \sum_{y \in Y} \omega(x) \mc{P}(x \rightarrow y) = \sum_{y \in Y} \sum_{x \in X} \ov{\mc{P}}(x \leftarrow y) \ov{\omega}(y) = \sum_{y \in Y} \ov{\omega}(y). \qedhere
\]
\end{proof}

The existence of a probabilistic bijection $\mc{P},\ov{\mc{P}}$ also implies the more refined identities
\[
\sum_{x \in X} \omega(x) \mc{P}(x \rightarrow y) = \ov{\omega}(y), \qquad\qquad \sum_{y \in Y} \ov{\mc{P}}(x \leftarrow y) \ov{\omega}(y) = \omega(x).
\]

\begin{rem}
If $\omega(x) = \ov{\omega}(y) = 1$ for all $x \in X$ and $y \in Y$, and $f: X \rightarrow Y$ is a bijection, then we may take $\mc{P}(x \rightarrow y) = \ov{\mc{P}}(x \leftarrow y) = \delta_{y,f(x)}$. Thus, the notion of probabilistic bijection generalizes that of bijection, allowing for situations in which $X$ and $Y$ have different cardinalities, or the same cardinality but differently distributed weight functions, etc.
\end{rem}

\begin{rem}
Suppose $\omega(x), \ov{\omega}(y) \in \bb{R}_{>0}$ for all $x \in X, y \in Y$, so that $\omega$ and $\ov{\omega}$ induce probability distributions $\mc{P}_X$ and $\mc{P}_Y$ on $X$ and $Y$ by
\[
\mc{P}_X(x)=\frac{\omega(x)}{\sum_{x \in X}\omega(x)}, \qquad \mc{P}_Y(y)=\frac{\ov{\omega}(y)}{\sum_{y \in Y}\ov{\omega}(y)}.
\]
Suppose further that
\begin{equation}
\label{eq_equality_of_weighted_sums}
\sum_{x \in X} \omega(x) = \sum_{y \in Y} \ov{\omega}(y).
\end{equation}
In this case, an $\bb{R}_{>0}$-valued probabilistic bijection between $(X,\omega)$ and $(Y,\ov{\omega})$ is equivalent to a joint distribution on $X \times Y$ whose marginal distributions are $\fP_X$ and $\fP_Y$, and whose conditional probabilities are given by the probabilistic bijection. Indeed, if $\fP,\bP$ is a probabilistic bijection, then
\[
Prob(x,y) = \fP_X(x) \fP(x \rightarrow y) = \bP(x \leftarrow y) \fP_Y(y)
\]
defines a joint distribution on $X \times Y$ with marginal distributions $\fP_X, \fP_Y$.

Conversely, a joint distribution $Prob(x,y)$ with marginal distributions $\fP_X,\fP_Y$ determines the probabilistic bijection given by
\[
\fP(x \rightarrow y) = \dfrac{Prob(x,y)}{\fP_X(x)}, \qquad\qquad \bP(x \leftarrow y) = \dfrac{Prob(x,y)}{\fP_Y(y)}.
\]
Note, however, that the compatibility condition for this probabilistic bijection is equivalent to the identity \eqref{eq_equality_of_weighted_sums}, so if our goal is to \emph{prove} \eqref{eq_equality_of_weighted_sums} by exhibiting a probabilistic bijection, the viewpoint of joint distributions is not especially useful.
\end{rem}

\subsection{The weighted sets $(\mc{D}^*(\la), \omega_\la)$ and $(\mc{U}(\la), \ov{\omega}_\la)$}
\label{sec_weighted_sets}

We return now to the identity
\begin{equation}
\label{eq_psi_vp}
\dfrac{1-t}{1-q} + \sum_{\mu \in \D(\la)} \psi_{\la/\mu} \vp_{\la/\mu} = \sum_{\nu \in \U(\la)} \psi_{\nu/\la} \vp_{\nu/\la}.
\end{equation}
Recall that $\D^*(\la) = \D(\la) \cup \{\la\}$. Define the \emph{weight} of $\mu \in \mc{D}^*(\la)$ and $\nu \in \mc{U}(\la)$ by
\[
\omega_\la(\mu) = \begin{cases} \ds \prod_{c \in \mc{R}_{\la/\mu}} \dfrac{b_\mu(c)}{b_\la(c)} \prod_{c \in \mc{C}_{\la/\mu}} \dfrac{b_\la(c)}{b_\mu(c)} & \text{ if } \mu \in \mc{D}(\la) \bigskip \\
1 & \text{ if } \mu = \la, \end{cases} \quad\quad
\ov{\omega}_\la(\nu) = \prod_{c \in \mc{R}_{\nu/\la}} \dfrac{b_\la(c)}{b_\nu(c)} \prod_{c \in \mc{C}_{\nu/\la}} \dfrac{b_\nu(c)}{b_\la(c)},
\]
so that equation \eqref{eq_psi_vp} becomes (after using \eqref{eq_psi_vp_one_cell} and dividing both sides by $\frac{1-t}{1-q}$)
\begin{equation}
\label{eq_weights}
\sum_{\mu \in \mc{D}^*(\la)} \omega_\la(\mu) = \sum_{\nu \in \mc{U}(\la)} \ov{\omega}_\la(\nu).
\end{equation}
The following example shows that, although the sets $\D^*(\la)$ and $\U(\la)$ have the same cardinality, the weights do not allow for a bijective proof of \eqref{eq_weights}.

\begin{ex}
\label{ex_weights_d=1}
Let $\la = (h^v)$ be a rectangular partition. The elements of $\mc{D}^*(\la)$ and $\mc{U}(\la)$ are shown below (for $h = 8, v = 4$), together with their weights. The reader may wish to verify \eqref{eq_weights} in this case.
\begin{center}
\begin{tikzpicture}[scale=0.6]

\draw[dashed] (5.5,7.5) -- (5.5,-0.5);

\draw (-3.8,7) node[left]{$\omega_\la$};
\draw (3,7) node[left]{$\mc{D}^*(\la)$};
\draw (8,7) node[right]{$\mc{U}(\la)$};
\draw (15,7) node[right]{$\ov{\omega}_\la$};

\draw[step=0.5, gray, thin] (0,4) grid (4,6);
\draw (-4.5,5) node{$1$};
\draw (0,4) rectangle (4,6);
\draw (2,3.6) node{$h$};
\draw (-0.4,5) node{$v$};

\draw[step=0.5, gray, thin] (0,0) grid (4,1.5);
\draw[step=0.5, gray, thin] (0,1.5) grid (3.5,2);
\draw (-4.5,1) node{$\dfrac{(1-t^v)(1-q^h)}{(1-qt^{v-1})(1-q^{h-1}t)}$};
\draw (0,0) --(4,0) -- (4,1.5)  -- (3.5,1.5) -- (3.5,2) -- (0,2) -- (0,0);
\draw (2,-0.4) node{$h$};
\draw (-0.4,1) node{$v$};

\draw[step=0.5, gray, thin] (7,4) grid (11,6);
\draw (15.5,5) node{$\dfrac{(1-t^v)(1-q^{h+1}t^{v-1})}{(1-qt^{v-1})(1-q^ht^v)}$};
\draw (7,4) -- (11.5,4) -- (11.5,4.5)  -- (11,4.5) -- (11,6) -- (7,6) -- (7,4);
\draw (9,3.6) node{$h$};
\draw (6.6,5) node{$v$};

\draw[step=0.5, gray, thin] (7,0) grid (11,2);
\draw (15.5,1) node{$\dfrac{(1-q^{h-1}t^{v+1})(1-q^h)}{(1-q^ht^v)(1-q^{h-1}t)}$};
\draw (7,0) -- (11,0) -- (11,2)  -- (7.5,2) -- (7.5,2.5) -- (7,2.5) -- (7,0);
\draw (9,-0.4) node{$h$};
\draw (6.6,1) node{$v$};

\end{tikzpicture}
\end{center}
\end{ex}

Since we cannot prove \eqref{eq_weights} with a bijection, we believe the next best thing is to prove it via a probabilistic bijection (Definition \ref{def: prob bij}). That is, for each $\mu \in \mc{D}^*(\la)$ and $\nu \in \mc{U}(\la)$, we will define ``probabilities'' $\mc{P}_\la(\mu \rightarrow \nu)$ and $\ov{\mc{P}}_\la(\mu \leftarrow \nu)$ such that
\begin{equation}
\label{eq_prob_bij_forward}
\sum_{\nu \in \mc{U}(\la)} \mc{P}_\la(\mu \rightarrow \nu) = 1 \quad\quad\quad \text{ for } \mu \in \mc{D}^*(\la),
\end{equation}
\begin{equation}
\label{eq_prob_bij_backward}
\sum_{\mu \in \mc{D}^*(\la)} \ov{\mc{P}}_\la(\mu \leftarrow \nu) = 1 \quad\quad\quad \text{ for } \nu \in \mc{U}(\la),
\end{equation}
and
\begin{equation}
\label{eq_prob_bij_compatible}
\omega_\la(\mu) \mc{P}_\la(\mu \rightarrow \nu) = \ov{\mc{P}}_\la(\mu \leftarrow \nu) \ov{\omega}_\la(\nu).
\end{equation}

\begin{rem}
We write ``probabilities'' because these expressions will actually be rational functions in $q$ and $t$, rather than numbers in $[0,1]$. However, we will see in Remark \ref{rem_row_probs}(1) that when $q,t \in [0,1)$ or $q,t \in (1,\infty)$, these expressions lie in $[0,1]$, and are thus actual probabilities. For this reason, we will consider the rational functions themselves to be probabilities, even though this is, strictly speaking, an abuse of terminology.
\end{rem}

The main ingredient in defining the probabilities $\mc{P}_\la$ and $\ov{\mc{P}}_\la$ is the following pair of rational expressions. For integers $i,j$, set $[i,j] = [i,j]_{q,t} = 1-q^it^j$.

\begin{defn}
For partitions $\la \lessdot \nu$, define
\begin{align*}
\alpha_{\nu/\la}(q,t) &= \prod_{c \in \mc{R}_{\nu/\la}} \dfrac{[a_\la(c),\ell_\la(c)+1]}{[a_\nu(c),\ell_\nu(c)+1]} \prod_{c \in \mc{C}_{\nu/\la}} \dfrac{[a_\la(c)+1,\ell_\la(c)]}{[a_\nu(c)+1,\ell_\nu(c)]}, \\
\ov{\alpha}_{\nu/\la}(q,t) &= \prod_{c \in \mc{R}_{\nu/\la}} \dfrac{[a_\la(c)+1,\ell_\la(c)]}{[a_\nu(c)+1,\ell_\nu(c)]} \prod_{c \in \mc{C}_{\nu/\la}} \dfrac{[a_\la(c),\ell_\la(c)+1]}{[a_\nu(c),\ell_\nu(c)+1]}.
\end{align*}
\end{defn}

It is clear that
\begin{equation}
\label{eq_half_psi_vp}
\ov{\omega}_\la(\nu) = \dfrac{\alpha_{\nu/\la}(q,t)}{\ov{\alpha}_{\nu/\la}(q,t)}.
\end{equation}
To get an intuitive feel for these expressions, it is helpful to set $q = t$ and take the limit $q \rightarrow 1$. For a partition $\kappa$, let $H_\kappa = \prod_{c \in \kappa} h_\kappa(c)$ be the product of the hook-lengths of $\kappa$.

\begin{lem}
\label{lem_q=t=1}
Suppose $\la \lessdot \nu$. Then
\[
\lim_{q \rightarrow 1} \alpha_{\nu/\la}(q,q) = \lim_{q \rightarrow 1} \ov{\alpha}_{\nu/\la}(q,q) = \dfrac{H_\la}{H_\nu}.
\]
\end{lem}

\begin{proof}
We have
\[
\lim_{q \rightarrow 1} \alpha_{\nu/\la}(q,q) = \lim_{q \rightarrow 1} \ov{\alpha}_{\nu/\la}(q,q) = \prod_{c \in \mc{R}_{\nu/\la} \cup \, \mc{C}_{\nu/\la}} \dfrac{h_\la(c)}{h_\nu(c)}.
\]
This product is equal to $H_\nu/H_\la$ because $h_\nu(c^*) = 1$ for the single cell $c^* \in \nu/\la$, and $h_\la(c) = h_\nu(c)$ for each $c \in \la$ which is not in the same row or column as $c^*$.
\end{proof}

In light of Lemma \ref{lem_q=t=1}, we view $\alpha_{\nu/\la}(q,t)$ and $\ov{\alpha}_{\nu/\la}(q,t)$ as $(q,t)$-analogues of the ratio $H_\la/H_\nu$. For later use, we record a conjugation symmetry of these $(q,t)$-analogues.

\begin{lem}
\label{lem_conjugation}
We have
\[
\alpha_{\nu/\la}(t,q) = \alpha_{\nu'/\la'}(q,t), \quad\quad\quad \ov{\alpha}_{\nu/\la}(t,q) = \ov{\alpha}_{\nu'/\la'}(q,t).
\]
\end{lem}

\begin{proof}
These identities follow from the observation that
\[
\prod_{c \in \mc{R}_{\nu/\la}} \dfrac{[a_\la(c),\ell_\la(c)+1]_{t,q}}{[a_\nu(c),\ell_\nu(c)+1]_{t,q}} = \prod_{c' \in \mc{C}_{\nu'/\la'}} \dfrac{[a_{\la'}(c')+1,\ell_{\la'}(c')]_{q,t}}{[a_{\nu'}(c')+1,\ell_{\nu'}(c')]_{q,t}}. \qedhere
\]
\end{proof}

\subsection{The probabilities}
\label{sec_probabilities}

In this section, we give our first definition of the probabilities $\mc{P}_\la(\mu \rightarrow \nu)$ and $\ov{\mc{P}}_\la(\mu \leftarrow \nu)$, and discuss several enumerative identities resulting from the $q = t \rightarrow 1$ specialization. In \S \ref{sec_formulas} and \S \ref{sec_proofs}, we give two alternative formulations of the probabilities that are better suited for certain purposes, such as explicit calculations.

\subsubsection{The case $\mu = \la$}
\label{sec_mu=la}

We start by defining the probabilities $\mc{P}_\la(\mu \rightarrow \nu)$, $\ov{\mc{P}}_\la(\mu \leftarrow \nu)$ in the case $\mu = \la$. Set $n(\nu/\lambda) = n(\nu) - n(\lambda)$, where $n(\kappa) = \sum_{c \in \kappa} \ell_\kappa(c)$.

\begin{defn}
\label{defn_r=0}
For $\nu \in \mc{U}(\la)$, define
\[
\mc{P}_\la(\la \rightarrow \nu) = t^{n(\nu/\lambda)} \alpha_{\nu/\la}(q,t), \quad\quad\quad \ov{\mc{P}}_\la(\la \leftarrow \nu) = t^{n(\nu/\lambda)} \ov{\alpha}_{\nu/\la}(q,t).
\]
\end{defn}

It follows from \eqref{eq_half_psi_vp} that these expressions satisfy the compatibility condition \eqref{eq_prob_bij_compatible}, and the following result shows that the expressions $\mc{P}_\la(\la \rightarrow \nu)$ are ``probabilities.''

\begin{thm}
\label{thm_r=0}
We have
\[
\sum_{\nu \in \mc{U}(\la)} t^{n(\nu/\lambda)} \alpha_{\nu/\la}(q,t) = 1.
\]
\end{thm}

We will give two proofs of this identity: an algebraic proof via Lagrange interpolation in \S\ref{sec_proofs}, and a probabilistic proof via a $(q,t)$-generalization of the Greene--Nijenhuis--Wilf hook walk in \S\ref{sec_hook_walk_qt}.

\begin{rem}
\label{rem_q=t=1}
In the limit $q = t \rightarrow 1$, Theorem \ref{thm_r=0} becomes (by Lemma \ref{lem_q=t=1}) the identity
\begin{equation}
\label{eq_upper_recursion}
\sum_{\nu \in \mc{U}(\la)} \dfrac{H_\la}{H_\nu} = 1.
\end{equation}
Let $f_\la$ be the number of standard Young tableaux of shape $\la$. Using the hook-length formula $f_\la = |\la|!/H_\la$, we may rewrite \eqref{eq_upper_recursion} as
\[
\sum_{\nu \in \mc{U}(\la)} f_\nu = (n+1) f_\la,
\]
where $n = |\la|$. This is a classical identity known as the \emph{upper recursion} for the numbers $f_\la$, and it has many different proofs. For example, it was proved by Greene, Nijenhuis, and Wilf in \cite{GNW2} using a ``random hook walk'' (see Theorem \ref{thm_GNW} below). It can also be proved by comparing the coefficients of a squarefree monomial on either side of the Pieri rule $h_1 s_\la = \sum_{\nu \in \mc{U}(\la)} s_\nu$.
\end{rem}

\subsubsection{The case $\mu \in \mc{D}(\la)$}
\label{sec_mu_neq_la}

For $\mu \in \mc{D}(\la)$, set
\[
\beta_{\la/\mu}(q,t) = \dfrac{1}{\alpha_{\la/\mu}(q,t)}, \quad\quad \ov{\beta}_{\la/\mu}(q,t) = \dfrac{1}{\ov{\alpha}_{\la/\mu}(q,t)}.
\]
Also set $n'(\la/\mu) = n'(\la) - n'(\mu)$, where $n'(\kappa) = n(\kappa') = \sum_{c \in \kappa} a_\kappa(c)$.

\begin{defn}
\label{defn_r>0}
For $\mu \in \mc{D}(\la)$ and $\nu \in \mc{U}(\la)$, define
\begin{align*}
\mc{P}_\la(\mu \rightarrow \nu) &= t^{n(\nu/\la) - n(\la/\mu)-1} \dfrac{\alpha_{\nu/\la}(q,t) \beta_{\la/\mu}(q,t)}{\gamma_{\nu/\la/\mu}(q,t)}, \\
\ov{\mc{P}}_\la(\mu \leftarrow \nu) &= t^{n(\nu/\la) - n(\la/\mu)-1} \dfrac{\ov{\alpha}_{\nu/\la}(q,t) \ov{\beta}_{\la/\mu}(q,t)}{\gamma_{\nu/\la/\mu}(q,t)},
\end{align*}
where
\[
\gamma_{\nu/\la/\mu}(q,t) = \dfrac{(1-q^{n'(\la/\mu)-n'(\nu/\la)}t^{n(\nu/\la) - n(\la/\mu)})(1-q^{n'(\la/\mu)-n'(\nu/\la)+1}t^{n(\nu/\la) - n(\la/\mu)-1})}{(1-q)(1-t)}.
\]
\end{defn}

It is immediate from \eqref{eq_half_psi_vp} that these expressions satisfy the compatibility condition \eqref{eq_prob_bij_compatible}. Much less obvious is the fact that they are ``probabilities.''

\begin{thm}
\label{thm_r>0}
For $\mu \in \mc{D}(\la)$, we have
\[
\sum_{\nu \in \mc{U}(\la)} \mc{P}_\la(\mu \rightarrow \nu) = 1,
\]
and for $\nu \in \mc{U}(\la)$, we have
\[
\sum_{\mu \in \mc{D}^*(\la)} \ov{\mc{P}}_\la(\mu \leftarrow \nu) = 1.
\]
\end{thm}

We prove this result using Lagrange interpolation in \S\ref{sec_proofs}. Equation \eqref{eq_half_psi_vp} and Theorems \ref{thm_r=0} and \ref{thm_r>0} show that $\mc{P}_\la$ and $\ov{\mc{P}}_\la$ satisfy \eqref{eq_prob_bij_forward}, \eqref{eq_prob_bij_backward}, and \eqref{eq_prob_bij_compatible}, so they give a probabilistic bijection between the weighted sets $(\mc{D}^*(\la), \omega_\la)$ and $(\mc{U}(\la), \ov{\omega}_\la)$.

As in the case $\mu = \la$, it is instructive to consider what happens to these probabilities in the limit $q = t \rightarrow 1$. Suppose $\mu \in \mc{D}(\la)$ and $\nu \in \mc{U}(\la)$. Let $c_1$ be the cell in the intersection of the row containing the single cell $\nu/\la$ and the column containing the single cell $\la/\mu$, and let $c_2$ be the cell in the intersection of the column containing $\nu/\la$ and the row containing $\la/\mu$. It is clear that exactly one of $c_1$ and $c_2$ is in $\la$; call this cell $c_{\mu,\nu}$. (Note that if the cells $\nu/\la$ and $\la/\mu$ are in the same row or column, then $\la/\mu = \{c_{\mu,\nu}\}$.)

\begin{lem}
\label{lem_gamma_q=t=1}
If $\mu \in \mc{D}(\la)$ and $\nu \in \mc{U}(\la)$, then
\begin{equation}
\label{eq_gamma_q=t=1}
\lim_{q \rightarrow 1} \gamma_{\nu/\la/\mu}(q,q) = (h_\la(c_{\mu,\nu}))^2.
\end{equation}
Thus, we have
\[
\lim_{q \rightarrow 1} \mc{P}_\la(\mu \rightarrow \nu)|_{q,q} = \lim_{q \rightarrow 1} \ov{\mc{P}}_\la(\mu \leftarrow \nu)|_{q,q} = \dfrac{(H_\la)^2}{H_\mu H_\nu} \dfrac{1}{(h_\la(c_{\mu,\nu}))^2}.
\]
\end{lem}

\begin{proof}
It is clear that
\[
\lim_{q \rightarrow 1} \gamma_{\nu/\la/\mu}(q,q) = (n'(\la/\mu) - n'(\nu/\la) + n(\nu/\la) - n(\la/\mu))^2,
\]
so to prove \eqref{eq_gamma_q=t=1}, it suffices to show that
\begin{equation}
\label{eq_abs_value}
h_\la(c_{\mu,\nu}) = |n'(\la/\mu) - n'(\nu/\la) + n(\nu/\la) - n(\la/\mu)|.
\end{equation}
Let $(i_1,j_1)$ and $(i_2,j_2)$ be the (Cartesian) coordinates of the single cells $\nu/\la$ and $\la/\mu$, respectively. These coordinates can be expressed as
\[
\begin{array}{l}
(i_1,j_1) = (n'(\nu/\la)+1, n(\nu/\la)+1), \\
(i_2,j_2)= (n'(\la/\mu)+1, n(\la/\mu)+1).
\end{array}
\]
Observe that $c_{\mu,\nu} = (i_1,j_2)$ if $i_1 \leq i_2$, and $c_{\mu,\nu} = (i_2,j_1)$ if $i_1 > i_2$, so we have
\[
h_\la(c_{\mu,\nu}) = \begin{cases}
(i_2-i_1) + (j_1-j_2-1)+1 & \text{ if } i_1 \leq i_2 \\
(i_1-i_2-1) + (j_2-j_1) + 1 & \text{ if } i_i > i_2
\end{cases}
\]
(the reader may find it helpful to refer to Figure \ref{fig_parameters}, where $c_{\mu,\nu}$ is the cell containing two circles). In both cases \eqref{eq_abs_value} holds, so we have proved \eqref{eq_gamma_q=t=1}; the second assertion of the lemma now follows from Lemma \ref{lem_q=t=1}.
\end{proof}

Combining Theorem \ref{thm_r>0}, Lemmas \ref{lem_q=t=1} and \ref{lem_gamma_q=t=1}, and the hook-length formula, we obtain two interesting identities involving the numbers $f_\la$.

\begin{cor}
\label{cor_r>0_q=t=1}
Let $n = |\la|$. For $\mu \in \mc{D}(\la)$,
\[
\sum_{\nu \in \mc{U}(\la)} \dfrac{f_\mu f_\nu}{(h_\la(c_{\mu,\nu}))^2} = \dfrac{n+1}{n} (f_\la)^2,
\]
and for $\nu \in \mc{U}(\la)$,
\[
\dfrac{f_\la f_\nu}{n} + \sum_{\mu \in \mc{D}(\la)} \dfrac{f_\mu f_\nu}{(h_\la(c_{\mu,\nu}))^2} = \dfrac{n+1}{n} (f_\la)^2.
\]
\end{cor}

We have not been able to find these identities in the literature. Although they follow from our algebraic proof of Theorem \ref{thm_r>0}, they remain mysterious to us from the combinatorial and probabilistic points of view. We believe these identities---and their $q$-analogues discussed in Remark \ref{rem_ps_Schur}---deserve further study.

\subsection{Explicit formulas for the probabilities}
\label{sec_formulas}

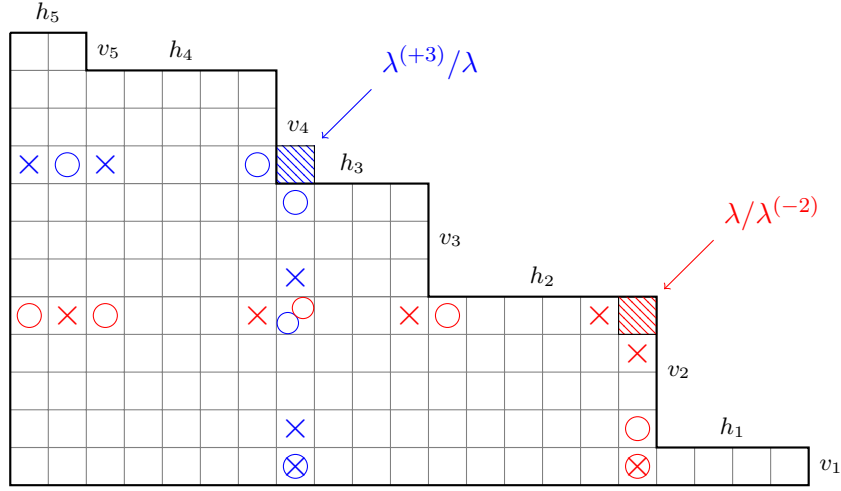
\begin{figure}
\begin{center}
\begin{tikzpicture}[scale=0.5]

\draw[gray, thin] (-5,-1) grid (-3,-2);
\draw[gray, thin] (-5,-2) grid (2,-5);
\draw[gray, thin] (-5,-5) grid (6,-8);
\draw[gray, thin] (-5,-8) grid (12,-12);
\draw[gray, thin] (-5,-12) grid (16,-13);

\draw[thick] (-5,-1) -- (-5,-13) -- (16,-13) --node[right]{\footnotesize $v_1$} (16,-12) --node[above]{\footnotesize $h_1$} (12,-12) --node[right]{\footnotesize $v_2$} (12,-8) --node[above]{\footnotesize $h_2$} (6,-8) --node[right]{\footnotesize $v_3$} (6,-5) --node[above]{\footnotesize $h_3$} (2,-5) --node[right]{\footnotesize $v_4$} (2,-2) --node[above]{\footnotesize $h_4$} (-3,-2) --node[right]{\footnotesize $v_5$}  (-3,-1) --node[above]{\footnotesize $h_5$} (-5,-1);

\draw[pattern=north west lines, pattern color=blue] (2,-5) rectangle (3,-4);
\draw[pattern=north west lines, pattern color=red] (11,-9) rectangle (12,-8);

\draw[color = blue, ->] (4.5,-2.5) node[above right]{$\la^{(+3)}/\la$} -- (3.2,-3.8);
\draw[color = red, ->] (13.5,-6.5) node[above right]{$\la/\la^{(-2)}$} -- (12.2,-7.8);

\draw[color = blue] (1.5,-4.5) circle[radius = 9pt];
\draw[color = blue] (-2.5,-4.5) node{\Large{$\times$}};
\draw[color = blue] (-3.5,-4.5) circle[radius = 9pt];
\draw[color = blue] (-4.5,-4.5) node{\Large{$\times$}};

\draw[color = blue] (2.5,-5.5) circle[radius=9pt];
\draw[color = blue] (2.5,-7.5) node{\Large{$\times$}};
\draw[color = blue] (2.3,-8.7) circle[radius=8pt];
\draw[color = blue] (2.5,-11.5) node{\Large{$\times$}};
\draw[color = blue] (2.5,-12.5) circle[radius=9pt] node{\Large{$\times$}};

\draw[color = red] (10.5,-8.5) node{\Large{$\times$}};;
\draw[color = red] (6.5,-8.5) circle[radius=9pt];
\draw[color = red] (5.5,-8.5) node{\Large{$\times$}};;
\draw[color = red] (2.7,-8.3) circle[radius=8pt];
\draw[color = red] (1.5,-8.5) node{\Large{$\times$}};;
\draw[color = red] (-2.5,-8.5) circle[radius=9pt];
\draw[color = red] (-3.5,-8.5) node{\Large{$\times$}};;
\draw[color = red] (-4.5,-8.5) circle[radius=9pt];

\draw[color = red] (11.5,-9.5) node{\Large{$\times$}};
\draw[color = red] (11.5,-11.5) circle[radius=9pt];
\draw[color = red] (11.5,-12.5) circle[radius=9pt] node{\Large{$\times$}};

\end{tikzpicture}
\end{center}

\caption{The Young diagram of the partition $\la = (21,17^4,11^3,7^3,2)$, whose parameters are $h = (4,6,4,5,2)$ and $v = (1,4,3,3,1)$. The partition $\nu = \la^{(+3)}$ is formed by adding the shaded blue box to $\la$, and $\mu = \la^{(-2)}$ is formed by removing the shaded red box from $\la$. The proof of Proposition \ref{prop_explicit_formulas} shows that the cells containing a blue circle (resp., a blue $\times$) contribute a factor $1-q^it^j$ to the numerator (resp., denominator) of $\alpha_3$ and $\ov{\alpha}_3$. Similarly, the cells containing a red circle (resp., a red $\times$) contribute a factor $1-q^it^j$ to the numerator (resp., denominator) of $\beta_2$ and $\ov{\beta}_2$. The cell containing one circle of each color is $c_{\mu,\nu}$; the terms coming from this cell are cancelled by the terms in the numerator of $\gamma_{r,s}'$. (The terms in the denominator of $\gamma_{r,s}'$ cancel with the terms coming from the two red $\times$'s adjacent to the shaded red box.)}
\label{fig_parameters}
\end{figure}

In this section, we give more explicit formulas for the probabilities $\mc{P}_\la(\mu \rightarrow \nu)$ and $\ov{\mc{P}}_\la(\mu \leftarrow \nu)$ in terms of a set of parameters associated to the partition $\la$. These formulas provide further insight into the probabilities, and they lead to an even more compact formulation that is essential for the proofs of Theorems \ref{thm_r=0} and \ref{thm_r>0}.

Suppose $\la$ is a partition with $d$ distinct part sizes $u_1 > \cdots > u_d > 0$. Define positive integers $h_1, \ldots, h_d$ and $v_1, \ldots, v_d$ by taking $v_i$ to be the multiplicity of the part $u_i$ in $\la$, and setting $h_i = u_i - u_{i+1}$ (with $u_{d+1} = 0$). The $h_i$ (resp., $v_i$) give the lengths of the horizontal (resp., vertical) segments of the boundary of the Young diagram of $\la$, starting from the southeast, as illustrated in Figure \ref{fig_parameters}. We call the vectors $h = (h_1, \ldots, h_d)$ and $v = (v_1, \ldots, v_d)$ the \emph{parameters} of $\la$. Set
\[
h_{i,j} = h_i + \cdots + h_j, \quad\quad v_{i,j} = v_i + \cdots + v_j
\]
for $i \leq j$, and $h_{i,j} = v_{i,j} = 0$ for $i > j$.

The set $\mc{U}(\la)$ consists of $d+1$ elements, which we denote by $\la^{(+0)}, \ldots, \la^{(+d)}$, where $\la^{(+s)}$ is obtained from $\la$ by adding a box in row $v_{1,s} + 1$. The set $\mc{D}^*(\la)$ also consists of $d+1$ elements, which we denote by $\la^{(-0)}, \ldots, \la^{(-d)}$, where $\la^{(-0)} = \la$, and for $r \geq 1$, $\la^{(-r)}$ is obtained from $\la$ by removing a box from row $v_{1,r}$. This is illustrated in Figure \ref{fig_parameters}.

For $r,s \in \{0,\ldots,d\}$, set
\[
p_{r,s}(q,t) = \mc{P}_\la(\la^{(-r)} \rightarrow \la^{(+s)}), \quad\quad \ov{p}_{r,s}(q,t) = \ov{\mc{P}}_\la(\la^{(-r)} \leftarrow \la^{(+s)}).
\]
For $s \in \{0,\ldots,d\}$ and $r \in \{1,\ldots,d\}$, define $\alpha_s(q,t), \ov{\alpha}_s(q,t), \beta_r(q,t), \ov{\beta}_r(q,t), \gamma_{r,s}(q,t)$ by
\[
\alpha_s(q,t) = \alpha_{\la^{(+s)}/\la}(q,t), \quad \beta_r(q,t) = \beta_{\la/\la^{(-r)}}(q,t), \quad \gamma_{r,s}(q,t) = \gamma_{\la^{(+s)}/\la/\la^{(-r)}}(q,t), \quad \text{ etc.}
\]
Often the partition $\la$ will be clear from context; if not, we will write $p_{r,s}[\la](q,t), \alpha_s[\la](q,t)$, etc. We will frequently omit the dependence of these expressions on $q$ and $t$ when we are not specializing either variable.

\begin{prop}
\label{prop_explicit_formulas}
Suppose $\la$ has $d$ distinct part sizes and parameters $(h_1, \ldots, h_d)$ and $(v_1, \ldots, v_d)$. Recall that $[i,j] = 1-q^it^j$, and set $[i,j]^+ = [i+1,j-1], [i,j]^- = [i-1,j+1]$. We have
\[
\alpha_s = \prod_{i=1}^s \dfrac{[h_{i,s},v_{i+1,s}]}{[h_{i,s},v_{i,s}]} \prod_{i=s+1}^d \dfrac{[h_{s+1,i-1},v_{s+1,i}]}{[h_{s+1,i},v_{s+1,i}]},
\]
\[
\ov{\alpha}_s = \prod_{i=1}^s \dfrac{[h_{i,s},v_{i+1,s}]^-}{[h_{i,s},v_{i,s}]^-} \prod_{i=s+1}^d \dfrac{[h_{s+1,i-1},v_{s+1,i}]^+}{[h_{s+1,i},v_{s+1,i}]^+},
\]
\[
\beta_r = \prod_{i=1}^{r-1} \dfrac{[h_{i,r-1},v_{i,r}]^+}{[h_{i,r-1},v_{i+1,r}]^+} \times \dfrac{[0,v_r]^+}{[0,1]^+} \dfrac{[h_r,0]^-}{[1,0]^-} \times \prod_{i=r+1}^d \dfrac{[h_{r,i},v_{r+1,i}]^-}{[h_{r,i-1},v_{r+1,i}]^-},
\]
\[
\ov{\beta}_r = \prod_{i=1}^{r-1} \dfrac{[h_{i,r-1},v_{i,r}]}{[h_{i,r-1},v_{i+1,r}]} \times \dfrac{[0,v_r]}{[0,1]} \dfrac{[h_{r},0]}{[1,0]} \times \prod_{i=r+1}^d \dfrac{[h_{r,i},v_{r+1,i}]}{[h_{r,i-1},v_{r+1,i}]},
\]
\[
\gamma_{r,s} = \begin{cases}
\gamma'_{r,s} & \text{ if } 0 < r \leq s \medskip \\
\gamma'_{r,s} \cdot q^{-1-2h_{s+1,r-1}} \cdot t^{1-2v_{s+1,r}}  & \text{ if } r > s,
\end{cases}
\]
where
\[
\gamma'_{r,s} = \begin{cases}
\dfrac{[h_{r,s},v_{r+1,s}][h_{r,s},v_{r+1,s}]^-}{[0,1][1,0]} & \text{ if } 0 < r \leq s \medskip \\
\dfrac{[h_{s+1,r-1},v_{s+1,r}][h_{s+1,r-1},v_{s+1,r}]^+}{[0,1][1,0]} & \text{ if } r > s.
\end{cases}
\]
Thus, the probabilities can be written as
\begin{equation}
\label{eq_row_probs}
p_{r,s}(q,t) = \begin{cases}
\ds t^{v_{1,s}} \alpha_s & \text{ if } r = 0 \bigskip \\
\ds \tau_{r,s} \dfrac{\alpha_s \beta_r}{\gamma'_{r,s}} & \text{ if } r > 0,
\end{cases}
\quad\quad\quad
\ov{p}_{r,s}(q,t) = \begin{cases}
\ds t^{v_{1,s}} \ov{\alpha}_s & \text{ if } r = 0 \bigskip \\
\ds \tau_{r,s} \dfrac{\ov{\alpha}_s \ov{\beta}_r}{\gamma'_{r,s}} & \text{ if } r > 0,
\end{cases}
\end{equation}
where
\[
\tau_{r,s} = \begin{cases}
t^{v_{r+1,s}} & \text{ if } 0 < r \leq s \\
q^{1 + 2h_{s+1,r-1}} \cdot t^{-1+v_{s+1,r}} & \text{ if } r > s.
\end{cases}
\]
\end{prop}

Before giving the proof, we make several observations and give an example.

\begin{rem} \
\label{rem_row_probs}
\begin{enumerate}
\item
Since the parameters $h_i$ and $v_i$ are strictly positive, one can easily verify that each of the factors $1-q^it^j$ appearing in \eqref{eq_row_probs} satisfies $i,j \geq 0$ and $i+j > 0$, and that the exponents of $q$ and $t$ in $\tau_{r,s}$ are nonnegative. This makes it clear that $\mc{P}_\la(\mu \rightarrow \nu)$ and $\ov{\mc{P}}_\la(\mu \leftarrow \nu)$ are nonnegative whenever $q,t \in [0,1)$ or $q,t \in (1,\infty)$ (in the latter case, we also use the fact that the expressions for $\alpha_s, \ov{\alpha}_s, \beta_r, \ov{\beta}_r, \gamma'_{r,s}$ each contain the same number of factors in their numerator and denominator). Since the expressions $\mc{P}_\la(\mu \rightarrow \nu)$ and $\ov{\mc{P}}_\la(\mu \leftarrow \nu)$ sum to 1, they are honest probabilities in such specializations.

\item
Each of the factors $1-q^it^j$ appearing in the formulas \eqref{eq_row_probs} is associated with a particular cell of $\mc{R}_{\nu/\la}, \mc{C}_{\nu/\la}, \mc{R}_{\la/\mu},$ or $\mc{C}_{\la/\mu}$, where $\nu = \la^{(+s)}$ and $\mu = \la^{(-r)}$. These cells are marked in Figure \ref{fig_parameters}. The proof of Proposition \ref{prop_explicit_formulas} consists in verifying that all the other terms in the products defining $\alpha_{\nu/\la}, \beta_{\la/\mu}$, etc. cancel out.
\end{enumerate}
\end{rem}

\begin{ex}
\label{ex_probs_d=1}
If $\la$ has one distinct part size, then $\la = (h_1^{v_1})$ is a rectangle. Writing $h = h_1, v = v_1$, we have
\[
\alpha_0 = \dfrac{1-t^v}{1-q^ht^v}, \quad\quad \alpha_1 = \dfrac{1-q^h}{1-q^ht^v}, \quad\quad \beta_1 = \dfrac{(1-qt^{v-1})(1-q^{h-1}t)}{(1-q)(1-t)},
\]
and the probabilities $p_{r,s}$ are:
\[
\begin{array}{ll}
p_{0,0} = \dfrac{1-t^v}{1-q^ht^v} & \quad\quad p_{0,1} = t^v\dfrac{1-q^h}{1-q^ht^v} \bigskip \\
p_{1,0} = qt^{v-1}\dfrac{1-q^{h-1}t}{1-q^ht^v} & \quad\quad p_{1,1} = \dfrac{1-qt^{v-1}}{1-q^ht^v}.
\end{array}
\]
Similarly, the probabilities $\ov{p}_{r,s}$ are:
\[
\begin{array}{ll}
\ov{p}_{0,0} = \dfrac{1-qt^{v-1}}{1-q^{h+1}t^{v-1}} & \quad\quad \ov{p}_{0,1} = t^v\dfrac{1-q^{h-1}t}{1-q^{h-1}t^{v+1}} \bigskip \\
\ov{p}_{1,0} = qt^{v-1}\dfrac{1-q^h}{1-q^{h+1}t^{v-1}} & \quad\quad \ov{p}_{1,1} = \dfrac{1-t^v}{1-q^{h-1}t^{v+1}}.
\end{array}
\]
\end{ex}

\begin{proof}[Proof of Proposition \ref{prop_explicit_formulas}]
Let $\nu = \la^{(+s)}$. By definition,
\begin{equation}
\label{eq_alpha_def}
\alpha_s = \alpha_{\nu/\la} = \prod_{c \in \mc{R}_{\nu/\la}} \dfrac{[a_\la(c),\ell_\la(c)+1]}{[a_\nu(c),\ell_\nu(c)+1]} \prod_{c \in \mc{C}_{\nu/\la}} \dfrac{[a_\la(c)+1,\ell_\la(c)]}{[a_\nu(c)+1,\ell_\nu(c)]}.
\end{equation}
Write $(x,y) = (h_{s+1,d}+1, v_{1,s}+1)$ for the coordinates of the cell $\nu/\la$. For $i = 1, \ldots, s$, the contribution to the second product in \eqref{eq_alpha_def} from the cells $(x,v_{1,i-1}+1), \ldots, (x,v_{1,i})$ is
\[
\dfrac{[h_{i,s},v_{i,s}-1]}{[h_{i,s},v_{i,s}]} \dfrac{[h_{i,s},v_{i,s}-2]}{[h_{i,s},v_{i,s}-1]} \cdots \dfrac{[h_{i,s},v_{i,s}-v_i]}{[h_{i,s},v_{i,s}-v_i+1]} = \dfrac{[h_{i,s},v_{i+1,s}]}{[h_{i,s},v_{i,s}]}.
\]
For $i = s+1, \ldots, d$, the contribution to the first product from the cells $(h_{i+1,d}+1,y), \ldots, (h_{i,d},y)$ is
\[
\dfrac{[h_{s+1,i}-1,v_{s+1,i}]}{[h_{s+1,i},v_{s+1,i}]} \dfrac{[h_{s+1,i}-2,v_{s+1,i}]}{[h_{s+1,i}-1,v_{s+1,i}]} \cdots \dfrac{[h_{s+1,i}-h_i,v_{s+1,i}]}{[h_{s+1,i}-h_i+1,v_{s+1,i}]} = \dfrac{[h_{s+1,i-1},v_{s+1,i}]}{[h_{s+1,i},v_{s+1,i}]}.
\]
This proves the formula for $\alpha_s$. The formulas for $\ov{\alpha}_s, \beta_r$, and $\ov{\beta}_r$ are proved in the same way, using $\mu = \la^{(-r)}$ instead of $\nu = \la^{(+s)}$ for the latter two; we omit the details.
The formula for $\gamma_{r,s}$ comes from the observation that
\[
\begin{array}{cc}
n'(\nu/\la) = h_{s+1,d} & \quad n'(\la/\mu) = h_{r,d}-1 \\
n(\nu/\la) = v_{1,s} & \quad n(\la/\mu) = v_{1,r}-1
\end{array}
\]
together with (in the case $r > s$) the identity $[-i,-j] = -q^{-i}t^{-j}[i,j]$.
\end{proof}

One immediate consequence of the formulas \eqref{eq_row_probs} and Remark \ref{rem_row_probs}(1) is that
\[
p_{r,s}(0,0) = \ov{p}_{r,s}(0,0) = \delta_{r,s}
\]
(the Kronecker delta). That is, when $q = t = 0$, the probabilistic bijection $\mc{P}_\la, \ov{\mc{P}}_\la$ reduces to the bijection $F_\la^\row$ (see \S \ref{sec: Up and Down}), which gives rise to the growth rules for ordinary Robinson--Schensted. This is consistent with the fact that setting $q = t = 0$ turns the Macdonald polynomials into the Schur polynomials. Note, however, that although setting $q = t$ is enough to recover the Schur polynomials, this specialization does not ``trivialize'' the probabilities! Instead, we obtain a one-parameter family of probabilistic bijections between the (trivially) weighted sets $(\mc{D}^*(\la), 1)$ and $(\mc{U}(\la),1)$, which contains the bijection $F_\la^\row$ at one extreme. The $q = t$ specialization is further examined in \S \ref{sec_q=t_spec}.

\begin{rem}
\label{rem_uniqueness}
The space of probabilistic bijections from $(\mc{D}^*(\la), \omega_\la)$ to $(\mc{U}(\la), \ov{\omega}_\la)$ is quite large in general. We have seen that our particular choice $\mc{P}_\la, \ov{\mc{P}}_\la$ has two very nice properties: it takes nonnegative values when $q,t \in [0,1)$ or $q,t \in (1, \infty)$ (Remark \ref{rem_row_probs}(1)), and it reduces to the row insertion bijection $F_\la^{\row}$ when $q = t = 0$. In addition, $\mc{P}_\la, \ov{\mc{P}}_\la$ reduces to $F_\la^{\col}$ when $q = t \rightarrow \infty$ (this is shown in \S \ref{sec_column_insertion}). It seems possible that $\mc{P}_\la, \ov{\mc{P}}_\la$ is the unique \emph{rational} probabilistic bijection (i.e., with ``probabilities'' in $\bb{R}(q,t)$) with these properties. Perhaps it is even the unique rational probabilistic bijection which takes nonnegative values for $q,t \in [0,1)$, and reduces to $F_\la^\row$ when $q = t = 0$.
\end{rem}

\subsection{Proofs of Theorems \ref{thm_r=0} and \ref{thm_r>0}}
\label{sec_proofs}

For the proofs of Theorem \ref{thm_r=0} and Theorem \ref{thm_r>0}, we first rewrite the probabilities $p_{r,s}$ and $\ov{p}_{r,s}$ in a more compact form.

\begin{prop}
\label{prop: 3rd version of probabilities}
Let $\la$ be a partition with parameters $(h_1,\ldots,h_d)$ and $(v_1,\ldots,v_d)$. Set $x_i = h_{1,i}$ and $y_i =v_{1,i}$ for $0 \leq i \leq d$ (in particular, $x_0 = y_0 = 0$). We have
\begin{align}
\label{eq: prob r=0 with x,y}
p_{0,s} = \frac{\prod\limits_{i=1}^d (q^{x_s}t^{y_s}-q^{x_{i-1}}t^{y_i})}{\prod\limits_{\substack{i=0 \\ i \neq s}}^d (q^{x_s}t^{y_s}-q^{x_i}t^{y_i})}, \qquad \qquad
\ov{p}_{0,s} = \frac{\prod\limits_{i=1}^d (q^{x_s-1}t^{y_s+1}- q^{x_{i-1}}t^{y_i})}{\prod\limits_{\substack{i=0 \\ i \neq s}}^d (q^{x_s-1}t^{y_s+1}- q^{x_i}t^{y_i})},
\end{align}
and for $r>0$,
\begin{align}
\label{eq: prob r>0 with x,y 1}
p_{r,s}&=\prod_{\substack{i=0 \\ i \neq s}}^d \frac{q^{x_{r-1}+1}t^{y_r-1}-q^{x_i}t^{y_i}}{q^{x_s}t^{y_s}-q^{x_i}t^{y_i}}
\prod_{\substack{i=1 \\ i \neq r}}^d \frac{q^{x_s}t^{y_s}-q^{x_{i-1}}t^{y_i}}{q^{x_{r-1}+1}t^{y_r-1}-q^{x_{i-1}}t^{y_i}} ,\\
\label{eq: prob r>0 with x,y 2}
\ov{p}_{r,s}&= \prod_{\substack{i=0 \\ i \neq s }}^d \frac{q^{x_{r-1}}t^{y_r}- q^{x_i}t^{y_i}}{q^{x_s-1}t^{y_s+1}-q^{x_i}t^{y_i}}
\prod_{\substack{i=1 \\ i \neq r}}^d \frac{q^{x_s-1}t^{y_s+1}-q^{x_{i-1}}t^{y_i}}{q^{x_{r-1}}t^{y_r}-q^{x_{i-1}}t^{y_i}}.
\end{align}
\end{prop}
\begin{proof}
The formulas follow directly from the explicit expressions for the probabilities $p_{r,s}, \ov{p}_{r,s}$ in Proposition \ref{prop_explicit_formulas}. We explain this for $p_{0,s}$ in more detail; the other formulas are obtained analogously. Using the definition of $x_i,y_i$ and the identity $[i,j]=-q^it^j[-i,-j]$, we obtain
\begin{multline*}
p_{0,s} = t^{y_s} \prod_{i=1}^s \frac{[x_s-x_{i-1},y_s-y_i]}{[x_s-x_{i-1},y_s-y_{i-1}]}
\prod_{i=s+1}^d \frac{[x_{i-1}-x_s,y_i-y_s]}{[x_i-x_s,y_i-y_s]} \\
= t^{y_s} \prod_{i=1}^s \frac{q^{x_s-x_{i-1}}t^{y_s-y_i}[x_{i-1}-x_s,y_i-y_s]}{q^{x_s-x_{i-1}}t^{y_s-y_{i-1}}[x_{i-1}-x_s,y_{i-1}-y_s]}
\prod_{i=s+1}^d \frac{[x_{i-1}-x_s,y_i-y_s]}{[x_i-x_s,y_i-y_s]}\\
= \frac{\prod\limits_{i=1}^d \left(1-q^{x_{i-1}-x_s}t^{y_i-y_s}\right)}{\prod\limits_{\substack{i=0 \\ i \neq s}}^d \left(1-q^{x_i-x_s}t^{y_i-y_s}\right)}
= \frac{\prod\limits_{i=1}^d \left(q^{x_s}t^{y_s}-q^{x_{i-1}}t^{y_i}\right)}{\prod\limits_{\substack{i=0 \\ i \neq s}}^d \left(q^{x_s}t^{y_s}-q^{x_i}t^{y_i}\right)}. \qedhere
\end{multline*}
\end{proof}

The main idea of our proof is to use Lagrange interpolation. Consider a function $f: k \rightarrow k$, where $k$ is a field. Given pairwise distinct elements $a_0, \ldots, a_d \in k$, define the \emph{interpolation polynomial} of $f$ at the positions $a_0, \ldots, a_d$ by
\[
I_f^{a_0, \ldots, a_d}(x) = \sum_{s=0}^d f(a_s) \prod_{\substack{i=0 \\ i \neq s}}^d \frac{x-a_i}{a_s-a_i}.
\]
\begin{lem}
\label{lem: interpolation poly}
If $f$ is a polynomial of degree at most $d$, then $I_f^{a_0, \ldots, a_d}(x) = f(x)$.
\end{lem}
\begin{proof}
Both $f$ and $I^{a_0, \ldots, a_d}_f$ are polynomials of degree at most $d$ which agree at the $d+1$ positions $a_0, \ldots, a_d$, so they are equal.
\end{proof}

\begin{proof}[Proof of Theorem \ref{thm_r=0}]
Writing $a_i=q^{x_i}t^{y_i}$ and  $b_i=q^{x_{i-1}}t^{y_i}$, we have
\begin{equation}
\label{eq: p_0,s with a,b}
p_{0,s}= \frac{\prod\limits_{i=1}^d (a_s-b_i)}{\prod\limits_{\substack{i=0 \\ i \neq s}}^d (a_s-a_i)}.
\end{equation}
(Note that since $h_i, v_i > 0$, the $x_i$ and $y_i$, and thus the $a_i$, are distinct.) The sum $\sum_{s =0}^d p_{0,s}$ is the leading coefficient of the interpolation polynomial
\[
I^{a_0, \ldots, a_d}_f(x) = \sum_{s=0}^d \prod_{i=1}^d (a_s-b_i) \prod_{\substack{i=0 \\i \neq s}}^d\frac{x-a_i}{a_s-a_i}
\]
for $f(x)=\prod_{i=1}^d(x-b_i)$, and hence by Lemma \ref{lem: interpolation poly} equal to $1$.
\end{proof}

\begin{proof}[Proof of Theorem \ref{thm_r>0}]
Fix $r > 0$. Set $a_i=q^{x_i}t^{y_i}$ for $0 \leq i \leq d$ and $b_i=q^{x_{i-1}}t^{y_i}$ for $1 \leq i \leq d$, with the exception $b_r=q^{x_{r-1}+1}t^{y_{r}-1}$. This allows us to rewrite $p_{r,s}$ as
\begin{equation}
\label{eq: p_r,s with a,b}
p_{r,s}= \prod_{\substack{i=0 \\i \neq s}}^d \frac{b_r-a_i}{a_s-a_i} \prod_{\substack{i=1 \\ i\neq r}}^d\frac{a_s-b_i}{b_r-b_i}.
\end{equation}
Interpolation of the polynomial $f(x)=\prod\limits_{\substack{i=1 \\ i \neq r}}^d (x-b_i)$ at the positions $a_0, \ldots, a_d$ gives the identity
\[
\prod\limits_{\substack{i=1 \\ i \neq r}}^d (x-b_i) =
\sum_{s=0}^d \prod_{\substack{i=1 \\ i \neq r}}^d(a_s-b_i) \prod_{\substack{i=0 \\ i \neq s}}^d\frac{x-a_i}{a_s-a_i}.
\]
Dividing both sides by $f(x)$ and setting $x = b_r$, we obtain $1 = \sum_{s=0}^d p_{r,s}$. (The positivity of $h_i$ and $v_i$ implies that $b_r \neq b_i$ for $i \neq r$.)
\\

Now fix $s \geq 0$. Set $a_i=q^{x_i}t^{y_i}$ for $0 \leq i \leq d$, with the exception $a_s=q^{x_{s}-1}t^{y_{s}+1}$, and set $b_i=q^{x_{i-1}}t^{y_i}$ for $1 \leq i \leq d$. This allows us to rewrite $\ov{p}_{r,s}$ as
\begin{equation}
\label{eq: ov_p_r,s with a,b}
\ov{p}_{0,s}= \frac{\prod\limits_{i=1}^d (a_s-b_i)}{\prod\limits_{\substack{i=0 \\ i \neq s}}^d (a_s-a_i)}, \qquad \ov{p}_{r,s}= \prod_{\substack{i=0 \\i \neq s}}^d \frac{b_r-a_i}{a_s-a_i} \prod_{\substack{i=1 \\ i\neq r}}^d\frac{a_s-b_i}{b_r-b_i} \quad \text{ for } r > 0.
\end{equation}
Consider the interpolation polynomial $I^{b_1, \ldots, b_d}_f$ for the degree $d-1$ polynomial
\[
f(x)= \prod_{\substack{i=0 \\ i \neq s}}^d (x-a_i) - \prod_{i=1}^d (x-b_i).
\]
Lemma \ref{lem: interpolation poly} gives
\[
\prod_{\substack{i=0 \\ i \neq s}}^d (x-a_i) - \prod_{i=1}^d (x-b_i)
= \sum_{r=1}^d \left( \prod_{\substack{i=0 \\ i \neq s}}^d (b_r-a_i) - \prod_{i=1}^d (b_r-b_i) \right)
\prod_{\substack{i=1 \\ i \neq r}}^d \frac{x-b_i}{b_r-b_i}.
\]
Setting $x= a_s$, adding $\prod_{i=1}^d (a_s-b_i)$ to both sides, and then dividing by $\prod\limits_{\substack{i=0 \\ i \neq s}}^d (a_s-a_i)$, we obtain
\[
1= \sum_{r=1}^d \prod_{\substack{i=0 \\ i \neq s}}^d \frac{b_r-a_i}{a_s-a_i} \prod_{\substack{i=1 \\ i \neq r}}^d \frac{a_s-b_i}{b_r-b_i} + \frac{\prod\limits_{i=1}^d (a_s-b_i)}{\prod\limits_{\substack{i=0 \\ i \neq s}}^d (a_s-a_i)} = \sum_{r=0}^d \ov{p}_{r,s}. \qedhere
\]
\end{proof}

\begin{rem}
Remarkably, the expressions for $p_{r,s}$ and $\ov{p}_{r,s}$ in \eqref{eq: p_0,s with a,b}, \eqref{eq: p_r,s with a,b}, and \eqref{eq: ov_p_r,s with a,b} coincide with the expressions obtained by setting $a_i = x_i + y_i$ and $b_i = x_{i-1} + y_i$ in the expressions $\lim_{q \rightarrow 1} p_{r,s}(q,q)$ and $\lim_{q \rightarrow 1} \ov{p}_{r,s}(q,q)$. Thus, proving Theorems \ref{thm_r=0} and \ref{thm_r>0} for arbitrary $q,t$ is equivalent to proving them in the limit $q = t \rightarrow 1$.
\end{rem}

\subsection{\qrst: Definition and examples}
\label{sec_qrst}

In \S \ref{sec_weighted_sets}-\ref{sec_proofs}, we proved the commutation relation for the $(q,t)$-up and down operators by introducing a family of probabilistic bijections $\mc{P}_\la, \ov{\mc{P}}_\la$ between the weighted sets $\mc{D}^*(\la)$ and $\mc{U}(\la)$. In this section, we introduce the probabilistic insertion algorithm $\qrst$ by interpreting the $\mc{P}_\la$ as (probabilistic) local growth rules.

We freely use the definitions and notation from \S \ref{sec: Fomin growth diagrams}. Let $\Lambda$ be a growth, and suppose $\square \in \Lambda$ is a square in the $n \times n$ grid whose vertices are labeled by $\Lambda$. The square $\square$ has one of the following three types of configurations of vertices:\footnote{Here we ignore the entry of the permutation matrix in the middle of the square, which allows us to merge the third and fourth configurations that appeared in \S \ref{sec: Fomin growth diagrams} into a single type.}

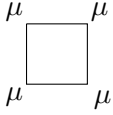
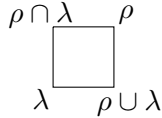
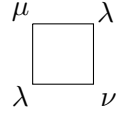
\begin{center}
\begin{tikzpicture}[scale=0.8]
\node at (0.5,1.2) {\textbf{Type I}};
\draw (0,0) -- (0,-1) -- (1,-1) -- (1,0) -- (0,0);
\node at (-0.2,0.2) {$\mu$};
\node at (1.2,0.2) {$\mu$};
\node at (-0.2,-1.2) {$\mu$};
\node at (1.25,-1.25) {$\mu$};

\begin{scope}[xshift=5.5cm]
\node at (0.5,1.2) {\textbf{Type II}};
\draw (0,0) -- (0,-1) -- (1,-1) -- (1,0) -- (0,0);
\node at (-0.2,0.2) {$\rho \cap \la$};
\node at (1.2,0.2) {$\rho$};
\node at (-0.2,-1.2) {$\lambda$};
\node at (1.25,-1.25) {$\rho \cup \la$};
\node at (.5,-2.2) {with $\lambda \neq \rho$};
\end{scope}

\begin{scope}[xshift=11cm]
\node at (0.5,1.2) {\textbf{Type III}};
\draw (0,0) -- (0,-1) -- (1,-1) -- (1,0) -- (0,0);
\node at (-0.2,0.2) {$\mu$};
\node at (1.2,0.2) {$\lambda$};
\node at (-0.2,-1.2) {$\lambda$};
\node at (1.25,-1.25) {$\nu$};
\node at (.5,-2.2) {with $\mu \in \D^*(\la), \nu \in \U(\la)$};
\end{scope}
\end{tikzpicture}
\end{center}

\noindent Set
\[
\mc{P}(\Lambda) = \prod_{\square \in \Lambda} \mc{P}(\square), \quad\quad \ov{\mc{P}}(\Lambda) = \prod_{\square \in \Lambda} \ov{\mc{P}}(\square),
\]
where
\[
\mc{P}(\square) = \begin{cases}
\mc{P}_\la(\mu \rightarrow \nu) & \text{ if } \square \text{ is of type III} \\
1 & \text{ otherwise,}
\end{cases}
\quad\quad\quad
\ov{\mc{P}}(\square) = \begin{cases}
\ov{\mc{P}}_\la(\mu \leftarrow \nu) & \text{ if } \square \text{ is of type III} \\
1 & \text{ otherwise.}
\end{cases}
\]

\begin{defn}
Suppose $\sigma \in S_n$ and $P,Q \in \SYT(\la)$ for some $\la \vdash n$. Define
\[
\mc{P}(\sigma \rightarrow P,Q) = \sum_{\Lambda : \sigma \rightarrow P,Q} \mc{P}(\Lambda), \quad\quad \ov{\mc{P}}(\sigma \leftarrow P,Q) = \sum_{\Lambda : \sigma \rightarrow P,Q} \ov{\mc{P}}(\Lambda).
\]
\end{defn}

\begin{thm}
\label{thm_qrst_is_prob_bij}
The expressions $\mc{P}(\sigma \rightarrow P,Q)$ and $\ov{\mc{P}}(\sigma \leftarrow P,Q)$ define a probabilistic bijection between the weighted sets $(S_n, \omega)$ and $\ds \left(\bigsqcup_{\la \vdash n} \SYT(\la) \times \SYT(\la), \ov{\omega}\right)$, where
\[
\omega(\sigma) = \dfrac{(1-t)^n}{(1-q)^n} \quad\quad \text{ and } \quad\quad \ov{\omega}(P,Q) = \psi_P(q,t)\vp_Q(q,t).
\]
Moreover, $\mc{P}(\sigma \rightarrow P,Q)$ and $\ov{\mc{P}}(\sigma \leftarrow P,Q)$ take values in $[0,1]$ when $q,t \in [0,1)$ or $q,t \in (1,\infty)$.
\end{thm}

\begin{proof}
For a fixed permutation $\sigma$, one can obtain all the growths associated with $\sigma$ by starting with the empty partition along the north and west boundaries, and recursively filling in the rest of the diagram. When the northwest, southwest, and northeast vertices ($\mu,\la,\rho$, respectively) of a square have been filled in, there is either a unique choice for the southeast vertex $\nu$ (for type I and II squares), or $\nu$ is chosen according to the probability distribution $\mc{P}_\la(\mu \rightarrow \nu)$ (for type III squares). Thus, the fact that the expressions $\mc{P}(\sigma \rightarrow P,Q)$ sum to 1 for fixed $\sigma$ follows from the fact that the ``local probabilities'' $\mc{P}_\la(\mu \rightarrow \nu)$ sum to $1$ for fixed $\mu$ and $\la$. Also, since the ``local probabilities'' take values in $[0,1]$ for $q,t \in [0,1)$ or $q,t \in (1, \infty)$ by Remark \ref{rem_row_probs}(1), the same is true of $\mc{P}(\sigma \rightarrow P,Q)$.

Similarly, by starting with a fixed pair $P,Q$ along the east and south boundaries and recursively filling in the rest of the growth according to the ``local backward probabilities'' $\ov{\mc{P}}_\la(\mu \leftarrow \nu)$, one sees that the expressions $\ov{\mc{P}}(\sigma \leftarrow P,Q)$ sum to 1 for fixed $P,Q$, and take values in $[0,1]$ for the appropriate values of $q,t$.

Finally, the compatibility relation $\omega(\sigma) \mc{P}(\sigma \rightarrow P,Q) = \ov{\mc{P}}(\sigma \leftarrow P,Q) \ov{\omega}(P,Q)$ follows from Lemma \ref{lem_Lambda_compatibility} below.
\end{proof}

\begin{lem}
\label{lem_Lambda_compatibility}
If $\Lambda$ is a growth from $\sigma$ to $P,Q$, then
\[
\dfrac{(1-t)^n}{(1-q)^n} \mc{P}(\Lambda) = \ov{\mc{P}}(\Lambda) \psi_P(q,t) \vp_Q(q,t).
\]
\end{lem}

\begin{proof}
Let $L$ be a lattice path in $\Lambda$ from the northeast corner $(0,n)$ to the southwest corner $(n,0)$ consisting of unit steps to the south or west. Let $\omega(L)$ be the product of the weights of the edges in $L$, where a vertical edge $\mu \atop {| \atop \la}$ has weight $\psi_{\la/\mu}(q,t)$, and a horizontal edge $\mu - \rho$ has weight $\vp_{\rho/\mu}(q,t)$ (by definition, $\psi_{\la/\la} = \vp_{\la/\la} = 1$). Observe that the lattice path consisting of $n$ south steps followed by $n$ west steps has weight $\psi_P(q,t) \vp_Q(q,t)$.

Let $\nw(L)$ be the set of squares of $\Lambda$ between $L$ and the path consisting of $n$ west steps followed by $n$ south steps, and let $\nw(L) \cap \sigma$ be the subset of squares in $\nw(L)$ containing a 1 of the permutation matrix $A_\sigma$. Define a partial order on the set of lattice paths by $L \leq L'$ if $\nw(L) \subseteq \nw(L')$. We will prove by induction with respect to this partial order that
\begin{equation}
\label{eq_L_induction}
\dfrac{(1-t)^{|\nw(L) \cap \sigma|}}{(1-q)^{|\nw(L) \cap \sigma|}} \prod_{\square \in \nw(L)} \mc{P}(\square) = \omega(L) \prod_{\square \in \nw(L)} \ov{\mc{P}}(\square).
\end{equation}

The base case is the path consisting of $n$ west steps followed by $n$ south steps, for which both sides of \eqref{eq_L_induction} are equal to 1. For the induction step, it suffices to consider the case where $\nw(L')$ is obtained from $\nw(L)$ by adding a single square, as depicted below.

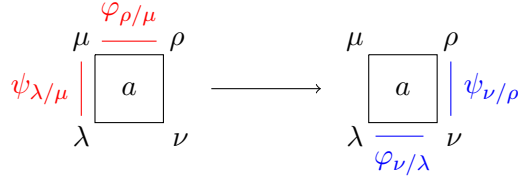
\begin{center}
\begin{tikzpicture}[scale=0.9]

\draw (0,0) -- (0,-1) -- (1,-1) -- (1,0) -- (0,0);
\draw[color=red] (0.1,0.2) -- (.9,0.2);
\draw[color=red] (-0.2,-.1) -- (-.2,-.9);
\node[color=red] at (.5,.6) {$\vp_{\rho/\mu}$}; 
\node[color=red] at (-.8,-.5) {$\psi_{\la/\mu}$};  
\node at (-0.2,0.2) {$\mu$};
\node at (1.2,0.2) {$\rho$};
\node at (-0.2,-1.2) {$\la$};
\node at (1.25,-1.25) {$\nu$};
\node at (.5,-.5) {$a$};

\begin{scope}[xshift=4cm]
\draw (0,0) -- (0,-1) -- (1,-1) -- (1,0) -- (0,0);
\draw[color=blue] (.1,-1.2) -- (.8,-1.2);
\draw[color=blue] (1.2,-.1) -- (1.2,-.9);
\node[color=blue] at (.5,-1.6) {$\vp_{\nu/\la}$}; 
\node[color=blue] at (1.8,-.5) {$\psi_{\nu/\rho}$};
\node at (-0.2,0.2) {$\mu$};
\node at (1.2,0.2) {$\rho$};
\node at (-0.2,-1.2) {$\la$};
\node at (1.25,-1.25) {$\nu$};
\node at (.5,-.5) {$a$};
\end{scope}

\draw[->] (1.7,-0.5) -- (3.3,-0.5);
\end{tikzpicture}
\end{center}

\noindent The path $L$ contains the red edges connecting $\la$ and $\rho$ with $\mu$, and the path $L'$ contains the blue edges connecting $\la$ and $\rho$ with $\nu$; otherwise the paths are the same. We assume that \eqref{eq_L_induction} holds for $L$, and show that it holds for $L'$ in each of the following four cases.

\bigskip

\noindent \textbf{Case 1:} $a = 0$, $\mu = \la = \rho = \nu$

In this case, replacing $L$ with $L'$ has no effect on \eqref{eq_L_induction}.

\bigskip

\noindent \textbf{Case 2:} $a = 0$, $\la \neq \rho$

In this case, $\square$ is of type II, so $\mc{P}(\square) = \ov{\mc{P}}(\square) = 1$, and the effect of replacing $L$ with $L'$ is to multiply the right-hand side of \eqref{eq_L_induction} by
\begin{equation}
\label{eq_ratio_psi_vp}
\dfrac{\psi_{\nu/\rho} \vp_{\nu/\la}}{\psi_{\la/\mu} \vp_{\rho/\mu}}.
\end{equation}
If $\mu = \la$, then $\nu = \rho$, and the expression \eqref{eq_ratio_psi_vp} is trivially equal to 1. Similarly, if $\mu = \rho$, then $\nu = \la$, and again \eqref{eq_ratio_psi_vp} is equal to 1. Finally, if $\mu \neq \la, \rho$, then \eqref{eq_ratio_psi_vp} is equal to 1 by the argument used in \S \ref{sec_qt_up_down} to prove the identity \eqref{eq_la_rho_qt}.

\bigskip

\noindent \textbf{Case 3:} $a = 0$, $\mu \lessdot \la = \rho \lessdot \nu$

In this case, replacing $L$ with $L'$ multiplies the left- and right-hand sides of \eqref{eq_L_induction} by
\[
 \mc{P}_\la(\mu \rightarrow \nu), \qquad \ov{\mc{P}}_\la(\mu \leftarrow \nu) \frac{\psi_{\nu/\la} \vp_{\nu/\la}}{\psi_{\la/\mu} \vp_{\la/\mu}},
\]
respectively. These two expressions are equal by the compatibility condition for the probabilistic bijection $\mc{P}_\la, \ov{\mc{P}}_\la$.

\bigskip

\noindent \textbf{Case 4:} $a = 1$, $\mu = \la = \rho \lessdot \nu$

In this case, replacing $L$ with $L'$ multiplies the left- and right-hand sides of \eqref{eq_L_induction} by
\[
\frac{1-t}{1-q} \mc{P}_\la(\la \rightarrow \nu), \qquad \ov{\mc{P}}_\la(\la \leftarrow \nu) \psi_{\nu/\la} \vp_{\nu/\la},
\]
respectively. As in the previous case, these two expressions are equal by the compatibility condition for the probabilistic bijection $\mc{P}_\la, \ov{\mc{P}}_\la$.
\end{proof}

Theorem \ref{thm_qrst_is_prob_bij} accomplishes our goal of proving the squarefree part of the Macdonald Cauchy identity by a probabilistic bijection. It also implies the more refined identities
\begin{align*}
&\dfrac{(1-t)^n}{(1-q)^n} \sum_{\sigma \in S_n} \mc{P}(\sigma \rightarrow P,Q) = \psi_P(q,t) \vp_Q(q,t), \\
&\sum_{P,Q} \ov{\mc{P}}(\sigma \leftarrow P,Q) \psi_P(q,t) \vp_Q(q,t) = \dfrac{(1-t)^n}{(1-q)^n}.
\end{align*}

The probabilities $\mc{P}(\sigma \rightarrow P,Q)$ and $\ov{\mc{P}}(\sigma \leftarrow P,Q)$ enjoy a symmetry property which generalizes that of Robinson--Schensted.

\begin{thm}
For any permutation $\sigma$ and pair of standard Young tableaux $P,Q$, we have
\[
\mc{P}(\sigma \rightarrow P,Q) = \mc{P}(\sigma^{-1} \rightarrow Q,P),
\]
and similarly for $\ov{\mc{P}}$.
\end{thm}

\begin{proof}
If $\Lambda$ is a growth from $\sigma$ to $P,Q$, then its transpose $\Lambda^t$ (i.e., the reflection of $\Lambda$ in the main diagonal) is a growth from $\sigma^{-1}$ to $Q,P$. It is clear that $\mc{P}(\square)$ is unchanged by swapping the partitions at the northeast and southwest corners of $\square$, so $\mc{P}(\Lambda) = \mc{P}(\Lambda^t)$, and the result follows.
\end{proof}

The discussion in \S \ref{sec: Fomin growth diagrams} (and especially Lemma \ref{lem_growth_insertion}) explains how to translate the $(q,t)$-local growth rules into a probabilistic insertion algorithm. Recall that for a semistandard Young tableau $T$, $T^{(k)}$ is the shape of the subtableau consisting of entries at most $k$.

\begin{defn}
Let $T$ be a partial standard Young tableau, and let $k$ be a number which is not an entry of $T$. The \emph{$(q,t)$-Robinson--Schensted (\qrst) insertion} of $k$ into $T$, denoted
\[
T \xleftarrow{\qrst} k,
\]
is the probability distribution computed as follows:
\begin{itemize}
\item For each $\nu \in \U(T^{(k)})$, place $k$ in the cell $\nu/T^{(k)}$ with probability $\mc{P}_{T^{(k)}}(T^{(k)} \rightarrow \nu)$.

\item Suppose an entry $z$ of $T$ is bumped by the placement of a smaller number. For each $\nu \in \mc{U}(T^{(z)})$, place $z$ in the cell $\nu/T^{(z)}$ with probability $\mc{P}_{T^{(z)}}(T^{(z-1)} \rightarrow \nu)$.
\end{itemize}
\end{defn}

\begin{ex}
\label{ex_qrst_insertion}
The insertion $\young(134,257) \xleftarrow{\qrst} 6$ produces
\[
\begin{array}{llll}
\young(1346,257) & \text{ with probability } p_{0,0}[(3,2)] &= \dfrac{(1-t)(1-qt^2)}{(1-qt)(1-q^3t^2)} \bigskip \\
\young(1357,256) & \text{ with probability } p_{0,1}[(3,2)] p_{1,0}[(3,3)] &= qt^2 \dfrac{(1-q)(1-t)}{(1-qt)(1-q^3t^2)} \bigskip \\
\young(134,256,7) & \text{ with probability } p_{0,1}[(3,2)] p_{1,1}[(3,3)] &= t \dfrac{(1-q)(1-t)}{(1-q^2t)(1-q^3t^2)} \bigskip \\
\young(134,257,6) & \text{ with probability } p_{0,2}[(3,2)] &= t^2 \dfrac{(1-q^2)(1-q^3t)}{(1-q^2t)(1-q^3t^2)}.
\end{array}
\]
We have computed these probabilities using the formulas of Proposition \ref{prop_explicit_formulas}.
\end{ex}

It follows from the preceding definitions and the discussion in \S \ref{sec: Fomin growth diagrams} that the probabilities $\mc{P}(\sigma \rightarrow P,Q)$ can be computed by recursively inserting $\sigma_1, \ldots, \sigma_n$ into the empty tableau according to $\qrst$, and summing the probabilities of all ``insertion paths'' that lead to the pair $P,Q$. For the permutations in the symmetric group $S_2$, these probabilities appeared in Figure \ref{fig_n=2} in the Introduction.

It is interesting to consider the insertion of the identity permutation. In this case, no bumping occurs, since at each step the number being inserted is larger than all entries currently in the tableau. This means that only the probabilities $\mc{P}_\la(\la \rightarrow \nu)$ come into play, and the recording tableau is equal to the insertion tableau. Using the definition of the probabilities $\mc{P}_\la(\la \rightarrow \nu)$ (Definition \ref{defn_r=0}), we obtain the following result.

\begin{lem}
\label{lem_id_probs}
For a pair $(P,Q)$ of standard tableaux of the same shape, we have
\begin{equation}
\label{eq_id_probs}
\mc{P}(\id \rightarrow P,Q) = \begin{cases}
\ds t^{n(\la)} \prod_{i = 1}^{|\la|} \alpha_{P^{(i)}/P^{(i-1)}}(q,t) & \text{ if } P = Q \text{ has shape } \la \\
0 & \text{ otherwise}.
\end{cases}
\end{equation}
\end{lem}

By Lemma \ref{lem_q=t=1}, we have $\lim_{q \rightarrow 1} \alpha_{\nu/\la}(q,q) = H_\la/H_\nu$, so in this specialization, \eqref{eq_id_probs} simplifies to
\begin{equation}
\label{eq_id_probs_q=t=1}
\mc{P}(\id \rightarrow P,Q)|_{q=t \rightarrow 1} = \begin{cases}
\dfrac{1}{H_\la} & \text{ if } P = Q \text{ has shape } \la \medskip \\
0 & \text{ otherwise}.
\end{cases}
\end{equation}
By the hook-length formula, this probability may also be expressed as $f_\la/|\la|!$. We conclude that in the $q = t \rightarrow 1$ specialization, the probability that the identity permutation inserts to a pair of SYTs of shape $\la$ is given by the Plancherel measure $(f_\la)^2/|\la|!$.

\section{Specializations}
\label{sec_specs}

\subsection{Column insertion: inverting $q$ and $t$}
\label{sec_column_insertion}

The Macdonald polynomials $P_\la$ and $Q_\la$ behave very simply under the substitution $(q,t) \mapsto (q^{-1},t^{-1})$:
\[
P_\la(\mb{x};q^{-1},t^{-1}) = P_\la(\mb{x};q,t), \quad\quad Q_\la(\mb{x};q^{-1},t^{-1}) = (qt^{-1})^{|\la|} Q_\la(\mb{x}; q,t).
\]
If we take the monomial expansions of Theorem \ref{thm_Mac_monomial} as the definition of $P_\la$ and $Q_\la$, this follows from the identities
\begin{equation}
\label{eq_psi_phi_inverse}
\psi_{\la/\mu}(q^{-1},t^{-1}) = \psi_{\la/\mu}(q,t), \quad\quad \vp_{\la/\mu}(q^{-1},t^{-1}) = (qt^{-1})^{|\la/\mu|} \vp_{\la/\mu}(q,t).
\end{equation}

It also follows from \eqref{eq_psi_phi_inverse} that the weights $\omega, \ov{\omega}$ of $\mu \in \mc{D}^*(\la)$ and $\nu \in \mc{U}(\la)$ are unchanged by the substitution $(q,t) \mapsto (q^{-1}, t^{-1})$. The expressions $\mc{P}_\la(\mu \rightarrow \nu)$ and $\ov{\mc{P}}_\la(\mu \leftarrow \nu)$, however, are not invariant under this substitution. For reasons that will soon become clear, we write $\mc{P}_\la^\col, \ov{\mc{P}}_\la^\col$ for the expressions obtained by substituting $(q,t) \mapsto (q^{-1}, t^{-1})$ in the definition of $\mc{P}_\la$ and $\ov{\mc{P}}_\la$. Similarly, using the notation introduced in \S\ref{sec_formulas}, we write
\[
p^\col_{r,s}(q,t) = p_{r,s}(q^{-1},t^{-1}), \quad\quad \ov{p}^\col_{r,s}(q,t) = \ov{p}_{r,s}(q^{-1},t^{-1}).
\]
(When we want to emphasize the dependence on $\la$, we will write $p^{\col}_{r,s}[\la], \ov{p}^{\col}_{r,s}[\la]$). By Remark \ref{rem_row_probs}, the expressions $p_{r,s}(q,t)$ and $\ov{p}_{r,s}(q,t)$ are honest probabilities whenever $q,t \in [0,1)$ or $q,t \in (1,\infty)$, so the same is true of $p_{r,s}^\col$ and $\ov{p}_{r,s}^\col$. In fact, these two sets of expressions differ by a monomial in $q$ and $t$.

\begin{lem}
\label{lem_col_probs}
For a partition $\la$ with parameters $(h_1, \ldots, h_d)$ and $(v_1, \ldots, v_d)$, we have
\begin{equation}
\label{eq_col_probs}
p_{r,s}^\col(q,t) = \begin{cases}
\ds q^{h_{s+1,d}} \alpha_s & \text{ if } r = 0 \bigskip \\
\ds \tau^\col_{r,s} \dfrac{\alpha_s \beta_r}{\gamma'_{r,s}} & \text{ if } r > 0,
\end{cases}
\quad\quad\quad
\ov{p}^\col_{r,s}(q,t) = \begin{cases}
\ds q^{h_{s+1,d}} \ov{\alpha}_s & \text{ if } r = 0 \bigskip \\
\ds \tau^\col_{r,s} \dfrac{\ov{\alpha}_s \ov{\beta}_r}{\gamma'_{r,s}} & \text{ if } r > 0,
\end{cases}
\end{equation}
where
\[
\tau^\col_{r,s} = \begin{cases}
q^{-1+h_{r,s}} \cdot t^{1 + 2v_{r+1,s}} & \text{ if } 0 < r \leq s \\
q^{h_{s+1,r-1}} & \text{ if } r > s.
\end{cases}
\]
\end{lem}

\begin{proof}
Applying the identity
\[
[i,j]_{q^{-1},t^{-1}} = -q^{-i}t^{-j}[i,j]_{q,t}
\]
to the formulas in Proposition \ref{prop_explicit_formulas}, we find that
\[
\begin{array}{lcl}
\alpha_s(q^{-1},t^{-1}) = q^{h_{s+1,d}} \cdot t^{v_{1,s}} \alpha_s(q,t) && \ov{\alpha}_s(q^{-1},t^{-1}) = q^{h_{s+1,d}} \cdot t^{v_{1,s}} \ov{\alpha}_s(q,t) \bigskip \\
\beta_r(q^{-1},t^{-1}) = q^{-h_{r,d}+1} \cdot t^{-v_{1,r}+1} \beta_r(q,t) && \ov{\beta}_r(q^{-1},t^{-1}) = q^{-h_{r,d}+1} \cdot t^{-v_{1,r}+1} \ov{\beta}_r(q,t)
\end{array}
\]
\[
\gamma'_{r,s}(q^{-1},t^{-1}) = \begin{cases}
q^{2-2h_{r,s}} \cdot t^{-2v_{r+1,s}} \gamma'_{r,s}(q,t) & \text{ if } 0 < r \leq s \\
q^{-2h_{s+1,r-1}} \cdot t^{2-2v_{s+1,r}} \gamma'_{r,s}(q,t) & \text{ if } r > s.
\end{cases}
\]
Now \eqref{eq_col_probs} follows from \eqref{eq_row_probs}.
\end{proof}

We observed at the end of \S\ref{sec_formulas} that $p_{r,s}(0,0) = \delta_{r,s}$, so the $q = t = 0$ specialization of $\qrst$ is the row insertion version of Robinson--Schensted. It follows from Lemma \ref{lem_col_probs} that
\[
\lim_{q \rightarrow \infty} p_{r,s}(q,q) = p^\col_{r,s}(0,0) = \delta_{r-1,s}
\]
where we interpret $\delta_{-1,s}$ as $\delta_{d,s}$. Thus, the $q = t \rightarrow \infty$ limit of $\qrst$ is the column insertion version of Robinson--Schensted.\\

Another useful feature of the column insertion probabilities is that interchanging $q$ and $t$ relates the row insertion probabilities associated to $\la$ and the column insertion probabilities associated to the conjugate partition $\la'$.

\begin{lem}
\label{lem_conjugate_probs}
Let $\la$ be a partition with $d$ distinct part sizes. For $r,s \in \{0, \ldots, d\}$, we have
\[
p_{r,s}[\la](t,q) = p_{d+1-r,d-s}^{\col}[\la'](q,t), \quad\quad\quad \ov{p}_{r,s}[\la](t,q) = \ov{p}_{d+1-r,d-s}^{\col}[\la'](q,t),
\]
where we interpret the first subscript modulo $d+1$ (e.g., $p^{\col}_{d+1,s}[\la'] = p^{\col}_{0,s}[\la']$).
\end{lem}

\begin{proof}
It is clear that
\[
(\la^{(+s)})' = (\la')^{(+(d-s))}, \quad\quad\quad (\la^{(-r)})' = (\la')^{(-(d+1-r))}
\]
for $r,s \in \{0, \ldots, d\}$ (with $(\la')^{(-(d+1))} = (\la')^{(-0)} = \la'$). Thus, by Lemma \ref{lem_conjugation}, we have
\[
\alpha_s[\la](t,q) = \alpha_{d-s}[\la'](q,t), \quad\quad \beta_r[\la](t,q) = \beta_{d+1-r}[\la'](q,t),
\]
and similarly for $\ov{\alpha}$ and $\ov{\beta}$.

If $\la$ has horizontal parameters $(h_1, \ldots, h_d)$ and vertical parameters $(v_1, \ldots, v_d)$, then $\la'$ has horizontal parameters $(v_d, \ldots, v_1)$ and vertical parameters $(h_d, \ldots, h_1)$. Using this, one verifies that
\[
(\gamma')_{r,s}[\la](t,q) = (\gamma')_{d+1-r,d-s}[\la'](q,t), \quad\quad \tau_{r,s}[\la](t,q) = \tau^{\col}_{d+1-r,d-s}[\la'](q,t),
\]
and the result follows from \eqref{eq_row_probs} and \eqref{eq_col_probs}.
\end{proof}

\subsection{$q$-Whittaker and Hall--Littlewood specializations}
\label{sec_qWhit_HL}

We now consider the $t = 0$ ($q$-Whittaker) and $q = 0$ (Hall--Littlewood) specializations of the probabilities $p_{r,s}(q,t)$ and $p^\col_{r,s}(q,t)$, and we describe the corresponding $q$- and $t$-deformations of Robinson--Schensted. To describe these specializations, it is convenient to set
\begin{equation}
\label{eq_infinite_boundary}
q^{k+h_0} = t^{k+v_{d+1}} = 0
\end{equation}
for any integer $k$. If we assume $q,t \in [0,1)$ (which is necessary if we want the specialized probabilities to lie in $[0,1]$), this can be interpreted as adding an infinite row beneath the first row of $\la$, and an infinite column to the left of the first column of $\la$.

Recall that $T^{(z)}$ is the shape of the subtableau of $T$ consisting of entries at most $z$, so $T^{(z)}_i$ is the number of entries in row $i$ of $T$ which are at most $z$.

\begin{lem}[$q$-Whittaker specializations] Suppose $\la$ has parameters $(h_1, \ldots, h_d)$ and $(v_1, \ldots, v_d)$.
\label{lem_qWhit_probs}
\begin{enumerate}
\item
The probabilities $p_{r,s}(q,0)$ are given by
\[
p_{r,s}(q,0) = \delta_{r,s}
\quad\quad\quad\quad \text{ if } r = 0 \text{ or } v_r > 1,
\]
\smallskip
\[
p_{r,s}(q,0) = \begin{cases}
\dfrac{q(1-q^{h_{r-1}})}{1-q^{1+h_{r-1}}} & \text{ if } s = r-1 \medskip \\
\dfrac{1-q}{1-q^{1+h_{r-1}}} & \text{ if } s = r \medskip \\
0 & \text{ otherwise}
\end{cases}
\quad\quad\quad\quad \text{ if } r > 0 \text{ and } v_r = 1.
\]

These probabilities give rise to a $q$-deformation of row insertion in which $T \leftarrow k$ is computed by the following rules:
\begin{itemize}
\item Insert $k$ into the first row.
\item If $z$ is bumped from row $i$, insert $z$ into
\[
\begin{cases}
\text{ row } i \text{ with probability } & \dfrac{q(1-q^{T^{(z)}_{i-1} - T^{(z)}_i})}{1-q^{1+T^{(z)}_{i-1} - T^{(z)}_i}} \bigskip \\
\text{ row } i+1 \text{ with probability } & \dfrac{1-q}{1-q^{1+T^{(z)}_{i-1} - T^{(z)}_i}}.
\end{cases}
\]
\end{itemize}

\item
The probabilities $p^\col_{r,s}(q,0)$ are given by
\[
p^\col_{0,s}(q,0) = q^{h_{s+1,d}} (1-q^{h_s}),
\]
\medskip
\[
p^\col_{r,s}(q,0) = \begin{cases}
q^{h_{s+1,r-1}}(1-q^{h_s}) & \text{ if } s < r \medskip \\
0 & \text{ if } s \geq r
\end{cases}
\quad\quad\quad\quad \text{ if } r > 0 \text{ and } v_r > 1,
\]
\medskip
\[
p^\col_{r,s}(q,0) = \begin{cases}
\dfrac{q^{h_{s+1,r-1}}(1-q)(1-q^{h_s})}{1-q^{1+h_{r-1}}} & \text{ if } s < r-1 \medskip \\
\dfrac{1-q^{h_{r-1}}}{1-q^{1+h_{r-1}}} & \text{ if } s = r-1 \medskip \\
0 & \text{ if } s \geq r
\end{cases}
\quad\quad\quad\quad \text{ if } r > 0 \text{ and } v_r = 1.
\]

These probabilities give rise to a $q$-deformation of column insertion in which $T \leftarrow k$ is computed by the following rules:
\begin{itemize}
\item Insert $k$ into row $j$ with probability
\[
q^{T^{(k)}_j}(1-q^{T^{(k)}_{j-1}-T^{(k)}_j}).
\]
\item If $z$ is bumped from row $i$, insert $z$ into row $j \leq i$ with probability
\[
\begin{cases}
\dfrac{1-q^{T^{(z)}_{i-1}-T^{(z)}_i}}{1-q^{1+T^{(z)}_{i-1}-T^{(z)}_i}} & \text{ if } j = i \bigskip \\
\dfrac{(1-q)q^{T^{(z)}_j - T^{(z)}_i}(1-q^{T_{j-1}^{(z)} - T_j^{(z)}})}{1-q^{1+T^{(z)}_{i-1}-T^{(z)}_i}} & \text { if } j < i.
\end{cases}
\]
\end{itemize}
\end{enumerate}
\end{lem}

\begin{rem}
\label{rem_q_RSK_etc}
The $q$-deformation of row insertion described in Lemma \ref{lem_qWhit_probs}(1) was introduced by Borodin and Petrov \cite{BorodinPetrov16}. The $q$-deformation of column insertion described in Lemma \ref{lem_qWhit_probs}(2) was introduced by O'Connell and Pei in an equivalent but somewhat different form \cite{OConnellPei13}, and reformulated by Pei in essentially the form we have given \cite{Pei14}. As discussed in Remark \ref{rem_RSK_etc}, both the row and column insertion versions of RS have well-known generalizations to bijections between nonnegative integer matrices and pairs of semistandard Young tableaux, which are known as the RSK and Burge correspondences, respectively. Matveev and Petrov have extended the $q$-deformations of row and column insertion to $q$-deformations of the RSK and Burge correspondences \cite{MatveevPetrov17}.\footnote{The description of the $q$-deformation of the Burge correspondence in \cite[\S 6.4]{MatveevPetrov17} seems to require the entries of the input matrix to be randomly sampled. However, by conditioning on the outcome that $X_1 + \cdots + X_j$ is equal to a fixed value (this sum is a $q$-geometric random variable by \cite[Rem. 6.7]{MatveevPetrov17}), one obtains a probabilistic insertion algorithm that can be applied to a fixed input matrix. We thank Leonid Petrov for explaining this to us.}
\end{rem}

\begin{proof}[Proof of Lemma \ref{lem_qWhit_probs}]
According to \eqref{eq_row_probs} and \eqref{eq_col_probs}, we have
\[
p_{r,s} = \begin{cases}
\ds t^{v_{1,s}} \alpha_s & \text{ if } r = 0 \bigskip \\
\ds \tau_{r,s} \dfrac{\alpha_s \beta_r}{\gamma'_{r,s}} & \text{ if } r > 0
\end{cases}
\quad\quad\quad
p^\col_{r,s} = \begin{cases}
\ds q^{h_{s+1,d}} \alpha_s & \text{ if } r = 0 \bigskip \\
\ds \tau^{\col}_{r,s} \dfrac{\alpha_s \beta_r}{\gamma'_{r,s}} & \text{ if } r > 0
\end{cases}
\]
where
\[
\tau_{r,s} = \begin{cases}
t^{v_{r+1,s}} & \text{ if } 0 < r \leq s \\
q^{1 + 2h_{s+1,r-1}} \cdot t^{-1+v_{s+1,r}} & \text{ if } r > s
\end{cases}
\quad\quad
\tau^\col_{r,s} = \begin{cases}
q^{-1+h_{r,s}} \cdot t^{1 + 2v_{r+1,s}} & \text{ if } 0 < r \leq s \\
q^{h_{s+1,r-1}} & \text{ if } r > s.
\end{cases}
\]
As noted in Remark \ref{rem_row_probs}(2), the formulas in Proposition \ref{prop_explicit_formulas} express $\alpha_s, \beta_r,$ and $\gamma'_{r,s}$ as products of terms of the form $1-q^it^j$ with $i,j \geq 0$ and $i+j > 0$. When $t = 0$, such a term becomes 1 if $j = 0$, and $1-q^i$ otherwise. Examining those formulas (and using \eqref{eq_infinite_boundary}), we find
\[
\alpha_s(q,0) = 1-q^{h_s},
\quad\quad\quad
\beta_r(q,0) = \begin{cases}
\dfrac{1}{1-q^{1+h_{r-1}}} & \text{ if } v_r = 1 \medskip \\
\dfrac{1}{1-q} & \text{ if } v_r > 1,
\end{cases}
\]
\[
\gamma'_{r,s}(q,0) = \begin{cases}
\dfrac{1-q^{h_r}}{1-q} & \text{ if } s = r \medskip \\
1 & \text{ if } s = r-1 \text{ and } v_r = 1 \medskip \\
\dfrac{1}{1-q} & \text{ otherwise}.
\end{cases}
\]
The formulas for $p_{r,s}(q,0)$ and $p^\col_{r,s}(q,0)$ follow by considering several cases.

The descriptions of the corresponding insertion procedures follow from the description of $\qrst$ in \S \ref{sec_qrst}, plus a straightforward translation from the parameter notation $(r,s,h_i,v_i)$ to the notation $T^{(z)}_i$.
\end{proof}

By interchanging $q$ and $t$, conjugating the partitions, and applying Lemma \ref{lem_conjugate_probs}, we obtain similar formulas for $p_{r,s}(0,t)$ and $p^\col_{r,s}(0,t)$. These probabilities give rise to $t$-deformations of row (resp., column) insertion, which admit similar descriptions to the $q$-deformations of column (resp., row) insertion in Lemma \ref{lem_qWhit_probs}. Let $T'^{(z)}_i$ be the number of entries in column $i$ of $T$ which are at most $z$.

\begin{lem}[Hall--Littlewood specializations] \
\label{lem_HL_probs}
\begin{enumerate}
\item
The probabilities $p_{r,s}(0,t)$ give rise to a $t$-deformation of row insertion in which $T \leftarrow k$ is computed by the following rules:
\begin{itemize}
\item Insert $k$ into column $j$ with probability
\[
t^{T'^{(k)}_j}(1-t^{T'^{(k)}_{j-1}-T'^{(k)}_j}).
\]
\item
If $z$ is bumped from column $i$, insert $z$ into column $j \leq i$ with probability
\[
\begin{cases}
\dfrac{1-t^{T'^{(z)}_{i-1}-T'^{(z)}_i}}{1-t^{1+T'^{(z)}_{i-1}-T'^{(z)}_i}} & \text{ if } j = i \bigskip \\
\dfrac{(1-t)t^{T'^{(z)}_j - T'^{(z)}_i}(1-t^{T_{j-1}'^{(z)} - T_j'^{(z)}})}{1-t^{1+T'^{(z)}_{i-1}-T'^{(z)}_i}} & \text { if } j < i.
\end{cases}
\]
\end{itemize}

\item
The probabilities $p^\col_{r,s}(0,t)$ give rise to a $t$-deformation of column insertion in which $T \leftarrow k$ is computed by the following rules:
\begin{itemize}
\item Insert $k$ into the first column.
\item If $z$ is bumped from column $i$, insert $z$ into
\[
\begin{cases}
\text{ column } i \text{ with probability } & \dfrac{t(1-t^{T'^{(z)}_{i-1} - T'^{(z)}_i})}{1-t^{1+T'^{(z)}_{i-1} - T'^{(z)}_i}} \bigskip \\
\text{ column } i+1 \text{ with probability } & \dfrac{1-t}{1-t^{1+T'^{(z)}_{i-1} - T'^{(z)}_i}}.
\end{cases}
\]
\end{itemize}
\end{enumerate}
\end{lem}

\begin{rem}
The $t$-deformation of column insertion described in Lemma \ref{lem_HL_probs}(2) was introduced by Bufetov and Petrov \cite{BufetovPetrov15}, and extended to a $t$-deformation of the Burge correspondence by Bufetov and Matveev \cite{BufetovMatveev18}.
\end{rem}

\subsection{The $q=t$ specialization}
\label{sec_q=t_spec}

The Schur polynomials are obtained from the Macdonald polynomials by setting $q = t$:
\[
P_\la(\mb{x};q,q) = Q_\la(\mb{x};q,q) = s_\la(\mb{x}).
\]
This can be seen immediately from the monomial expansions of Theorem \ref{thm_Mac_monomial}, since each $b_\la(c)$ is equal to 1 when $q = t$, and thus $\psi_{\la/\mu}(q,q) = \vp_{\la/\mu}(q,q) = 1$. We also see that the weights $\omega$ become 1 in this specialization, which means that the forward and backward probabilities $\mc{P}_\la$ and $\mc{\ov{P}}_\la$ become equal. Interestingly, we do not get a bijection by setting $q = t$ in $\mc{P}_\la(\mu \rightarrow \nu)$; instead, we get a one-parameter family of probabilistic bijections which interpolate between the row insertion bijection $F_\la^\row$ at $q = 0$ and the column insertion bijection $F_\la^\col$ at $q = \infty$. We considered the intermediate value $q = 1$ in \S \ref{sec_weighted_sets}-\ref{sec_probabilities}, and saw that in this specialization the probabilities are related to hook lengths and the enumeration of standard tableaux. We now consider the family $\mc{P}_\la(\mu \rightarrow \nu)|_{q=t}$ in general.

Let $[n]_q = 1-q^n$, and let
\[
H_\la(q) = \prod_{c \in \la} [h_\la(c)]_q.
\]
This is a $q$-analogue of the product of hook-lengths of $\la$.

\begin{lem}
\label{lem_q_q_probs}
For $\mu \in \mc{D}^*(\la)$ and $\nu \in \mc{U}(\la)$, we have
\begin{equation}
\label{eq_q_q_probs}
\mc{P}_\la(\mu \rightarrow \nu)|_{q=t} = \begin{cases}
q^{n(\nu/\la)} (1-q) \dfrac{H_\la(q)}{H_\nu(q)} & \text{ if } \mu = \la \medskip \\
q^{n(\nu/\la) - n(\la/\mu)-1} \dfrac{(H_\la(q))^2}{H_\mu(q) H_\nu(q)} \dfrac{(1-q)^2}{([h_\la(c_{\mu,\nu})]_q)^2} & \text{ if } \mu \in \mc{D}(\la)
\end{cases}
\end{equation}
(see \S \ref{sec_mu_neq_la} for the definition of the cell $c_{\mu,\nu}$).
\end{lem}

\begin{proof}
Arguing as in the proof of Lemma \ref{lem_q=t=1}, we obtain
\[
\alpha_{\nu/\la}(q,q) = (1-q) \dfrac{H_\la(q)}{H_\nu(q)}, \quad\quad \beta_{\la/\mu}(q,q) = \dfrac{1}{\alpha_{\la/\mu}(q,q)} = \dfrac{1}{1-q} \dfrac{H_\la(q)}{H_\mu(q)},
\]
where the $(1-q)$ is needed to cancel out the term in $H_\nu(q)$ (resp., $H_\la(q)$) coming from the cell $c = \nu/\la$ (resp., $c = \la/\mu$). The formula \eqref{eq_q_q_probs} in the case $\mu = \la$ follows immediately. For $\mu \in \mc{D}(\la)$, \eqref{eq_q_q_probs} follows from the proof of Lemma \ref{lem_gamma_q=t=1}.
\end{proof}

Combining Lemma \ref{lem_q_q_probs} with Lemma \ref{lem_id_probs}, we obtain a simple formula for the $q=t$ specialization of the $\qrst$ probabilities $\mc{P}(\id \rightarrow P,Q)$:
\begin{equation}
\label{eq_id_probs_q=t}
\mc{P}(\id \rightarrow P,Q)|_{q=t} = \begin{cases}
q^{n(\la)} \dfrac{(1-q)^{|\la|}}{H_\la(q)} & \text{ if } P = Q \text{ has shape } \la \medskip \\
0 & \text{ otherwise}.
\end{cases}
\end{equation}
This implies that in the $q = t$ specialization, the probability that the identity permutation inserts to a pair of SYTs of shape $\la$ is equal to
\[
q^{n(\la)} \dfrac{(1-q)^{|\la|} f_\la}{H_\la(q)},
\]
a $q$-analogue of the Plancherel measure $(f_\la)^2/|\la|!$.

\begin{rem}
\label{rem_ps_Schur}
The probabilities \eqref{eq_q_q_probs} may be reformulated in terms of principal specializations of Schur functions. The \emph{principal specialization} of the Schur function $s_\la$ is defined by $\ps(s_\la) = s_\la(1,q,q^2, \ldots)$. It is well-known (see, e.g., \cite[Cor. 7.21.3]{EC2}) that
\begin{equation}
\label{eq_ps_Schur}
\ps(s_\la) = \dfrac{q^{n(\la)}}{H_\la(q)}.
\end{equation}
Using \eqref{eq_ps_Schur}, the $q=t$ specialization of the identity $\sum_{\nu \in \mc{U}(\la)} \mc{P}(\la \rightarrow \nu) = 1$ can be rewritten as
\[
\sum_{\nu \in \mc{U}(\la)} \ps(s_\nu) = \dfrac{1}{1-q} \ps(s_\la),
\]
which is the principal specialization of the Pieri rule $h_1 s_\la = \sum_{\nu \in \mc{U}(\la)} s_\nu$ (c.f. Remark \ref{rem_q=t=1}). Similarly, the $q=t$ specializations of the identities $\sum_{\nu \in \mc{U}(\la)} \mc{P}(\mu \rightarrow \nu) = 1$ for $\mu \in \mc{D}(\la)$ and $\sum_{\mu \in \mc{D}^*(\la)} \mc{\bP}(\mu \leftarrow \nu) = 1$ for $\nu \in \mc{U}(\la)$ can be rewritten as
\[
\sum_{\nu \in \mc{U}(\la)} \dfrac{\ps(s_\mu) \ps(s_\nu)}{(1-q^{h_\la(c_{\mu,\nu})})^2} = \dfrac{q}{(1-q)^2} \ps(s_\la)^2
\]
and
\[
\dfrac{q}{1-q} \ps(s_\la) \ps(s_\nu) + \sum_{\mu \in \mc{D}(\la)} \dfrac{\ps(s_\mu) \ps(s_\nu)}{(1-q^{h_\la(c_{\mu,\nu})})^2} = \dfrac{q}{(1-q)^2} \ps(s_\la)^2.
\]
These can be viewed as $q$-analogues of the identities in Corollary \ref{cor_r>0_q=t=1}. We would like to better understand them.
\end{rem}

\section{Hook walks}
\label{sec_hook_walk}

In \S \ref{sec_proofs}, we proved that for each $\mu \in \mc{D}^*(\la)$, the rational expressions $\mc{P}_\la(\mu \rightarrow \nu)$ sum to 1, and therefore define a probability distribution on $\mc{U}(\la)$. In this section, we show that in the case $\mu = \la$, the expressions $\mc{P}_\la(\la \rightarrow \nu)$ arise from a simple random process on Young diagrams, thereby explaining why they are probabilities in a more conceptual way.

\subsection{The Greene--Nijenhuis--Wilf hook walk}

We saw in \S \ref{sec_mu=la} that $\mc{P}_\la(\la \rightarrow \nu)|_{q = t \rightarrow 1} = H_\la/H_\nu$, so the identity
\begin{equation}
\label{eq_upper_recursion_hooks}
\sum_{\nu \in \mc{U}(\la)} \dfrac{H_\la}{H_\nu} = 1
\end{equation}
is a special case of Theorem \ref{thm_r=0}. Greene, Nijenhuis, and Wilf gave a beautiful probabilistic proof of the identity \eqref{eq_upper_recursion_hooks} by means of a random ``hook walk'' \cite{GNW1,GNW2}, whose definition we now recall.

Let $\la$ be a partition. As in previous sections, we identify $\la$ with its Young diagram, which is the set of cells
\[
\la = \{(x,y) \,|\, 1 \leq y \leq \la'_1, 1 \leq x \leq \la_y\}
\]
in the first quadrant. Let $\ov{\la} = (\bbZ_{> 0} \times \bbZ_{> 0}) - \la$ be the complement of $\la$ in the first quadrant. Given a cell $c = (x,y) \in \ov{\la}$, define the \emph{arm} and the \emph{leg} of $c$ (with respect to $\la$) by
\[
\arm_\la(c) = \{(i,y) \,|\, \la_y < i < x\}, \quad\quad \leg_\la(c) = \{(x,j) \,|\, \la'_x < j < y\},
\]
where we consider $\la_j = 0$ for $j > \la'_1$ and $\la'_i = 0$ for $i > \la_1$ (see Figure \ref{fig_exterior_hook}). Set
\[
a_\la(c) = |\arm_\la(c)|, \quad\quad \ell_\la(c) = |\leg_\la(c)|, \quad\quad h_\la(c) = a_\la(c) + \ell_\la(c) + 1.
\]
We call $h_\la(c)$ the \emph{exterior hook-length} of the cell $c$. (When $\la$ is understood, we may omit the subscript $\la$ and write $\arm(c), a(c)$, etc.) It is clear that a cell $c \in \ov{\la}$ has exterior hook-length equal to 1 if and only if $c$ is an outer corner of $\la$. In this case, we will write $c = \nu/\la$, where $\nu$ is the partition $\la \cup \{c\}$.

\begin{figure}
\begin{center}
\begin{tikzpicture}[scale=0.8]
	\Yboxdim{16 pt}
	\Ylinecolour{lightgray}
	\Ycyan
	\tgyoung(0cm,0cm,,,,:::::;;;;;;;)
	\Yred
	\tgyoung(0cm,0cm,::::::::::::;,::::::::::::;,::::::::::::;)
	\Ywhite
	\tyng(0,0,10,9,7,5,5,3,2)
	\draw[line width=2pt] (0,0) -- (10*16pt,0) -- (10*16pt,16pt) -- (9*16pt,16pt) -- (9*16pt,2*16pt) -- (7*16pt,2*16pt) -- (7*16pt,3*16pt) -- (5*16pt,3*16pt) -- (5*16pt,5*16pt) -- (3*16pt,5*16pt) -- (3*16pt,6*16pt) -- (2*16pt,6*16pt) -- (2*16pt,7*16pt) -- (0*16pt,7*16pt) -- (0,0) --(7*16pt,0);
	\Ylinecolour{black}
	\tgyoung(0cm,0cm,,,,::::::::::::;)
	\node at (12.5*16pt,3.5*16pt) {$c$};
	\node at (9.5*16pt,4.5*16pt) {$\arm_\lambda(c)$};
	\node at (14.5*16pt,1.75*16pt) {$\leg_\lambda(c)$};

\end{tikzpicture}
\end{center}
\caption{The arm and leg of a cell $c$ in the complement of the partition $\la=(10,9,7,5,5,3,2)$.}
\label{fig_exterior_hook}
\end{figure}
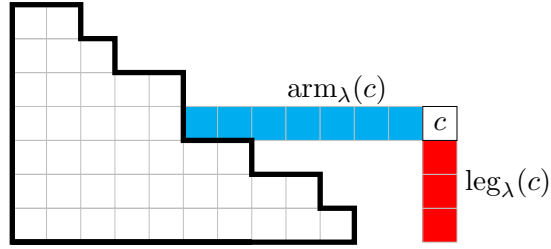

\begin{defn}[Greene--Nijenhuis--Wilf \cite{GNW2}]
\label{defn_GNW}
Fix a partition $\la$. For $c \in \ov{\la}$, the \emph{exterior hook walk} starting at $c$ is the following random process:

\emph{If $c$ is an outer corner of $\la$, the process terminates. If $c$ is not an outer corner of $\la$, choose a cell $c'$ in $\arm_\la(c) \, \cup \, \leg_\la(c)$ uniformly at random, and move to that cell. Repeat the process for $c'$.}

For $\nu \in \mc{U}(\la)$, let $\condP{\nu}{c}$ be the probability that the exterior hook walk starting at $c$ ends at the outer corner $\nu/\la$.
\end{defn}

\begin{thm}[Greene--Nijenhuis--Wilf \cite{GNW2}]
\label{thm_GNW}
Let $\la$ be a partition, and fix $c = (x,y) \in \ov{\la}$. If $x > \la_1$ and $y > \la'_1$, then
\[
\condP{\nu}{c} = \dfrac{H_\la}{H_\nu}.
\]
for $\nu \in \mc{U}(\la)$. In particular, $\condP{\nu}{c}$ is independent of $c$, as long as neither the row nor column of $c$ intersects the diagram of $\la$.
\end{thm}

\begin{ex}
If $\la = (h^v)$ is a rectangle, then the exterior hook walk starting at $c = (h+1,v+1)$ ends at $\nu_0/\la = (h+1,1)$ with probability $v/(v+h)$, and at $\nu_1/\la = (1,v+1)$ with probability $h/(v+h)$. The reader may easily verify that Theorem \ref{thm_GNW} holds in this case.
\end{ex}

\begin{rem}
The original hook walk introduced by Greene, Nijenhuis, and Wilf takes place inside the Young diagram of $\la$, and gives a probabilistic proof of the hook-length formula \cite{GNW1}.
\end{rem}

\subsection{A $(q,t)$-hook walk}
\label{sec_hook_walk_qt}

In this section, we present a $(q,t)$-generalization of the exterior hook walk which gives rise to the probabilities $\mc{P}_\la(\la \rightarrow \nu)$. Our definition was inspired by the $(q,t)$-hook walk of Garsia and Haiman \cite{GarHai}, which generalizes the ``interior'' hook walk of \cite{GNW1}. (We note also that the $q = t$ specialization of our hook walk is similar to Kerov's $q$-hook walk \cite{Kerov}.)

\begin{defn}
Let $\la$ be a partition, and fix $c = (x,y) \in \ov{\la}$. For $c' \in \arm_\la(c) \, \cup \, \leg_\la(c)$, define
\[
P(c \rightarrow c') = \begin{cases}
q^{a(c)-i} \dfrac{t^{\ell(c)} (1-q)}{1-q^{a(c)}t^{\ell(c)}} & \text{ if } c' = (x-i,y) \in \arm_\la(c) \bigskip \\
t^{j-1} \dfrac{1-t}{1-q^{a(c)}t^{\ell(c)}} & \text{ if } c' = (x,y-j) \in \leg_\la(c).
\end{cases}
\]
The \emph{exterior $(q,t)$-hook walk} starting at $c$ is the random process which terminates if $c$ is an outer corner of $\la$, and otherwise moves from $c$ to $c' \in \arm_\la(c) \, \cup \, \leg_\la(c)$ with probability $P(c \rightarrow c')$, and then repeats.

For $\nu \in \mc{U}(\la)$, write $\qtP{\nu}{c}$ for the probability that the exterior $(q,t)$-hook walk starting at $c$ terminates at the outer corner $\nu / \la$.
\end{defn}

It is easy to verify that
\[
\sum_{c' \in \arm_\la(c) \, \cup \, \leg_\la(c)} P(c \rightarrow c') = 1,
\]
so the exterior $(q,t)$-hook walk does in fact define a probability distribution on $\mc{U}(\la)$. It is also clear that if $q,t \in [0,1)$ or $q,t \in (1,\infty)$, then $\qtP{\nu}{c} \in [0,1]$. In the limit $q = t \rightarrow 1$, $P(c \rightarrow c')$ becomes the uniform distribution on $\arm_\la(c) \, \cup \, \leg_\la(c)$, so the $(q,t)$-hook walk reduces to the Greene--Nijenhuis--Wilf hook walk of Definition \ref{defn_GNW}.

\begin{thm}
\label{thm_qt_hook}
Let $\la$ be a partition, and fix $c = (x,y) \in \ov{\la}$. If $x > \la_1$ and $y > \la'_1$, then
\[
\qtP{\nu}{c} = \mc{P}_\la(\la \rightarrow \nu)
\]
for $\nu \in \mc{U}(\la)$. In particular, $\qtP{\nu}{c}$ is independent of $c$, as long as neither the row nor column of $c$ intersects the diagram of $\la$.
\end{thm}

This result proves that $\mc{P}_\la(\la \rightarrow \nu)$ are probabilities, thereby giving an alternative proof of Theorem \ref{thm_r=0}.

Our proof of Theorem \ref{thm_qt_hook} uses the strategy of the proof of Theorem \ref{thm_GNW} in \cite{GNW2}. It is based on two straightforward lemmas. We encourage the reader to refer to Figure \ref{fig_qt_hook_thm} while reading the statements of these lemmas and their proofs.

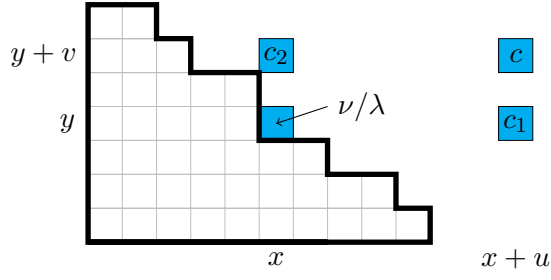
\begin{figure}
\begin{center}
\begin{tikzpicture}[scale=0.8]
	\Yboxdim{16 pt}
	\Ylinecolour{lightgray}
	\tyng(0,0,10,9,7,5,5,3,2)
	\Ylinecolour{black}
	\Ycyan
	\tgyoung(0cm,0cm,,,,::::::::::::;)
	\tgyoung(0cm,0cm,,,,,,::::::::::::;)
	\tgyoung(0cm,0cm,,,,,,:::::;)
	\tgyoung(0cm,0cm,,,,:::::;)
	\draw[line width=2pt] (0,0) -- (10*16pt,0) -- (10*16pt,16pt) -- (9*16pt,16pt) -- (9*16pt,2*16pt) -- (7*16pt,2*16pt) -- (7*16pt,3*16pt) -- (5*16pt,3*16pt) -- (5*16pt,5*16pt) -- (3*16pt,5*16pt) -- (3*16pt,6*16pt) -- (2*16pt,6*16pt) -- (2*16pt,7*16pt) -- (0*16pt,7*16pt) -- (0,0) --(7*16pt,0);
	\node at (12.5*16pt,3.5*16pt) {$c_1$};
	\node at (12.5*16pt,5.5*16pt) {$c$};
	\node at (5.5*16pt,5.5*16pt) {$c_2$};
	\draw[->] (7*16pt,4*16pt) -- (5.5*16pt,3.5*16pt);
	\node at (8*16pt,4*16pt) {$\nu / \lambda$};
	\node at (-.6*16pt,3.5*16pt) {$y$};
	\node at (-1.3*16pt,5.5*16pt) {$y+v$};
	\node at (5.5*16pt,-.5*16pt) {$x$};
	\node at (12.5*16pt,-.5*16pt) {$x+u$};	
\end{tikzpicture}
\end{center}
\caption{The cells involved in Lemma \ref{lem_qt_hook_1}.}
\label{fig_qt_hook_thm}
\end{figure}

\begin{lem}
\label{lem_qt_hook_1}
Suppose $\nu \in \mc{U}(\la)$, and let $(x,y) = \nu/\la$. If $c = (x+u,y+v)$, then
\[
\qtP{\nu}{c} = \qtP{\nu}{c_1} \qtP{\nu}{c_2},
\]
where $c_1 = (x+u,y)$ and $c_2 = (x,y+v)$.
\end{lem}

\begin{proof}
Throughout this proof, we write $a(c) = a_\la(c)$ and $\ell(c) = \ell_\la(c)$.

We use induction on $u+v$, with the base cases $u = 0$ or $v = 0$ being immediate, since $\qtP{\nu}{\nu} = 1$. We compute
\begin{align*}
\qtP{\nu}{c} &= \sum_{c' \in \arm_\la(c)} P(c \rightarrow c') \qtP{\nu}{c'} + \sum_{c'' \in \leg_\la(c)} P(c \rightarrow c'') \qtP{\nu}{c''} \\
&= \dfrac{t^{\ell(c)}(1-q)}{1-q^{a(c)}t^{\ell(c)}} \qtP{\nu}{c_2} \sum_{i = 1}^u q^{a(c)-i} \qtP{\nu}{(x+u-i,y)} \\
& \quad\quad\quad\quad\quad\quad + \dfrac{1-t}{1-q^{a(c)}t^{\ell(c)}} \qtP{\nu}{c_1} \sum_{j=1}^v t^{j-1} \qtP{\nu}{(x,y+v-j)},
\end{align*}
where the first equality comes from the definition of the $(q,t)$-hook walk, and the second comes from the inductive hypothesis, together with the fact that $\qtP{\nu}{(x',y+v)} = 0$ if $x' < x$, and $\qtP{\nu}{(x+u,y')} = 0$ if $y' < y$.

We also have, by definition,
\begin{align*}
\qtP{\nu}{c_1} &= \dfrac{t^{\ell(c_1)} (1-q)}{1-q^{a(c_1)}t^{\ell(c_1)}}\sum_{i = 1}^u q^{a(c_1)-i} \qtP{\nu}{(x+u-i,y)} \\
&= q^{u-a(c)} \dfrac{t^{\ell(c)-v} (1-q)}{1-q^u t^{\ell(c)-v}} \sum_{i=1}^u q^{a(c)-i} \qtP{\nu}{(x+u-i,y)},
\end{align*}
and similarly
\[
\qtP{\nu}{c_2} = \dfrac{1-t}{1-q^{a(c)-u} t^v} \sum_{j=1}^v t^{j-1} \qtP{\nu}{(x,y+v-j)}.
\]
By isolating the sums in these two expressions and substituting into the expression for $\qtP{\nu}{c}$, we obtain
\begin{align*}
\qtP{\nu}{c} &= \qtP{\nu}{c_1}\qtP{\nu}{c_2}\left( q^{a(c)-u} t^v \dfrac{1-q^ut^{\ell(c)-v}}{1-q^{a(c)}t^{\ell(c)}} + \dfrac{1-q^{a(c)-u} t^v}{1-q^{a(c)}t^{\ell(c)}} \right) \\
&= \qtP{\nu}{c_1}\qtP{\nu}{c_2},
\end{align*}
completing the induction.
\end{proof}

\begin{lem} Suppose $\nu \in \mc{U}(\la)$, and let $(x,y) = \nu/\la$.
\label{lem_qt_hook_2}
\begin{enumerate}
\item If $c = (x+u,y)$, then
\[
\qtP{\nu}{c} = t^{\ell_\la(c)} \dfrac{1-q}{1-q^{a_\nu(c)+1}t^{\ell_\nu(c)}} \prod_{c' \in \arm_\la(c) - \nu/\la} \dfrac{1-q^{a_\la(c')+1}t^{\ell_\la(c')}}{1-q^{a_\nu(c')+1}t^{\ell_\nu(c')}}.
\]
\item If $c = (x,y+v)$, then
\[
\qtP{\nu}{c} = \dfrac{1-t}{1-q^{a_\nu(c)}t^{\ell_\nu(c)+1}} \prod_{c'' \in \leg_\la(c) - \nu/\la} \dfrac{1-q^{a_\la(c'')}t^{\ell_\la(c'')+1}}{1-q^{a_\nu(c'')}t^{\ell_\nu(c'')+1}}.
\]
\end{enumerate}
\end{lem}

\begin{proof}
We prove part (1); part (2) is proved in the same way. Let $c_i = (x+u-i,y)$, so that $c_0 = (x+u,y)$, $c_u = \nu/\la$, and
\[
\arm_\la(c_0) = \{c_1, \ldots, c_u\},
\]
Note that $a_\la(c_i) = u-i$. By definition, we have
\begin{align*}
\qtP{\nu}{c_0} &= t^{\ell_\la(c_0)} \dfrac{1-q}{1-q^{a_\la(c_0)}t^{\ell_\la(c_0)}} \sum_{i=1}^u q^{u-i} \qtP{\nu}{c_i}, \\
\qtP{\nu}{c_1} &= t^{\ell_\la(c_1)} \dfrac{1-q}{1-q^{a_\la(c_1)}t^{\ell_\la(c_1)}} \sum_{i=2}^u q^{u-i} \qtP{\nu}{c_i}.
\end{align*}
Isolating the sums and subtracting the second equation from the first, we obtain
\[
\dfrac{1-q^{a_\la(c_0)}t^{\ell_\la(c_0)}}{t^{\ell_\la(c_0)}(1-q)} \qtP{\nu}{c_0} - \dfrac{1-q^{a_\la(c_1)}t^{\ell_\la(c_1)}}{t^{\ell_\la(c_1)}(1-q)} \qtP{\nu}{c_1} = q^{a_\la(c_1)} \qtP{\nu}{c_1},
\]
and thus
\[
\qtP{\nu}{c_0} = t^{\ell_\la(c_0)-\ell_\la(c_1)} \dfrac{1-q^{a_\la(c_1)+1}t^{\ell_\la(c_1)}}{1-q^{a_\la(c_0)}t^{\ell_\la(c_0)}} \qtP{\nu}{c_1}.
\]
Iterating this argument $u$ times, we find
\[
\qtP{\nu}{c_0} = t^{\ell_\la(c_0)} \prod_{i=0}^{u-1} \dfrac{1-q^{a_\la(c_{i+1})+1}t^{\ell_\la(c_{i+1})}}{1-q^{a_\la(c_i)}t^{\ell_\la(c_i)}}.
\]
The expression for $\qtP{\nu}{c_0}$ given in the statement of the lemma is obtained by pulling out the denominator of the $i = 0$ term and the numerator of the $i = u-1$ term, and using the fact that $a_\nu(c_i) = a_\la(c_i) - 1$ and $\ell_\nu(c_i) = \ell_\la(c_i)$.
\end{proof}

\begin{proof}[Proof of Theorem \ref{thm_qt_hook}]
Fix $\nu \in \mc{U}(\la)$, and let $(x,y) = \nu/\la$. Also fix $c = (x+u,y+v)$, with $x+u > \la_1$ and $y+v > \la'_1$. Set $c_1 = (x+u,y)$ and $c_2 = (x,y+v)$. By Lemmas \ref{lem_qt_hook_1} and \ref{lem_qt_hook_2}, we have
\[
\qtP{\nu}{c} = t^{\ell_\la(c_1)} \Pi_1 \Pi_2,
\]
where
\[
\Pi_1 = \dfrac{1-q}{1-q^{a_\nu(c_1)+1}t^{\ell_\nu(c_1)}} \prod_{c' \in \arm_\la(c_1) - \nu/\la} \dfrac{1-q^{a_\la(c')+1}t^{\ell_\la(c')}}{1-q^{a_\nu(c')+1}t^{\ell_\nu(c')}},
\]
\[
\Pi_2 = \dfrac{1-t}{1-q^{a_\nu(c_2)}t^{\ell_\nu(c_2)+1}} \prod_{c'' \in \leg_\la(c_2) - \nu/\la} \dfrac{1-q^{a_\la(c'')}t^{\ell_\la(c'')+1}}{1-q^{a_\nu(c'')}t^{\ell_\nu(c'')+1}}.
\]
We must show that $t^{\ell_\la(c_1)} \Pi_1 \Pi_2$ is equal to
\[
\mc{P}_\la(\la \rightarrow \nu) = t^{n(\nu/\la)} \alpha_{\nu/\la}(q,t).
\]
Since $x+u > \la_1$, we have $\ell_\la(c_1) = y-1 = n(\nu/\la)$. To show that $\Pi_1 \Pi_2 = \alpha_{\nu/\la}(q,t)$, we use an argument similar to the proof of Proposition \ref{prop_explicit_formulas}.

Let $(h_1, \ldots, h_d)$ and $(v_1, \ldots, v_d)$ be the parameters of $\la$, and suppose $\nu = \la^{(+s)}$, so that $(x,y) = (h_{s+1,d}+1,v_{1,s}+1)$. For $i = 1, \ldots, s-1$, consider the contribution to $\Pi_1$ of the cells in $\arm_\la(c_1)$ which lie in columns $h_{i+1,d}+1, \ldots, h_{i,d}$. Each of these cells has leg-length $v_{i+1,s}$ with respect to both $\la$ and $\nu$, so the product over these cells telescopes, and their net contribution to $\Pi_1$ is
\[
\dfrac{[h_{i,s},v_{i+1,s}]}{[h_{i+1,s},v_{i+1,s}]}.
\]
Similarly, the net contribution to $\Pi_1$ of the cells in $\arm_\la(c_1)$ which lie in columns $h_{s+1,d} + 2, \ldots, h_{s,d}$ is
\[
\dfrac{[h_s,0]}{1-q},
\]
and the net contribution of the cells in $\arm_\la(c_1)$ which lie in columns $h_{1,d} + 1, \ldots, x+u-1$ is
\[
\dfrac{[u,v_{1,s}]}{[h_{1,s},v_{1,s}]} = \dfrac{1-q^{a_\nu(c_1)+1}t^{\ell_\nu(c_1)}}{[h_{1,s},v_{1,s}]}
\]
(here we again use the fact that $x+u > h_{1,d}$). Thus, we have
\[
\Pi_1 = \dfrac{[h_s,0]}{[h_{1,s},v_{1,s}]} \prod_{i=1}^{s-1} \dfrac{[h_{i,s},v_{i+1,s}]}{[h_{i+1,s},v_{i+1,s}]} = \prod_{i=1}^s \dfrac{[h_{i,s},v_{i+1,s}]}{[h_{i,s},v_{i,s}]}.
\]

A similar argument (using the fact that $y+v > \la_1'$) shows that
\[
\Pi_2 = \prod_{i=s+1}^d \dfrac{[h_{s+1,i-1},v_{s+1,i}]}{[h_{s+1,i},v_{s+1,i}]}.
\]
By Proposition \ref{prop_explicit_formulas}, we conclude that $\Pi_1 \Pi_2 = \alpha_s(q,t) = \alpha_{\nu/\la}(q,t)$.
\end{proof}

\begin{rem}
We would very much like to have a similar interpretation of the probabilities $\mc{P}_\la(\mu \rightarrow \nu)$ for $\mu \in \mc{D}(\la)$, and of the backward probabilities $\ov{\mc{P}}_\la(\mu \leftarrow \nu)$. That is, we would like to define two random processes (perhaps similar to the $(q,t)$-hook walk) whose ``end'' probabilities are $\mc{P}_\la(\mu \rightarrow \nu)$ for fixed $\mu \in \mc{D}(\la)$, and $\ov{\mc{P}}_\la(\mu \leftarrow \nu)$ for fixed $\nu \in \mc{U}(\la)$.
\end{rem}

\bibliographystyle{abbrv}
\bibliography{qrst_revised}

\bigskip

{\footnotesize
  \textsc{LaCIM, Universit\'e du Qu\'ebec \`a Montr\'eal, Montr\'eal, QC, Canada} \par  
  \textit{Email addresses}: \texttt{florian.aigner@univie.ac.at}, \texttt{gabriel.frieden@lacim.ca} \par
}

\end{document}